\documentclass[10pt]{amsart}

\usepackage{amsmath}

\usepackage{amssymb}

\usepackage{graphicx}

\newtheorem{thm}{Theorem}[section]
\newtheorem{prop}{Proposition}
\newtheorem{lemma}[thm]{Lemma}
\newtheorem{cor}[thm]{Corollary}

\theoremstyle{definition}
\newtheorem{defn}[thm]{Definition}

\theoremstyle{remark}
\newtheorem{remark}[thm]{Remark}

\numberwithin{equation}{section}
\newcommand{\mathsym}[1]{{}}
\newcommand{\unicode}[1]{{}}

\def\C{\mathbb{C}}
\def\R{\mathbb{R}}
\def\CP{\mathbb{CP}}
\def\Z{\mathbb{Z}}

\def\G{\mathfrak{G}}

\def\PP2{\mathcal{P}^2}
\def\G2{\mathrm{G}^+_2}
\def\G1{\mathrm{G}^+_1}
\def\NR1{\mathrm{N}^+_1}

\begin{document}
\title[]{The Cauchy-Riemann strain functional for Legendrian curves in the $3$-sphere}

\author{Emilio Musso}
\address{(E. Musso) Dipartimento di Matematica, Politecnico di Torino,
Corso Duca degli Abruzzi 24, I-10129 Torino, Italy}
\email{emilio.musso@polito.it}

\author{Filippo Salis}
\address{(F. Salis) Dipartimento di Matematica, Politecnico di Torino,
Corso Duca degli Abruzzi 24, I-10129 Torino, Italy}
\email{filippo.salis@polito.it}

\thanks{Authors partially supported by PRIN 2017 "Real and 
Complex Manifolds: Topology, Geometry and holomorphic 
dynamics"  project 2017JZ2SW5; by the GNSAGA, Italy of 
INDAM. The present research was also partially supported 
by MIUR, Italy grant ‘‘Dipartimenti di Eccellenza’’ 
2018–2022, CUP: E11G18000350001, DISMA, Politecnico di 
Torino, Italy.}


\date{Version of \today}


\begin{abstract}
The lower-order  cr-invariant variational problem for Legendrian curves in the $3$-sphere is studied  and its Euler-Lagrange equations are deduced. Closed critical curves are investigated. Closed critical curves with non-constant cr-curvature are characterized. We prove that their cr-equivalence classes are in one-to-one correspondence with the rational points of a connected planar domain. A procedure to explicitly build all such curves is described. In addition, a geometrical interpretation of the rational parameters in terms of three phenomenological invariants is given.
\end{abstract}
\maketitle

\section{Introduction}\label{0} The present paper is a first step toward a more ambitious research plan, aimed at linking the topology of Legendrian knots in a contact $3$-manifold to their differential invariants with respect to a compatible Cauchy-Riemann structure \cite{CH}. The invariants can be build from the Chern's structure bundle and its Cartan connection \cite{Ca1,CM} via the moving frames method \cite{CH1,GG,JMN}. Equivalently, one can resort to the Fefferman conformal structure \cite{BDS,Fe,L} and to its normal conformal connection \cite{Ca3,EMN,Ko}. In cr-geometry, most of the attention has been focused on a family of curves transversal to the contact distribution, know as {\it chains} \cite{BDS,Ca1,CM,Fe,Koch}. Chains arise as projections of null geodesics of the Fefferman conformal structure. Inspired by the strong interrelationships between cr and Lorentzian conformal geometry and by some earlier works on conformal geometry of curves \cite{DMN,MMR,M0,MN4}, we analyze global properties of Legendrian curves in the 3-sphere equipped with its standard cr-structure. In addition to the aforementioned interrelationships with Lorentzian conformal geometry, the fact that the cr-transformation group of ${\rm S}^3$ is a real form of ${\rm PSL}(3,\C)$, explains the many formal similarities with classical projective differential geometry of plane curves \cite{Ca2,Ha,MN1,M2,OT}. For instance, one can associate to a Legendrian curve $\gamma$ of ${\rm S}^3$ a cubic form ${\mathfrak a}=adt^3$ and a projective structure on the curve. They originate a higher-order differential invariant, the {\it cr-stress tensor}. If $\gamma$ is {\it generic}, i.e. if its cubic form is everywhere different from zero, then one can find parameterizations such that $a=1$. Hence, a generic Legendrian curve comes equipped with an intrinsic orientation and its shape is detected by a single differential invariant, the {\it cr-curvature} $\kappa$. Integrating the linear differential form $\mathfrak{s}=\sqrt[3]{|a|}dt$ one gets an analogue of the {\it projective length} of a plane curve.  Since $\mathfrak{s}$ is dimensionless, it is called the {\it infinitesimal strain} and the integral $\mathfrak{S}_{\gamma}$ is said the {\it total strain} of $\gamma$. The  {\it Arnold-Liouville} and the {\it collective complete}\footnote{{\it non-commutative integrability}, in the terminology of \cite{FT,JO}} {\it integrability}  \cite{FT,GS,JO} of the Hamiltonian  contact system governing the geometry of generic critical curves was studied in \cite{M1}. Accordingly, generic critical curves can be found by quadratures and explicit parametrizations can be given in terms of elliptic functions and integrals. In this paper we address the question of existence and global properties of {\it closed} critical curves. 

If $\gamma, \widetilde{\gamma}  : \R \to {\rm S}^3$ are two curves and $|[\gamma]|, |[\widetilde{\gamma}]|$ denote their trajectories, then $\gamma$ and $\widetilde{\gamma}$  are said equivalent if there is an element $[{\bf A}]$ of the cr-transformation group $\widehat{\rm{G}}$ of ${\rm S}^3$, such that $[{\bf A}]\cdot |[\gamma]| = |[\widetilde{\gamma}]|$. By a symmetry of $\gamma$ is meant an element $[{\bf A}] \in \widehat{{\rm G}}$, such that $[{\bf A}] \cdot |[\gamma]| = |[\gamma]|$. The set of all symmetries of $\gamma$ is a subgroup $\widehat{{\rm G}}_{\gamma}$ of $\widehat{{\rm G}}$. The {\it  symmetry group} of a generic closed curve with non-constant cr-curvature is finite and its cardinality is called the {\it wave number}. From the viewpoint of the cr-geometry, the most elementary Legendrian knots are the {\it cycles}, characterized by having null cubic form and generic Legendrian knots with constant cr-curvature. A cycle is equivalent to the trivial Legendrian knot $t\in \R\to (\cos(t),-i\sin(t))\in {\rm S}^3\subset \C^2$. The symmetry group of a cycle is isomorphic to $\rm{SL}(2,\R)$, its {\it Maslov index}\footnote{or {\it turning number}} \cite{GE}  is zero and its {\it Thurston-Bennequin} invariant \cite{GE} is $-1$. Closed generic Legendrian curves with constant cr-curvature are orbits of one-parameter subgroups and their symmetry groups are isomorphic to ${\rm S}^1$ . The equivalence classes of closed generic Legendrian curves with constant curvature are in one-to-one correspondence with pairs $(m,n)$ of relatively prime postive integers such that $m>n$ (see Theorem \ref{S3.6.P1}). A generic Legendrian curve with constant cr-curvature and characteristic numbers $m,n$ is a torus knot of type $(-m,n)$ with Maslov index equal to $m-n$ and Bennequin-Thurston invariant equal to $-mn$ . Then, in view of the classification \cite{EH} of Legendrian torus knots, each isotopy class of a negative Legendrian torus knot with maximal Maslov index and maximal Thurston-Bennequin invariant is represented by a Legendrian curve with constant cr-curvture. The maximal tori of $\widehat{{\rm G}}$ (ie, maximal compact abelian subgroups) are $2$-dimensional and conjugates each other. The action of a maximal torus ${\rm T}^2\subset \widehat{{\rm G}}$ on ${\rm S}^3$ has two special orbits, the {\it axes of symmetry}. These orbits are chains and have a natural positive orientation. Now we state the three main results.
\medskip

\noindent {\bf Theorem A}.{\it \hspace{1pt} A Legendrian curve $\gamma$ is critical for the total strain functional if and only if its stress tensor vanishes. A critical curves is either a cycle or else is generic.}
\medskip

\noindent Generic Legendrian curves with constant curvature are critical points of the strain functional. For brevity, generic critical curves with non-constant periodic curvature are called {\it strings}.
\medskip

\noindent {\bf Theorem B}.{\it \hspace{1pt} The equivalence classes of closed strings are in one to one correspondence with the rational points of the domain 
$${\mathcal M}=\{(x,y)\in \R^2 : x^2+xy+y^2<1/4,\hspace{2pt} x-y>0,\hspace{2pt} x+y>1/2 \}.$$}

\noindent The rational points of ${\mathcal M}$ are called the {\it moduli of closed strings}. If $\gamma$ is a cr-string with modulus $(q_2,q_3)$, the positive integers ${\mathtt h}_1,{\mathtt k}_1,{\mathtt h}_2$ and ${\mathtt k}_2$ such that ${\rm gcd}({\mathtt h}_1,{\mathtt k}_1)={\rm gcd}({\mathtt h}_2,{\mathtt k}_2)=1$ and that
${\mathtt h}_1/{\mathtt k}_1=2q_2+q_3$, ${\mathtt h}_2/{\mathtt k}_2=q_3-q_2$ are called the {\it characteristic numbers} of $\gamma$. The third main result is the following.
\medskip

\noindent {\bf Theorem C}.{\it \hspace{1pt} Let $\gamma$ be a closed string with characteristic numbers $({\mathtt h}_1,{\mathtt k}_1,{\mathtt h}_2,{\mathtt k}_2)$ and wave number ${\mathtt n}$, then

\noindent $\bullet$ $\widehat{{\rm G}}_{\gamma}$ is a non-trivial subgroup of a unique maximal torus ${\rm T}^2_{\gamma}$;

\noindent $\bullet$  $|[\gamma]|$ doesn't intersect the axes of symmetry;

\noindent $\bullet$ ${\mathtt n}={\rm lcm}({\mathtt k}_1,{\mathtt k}_2)$ and the integers ${\mathtt l}_1={\mathtt n}{\mathtt h}_2/{\mathtt k}_2$, ${\mathtt l}_2=-{\mathtt n}{\mathtt h}_1/{\mathtt k}_1$ are the linking numbers of $\gamma$ with the symmetry axes.}
\medskip

\noindent A consequence of Theorem {\rm C} is that the shape of a closed string is detected by three phenomenological invariants: the wave number and the linking numbers with the two axes of symmetry. It also provides a sort of quantization for closed critical curves of the total strain functional. The reconstruction of a string from the phenomenological invariants requires the inversion of the {\it period map} (see Definition \ref{periodM}). This can be achieved by numerical methods. All other steps involve explicit formulas containing elliptic functions and elliptic integrals. Thus, the procedure can be made operational with the help of a software supporting numerical routines and elliptic functions. 
\medskip

\noindent The paper is organized as follows. Section 1 collects some basic facts about the standard cr-structure of the $3$-sphere. Section 2 is devoted to a preliminary analysis of the main cr-differential invariant of a Legendrian curve. In Section 3 we prove Theorem A. In Section $4$ we investigate closed Legendrian curves with constant curvature  and we characterize closed strings (Theorem \ref{closureconditions}). 
In Section $5$ we prove Theorem {\rm B}. In the last section we find explicit parameterizations of closed strings and we prove Theorem {\rm C}. At the end of the section we discuss some explicit examples.
\medskip

\noindent Numerical and symbolic computations, as well as graphics, are made with the software {\it Mathematica}. In the fourth, fifth and sixth sections, properties of the elliptic functions and integrals are used in a substantial way. In this regard, we follow the standard notation however, we advise the reader that the square of the modulus is used as the fundamental parameter for the Jacobian functions and their integrals. As basic references for the theory of elliptic functions and integrals we use the monographs
\cite{By,La}. For the few basic notions about Legendrian knots used in the paper we refer to \cite{FT,GE}.

\section{Preliminaries}\label{1}
\subsection{The Cauchy-Riemann structure of the 3-sphere}\label{S1.1} Let $\C^{(2,1)}$ denote $\C^3$ with the pseudo-Hermitian inner product
\begin{equation}\label{hp}\langle \mathbf{z},\mathbf{w} \rangle = i(\overline{z}^1w^3-\overline{z}^3w^1)+\overline{z}^2w^2=\sum_{i=1}^3 h_{ij}\overline{z}^i w^j,\quad h_{ij}=\overline{h}_{ji}\end{equation}
and with the complex volume form $\Omega=dz^1\wedge dz^2\wedge dz^3$. The map
$$
z=(z^1,z^2)\in {\rm S}^3\subset \C^2\to [^t(\frac{1+z^1}{2},i\frac{z^2}{\sqrt{2}},i\frac{1-z^1}{2})]\in \CP^2.$$
is an embedding of the 3-dimensional sphere into the complex projective plane, whose image is the
strongly pseudo-convex real hyperquadric $\mathcal{S}\subset \CP^2$ defined by the equation $\langle \mathbf{z},\mathbf{z} \rangle = 0$. The differential $1$-form
\begin{equation}\label{cf}\zeta =- \frac{i}{\overline{{\bf z}}^t\cdot {\bf z}} \langle \bf{z},d{\bf z}\rangle |_{{\rm T}(\mathcal{S})}
\end{equation}
gives on $\mathcal{S}$ an oriented contact structure. The annihilator of $\zeta$ is a complex sub-bundle of ${\rm T}(\CP^2)|_{{\rm T}(\mathcal{S})}$ and defines a {\it Cauchy-Riemann (cr) structure}  on $\mathcal{S}$. Let ${\rm P}_{\infty}\in {\mathcal S}$ be the pont with homogeneous coordinates $(0,0,1)$. The {\it Heisenberg projection}
$$p_h:[{\bf z}]\in \mathcal{S}\setminus \{{\rm P}_{\infty}\}\to \left(\mathrm{Re}(z^2/z^1),\mathrm{Im}(z^2/z^1),\mathrm{Re}(z^3/z^1)\right)\in \R^3$$
is a contact diffeomorphism between  $\mathcal{S}\setminus \{P_{\infty}\}$ and $\R^3$ equipped with the contact form $\widetilde{\zeta}=dz-ydx+xdy$. The special  unitary group ${\mathrm G}\cong {\rm SU}(2,1)$ of (\ref{hp}) acts transitively and almost effectively on ${\mathcal S}$ in the usual way: given a point $[{\bf z}]\in \mathcal{S}$ represented by the isotropic non-zero vector ${\bf z}\in \C^{2,1}$, and given ${\bf A}\in {\rm G}$, then 
${\bf A}\cdot [{\bf z}]=[{\bf A} {\bf z}]$. This action  gives all the cr-transformations of $\mathcal{S}$ \cite{Ca1,CM}. Actually, the cr-transformation group of $\mathcal{S}$ is the quotient Lie group 
$\widehat{{\rm G}}={\rm G}/{\rm Z}_{\rm{G}}$ of ${\rm G}$ by its center ${\rm Z}_{\rm{G}}\cong \Z_3$. For each ${\bf A}\in {\rm G}$, we denote by $[{\bf A}]$ its equivalence class in $\widehat{{\rm G}}$ and by
${\rm A}_1, {\rm A}_2,{\rm A}_3$ its column vectors. Then, $({\rm A}_1, {\rm A}_2,{\rm A}_3)$ is a {\it light cone} basis of $\C^{2,1}$, that is a basis such that $\langle {\rm A}_i,{\rm A}_j\rangle = h_{ij}$, $i,j=1,2,3$ and that $\Omega({\rm A}_1, {\rm A}_2,{\rm A}_3)=1$.
Conversely, if $({\rm A}_1, {\rm A}_2,{\rm A}_3)$ is a light-cone basis of $\C^{2,1}$, then the matrix ${\bf A}$ with column vectors 
${\rm A}_1, {\rm A}_2,{\rm A}_3$ is an element of ${\rm G}$. Choose the point $\mathrm{P}_0$ with homogeneous coordinates $(1,0,0)$ as the origin of $\mathcal{S}$. The isotropy subgroup at $\mathrm{P}_0$ is the closed subgroup
\begin{equation}\label{gauge}{\rm G}_0=\left\lbrace {\bf Y}(\rho,\phi,z,r)=\left(
                     \begin{array}{ccc}
                       \rho e^{i\phi} & -i\rho e^{-i\phi}\overline{z} & e^{i\phi}(r-\frac{i}{2}\rho \|z\|^2) \\
                       0 & e^{-2i\phi} & z \\
                       0 & 0 & \rho^{-1}e^{i\phi} \\
                     \end{array}\right)
                   \right\rbrace,\end{equation}
where $z\in \C$, $\phi,r,\rho\in \R$ and $\rho>0$. The map
$\pi_0:{\bf A}\in {\rm G}\to {\bf A}\cdot \mathrm{P}_0=[{\rm A}_1]\in \mathcal{S}$
is then a principal ${\rm G}_0$-bundle. The Lie algebra of ${\rm G}$ consists of all traceless, skew-adjoint matrices of (\ref{hp}), that is
$${\mathfrak g}=\{{\rm X}\in {\mathfrak sl}(3,\C) :\hspace{2pt}  ^t\overline{{\rm X}}\cdot h + h{\rm X}=0,\hspace{2pt} h=(h_{ij})\}.
$$
We denote by ${\mathfrak h}$ the vector space of the traceless self-adjoint matrices of the pseudo-Hermitian inner product (\ref{hp}).

\subsection{Maximal compact Abelian subgroups}\label{S1.2} 
The maximal compact Abelian subgroups of $\widehat{{\rm G}}$ are conjugate to the two-dimensional torus\footnote{${\bf E}_a^b$, $a,b=1,2,3$ are the elementary matrices $^t(\delta_a^1,\delta_a^2,\delta_a^3)\cdot (\delta_1^b,\delta_1^b,\delta_3^b)$, $a,b=1,2,3$.}
\begin{equation}\label{torus}
{\rm T}^2=\{[ {\rm R}(\theta,\phi)] :{\rm R}(\theta,\phi)= \mathfrak{U}\cdot( e^{i\theta}{\bf E}_1^1+e^{i\phi}{\bf E}_2^2+e^{-i(\phi+\theta)}{\bf E}_3^3) \cdot \mathfrak{U}^{-1},\ \phi,\theta \in \R/2\pi \Z\},
\end{equation}
where
\begin{equation}\label{U}{\mathfrak U} =\frac{1}{\sqrt{2}}({\bf E}_1^1+\sqrt{2}{\bf E}_2^2+{\bf E}_3^3)
+\frac{i}{\sqrt{2}}({\bf E}_3^1+{\bf E}^3_1)
\end{equation}
The arc $\Sigma=\{[^t(1,r,ir^2/2)] : r\in [0,\sqrt{2}]\}\subset {\mathcal S}$ is a slice for the action of ${\rm T}^2$ on ${\mathcal S}$. The orbits ${\mathcal T}_r\subset {\mathcal S}$, $r\in (0,\sqrt{2})$ are regular. They can be regarded as the Cauchy-Riemann analogues of the {\it Cyclides of Dupin} in M\"obius geometry \cite{Cay2,Ma,JMN}. By identifying ${\mathcal T}_r$ with its image in $\R^3$ by means of the Heisenberg projection, ${\mathcal T}_r$ is the torus (see Figure \ref{FIG0}) generated by the rotation around the $Oz$-axis of the ellipse parameterized by
$\eta_r:\theta\in \R\to (x_r(\theta),y_r(\theta),z_r(\theta))\in \R^3$, where 
\begin{equation}\label{elp}\begin{cases}
x_r(\theta)= 2r\left(\frac{2+r^2+(2-r^2)\cos(\theta)}{4+r^4+(4-r^4)\cos(\theta)}\right),\\
y_r(\theta)=-\frac{2r(r^2-2)\sin(\theta)}{4+r^4+(4-r^4)\cos(\theta)},\\
z_r(\theta)=\frac{(r^4-4)\sin(\theta)}{4+r^4+(4-r^4)\cos(\theta)}.
\end{cases}
\end{equation}

\begin{defn}\label{AS}{We call ${\mathcal T}_r$ the {\it standard Heisenberg Cyclide} with parameter $r$. The singular orbits of the action of ${\rm T}^2$ are ${\mathcal O}_1={\mathcal T}_0$ and ${\mathcal O}_2={\mathcal T}_{\sqrt{2}}$. Note that ${\mathcal O}_1$ is the intersection of ${\mathcal S}$ with the complex line ${\mathbb P}'=\{[z]\in \CP^2 : z_2=0\}\subset \CP^2$ while ${\mathcal O}_2$ is the intersection of ${\mathcal S}$ with the complex line ${\mathbb P}''=\{[z]\in \CP^2 : z_3=iz_1\}\subset \CP^2$. Hence, ${\mathcal O}_1$ and ${\mathcal O}_2$ are two {\it chains} of ${\mathcal S}$ \cite{Ca1,CM,Koch}. Since they are transversal to the contact distribution we choose the positive orientation with respect to the oriented contact structure of ${\mathcal S}$.}\end{defn}

\noindent In the Heisenberg picture, ${\mathcal O}_1$ is the $Oz$-axis with the orientation that goes from the bottom to the top and ${\mathcal O}_2$ is the Clifford circle $x^2+y^2=2, z=0$, with the counterclockwise orientation with respect to the $Oz$-axis oriented as above  (see Figure \ref{FIG0}). Let ${\bf L}\in {\rm G}$ be the cr-automorphism of order four defined by
\begin{equation}\label{dL}
{\bf L}=\frac{1}{2}\left(({\bf E}_1^1+{\bf E}_3^3)+i({\bf E}_3^1-{\bf E}_1^3)\right)-\frac{1}{\sqrt{2}}
\left(({\bf E}_3^2-{\bf E}_2^3)+i({\bf E}_2^1+{\bf E}_1^2)\right)
\end{equation}
Then, ${\bf L}\cdot {\rm R}(\phi,\psi)\cdot {\bf L}^{-1}={\rm R}(\phi,-(\phi+\psi))$, where ${\rm R}(\phi,\psi)$ is as in (\ref{torus}).
This implies that  $[{\bf L}]$ stabilizes ${\rm T}^2$ and exchanges the two symmetry axes ${\mathcal O}_1$ and ${\mathcal O}_2$.

\begin{figure}[h]
\begin{center}
\includegraphics[height=6.2cm,width=6.2cm]{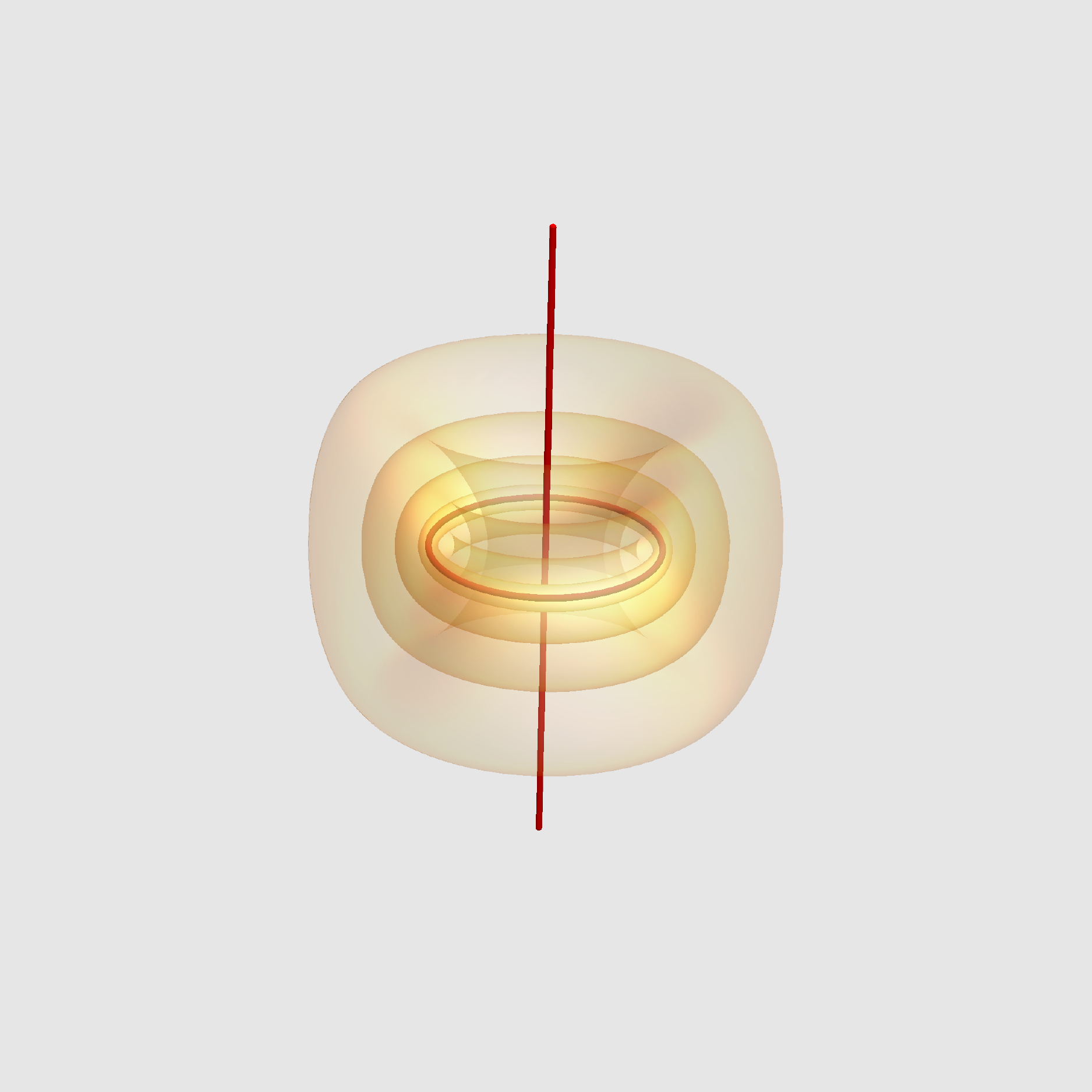}
\includegraphics[height=6.2cm,width=6.2cm]{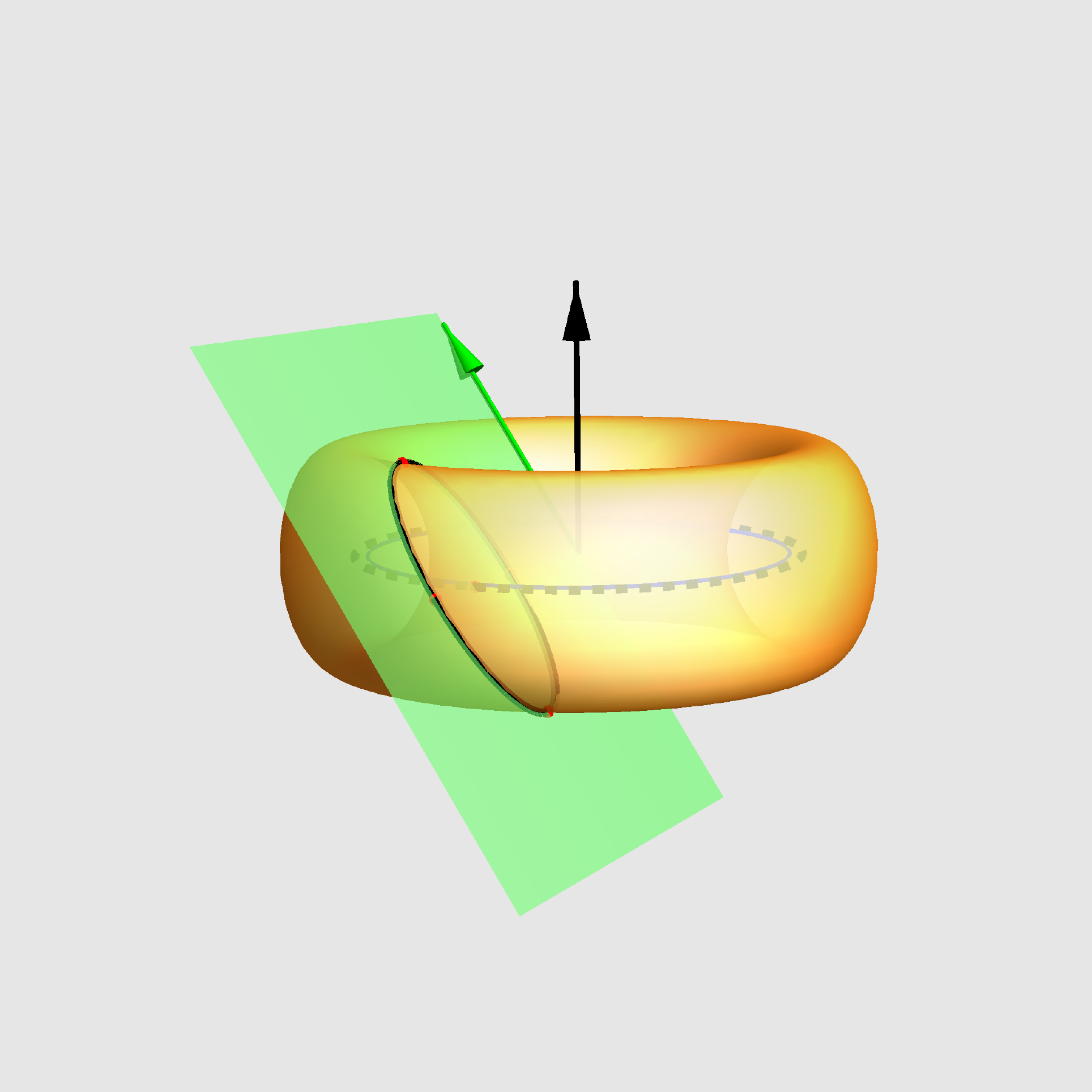}
\caption{\small{Regular and singular orbits (left), the standard Heisenberg Cyclide with $r=1$ and its elliptical profile (right).}}\label{FIG0}
\end{center}
\end{figure}

\section{Legendrian curves}\label{2}

\begin{defn}{A {\it Legendrian curve} is a smooth immersion $\gamma :{\rm I}\subset \R \to \mathcal{S}$ tangent to the contact distribution. Two Legendrian curves $\gamma : {\rm I}\to  \mathcal{S}$ and $\widetilde{\gamma} : \widetilde{{\rm I}}\to  \mathcal{S}$
are said to be {\it cr-congruent} to each other if ${\rm I}=\widetilde{{\rm I}}$ and if there exist ${\bf A}\in {\rm G}$ such that $\widetilde{\gamma}={\bf A}\cdot \gamma$. They are said to be {\it cr-equivalent} to each other if there exist a reparameterization $h: \widetilde{{\rm I}}\to {\rm I}$ such that $\widetilde{\gamma}$ and $\gamma\circ h$ are cr-congruent. A {\it lift} of $\gamma$ is a map $\Gamma:{\rm I}\to \C^3\setminus \{\bf{0}\}$ such that 
$\gamma=[\Gamma]$. We say that $\Gamma$ is {\it normalized} if $\mathrm{det}(\Gamma(t),\Gamma'(t),\Gamma''(t))=i$.}\end{defn}

\noindent It is an easy matter to prove the following Lemma:

\begin{lemma}\label{S2.1L1}{Any Legendrian curve admits a normalized lift. In addition, normalized lifts are uniquely determined up to multiplication by a cubic root of the unity.}\end{lemma}

\noindent Let $\Gamma$ be a normalized lift, the functions
\begin{equation}\label{S2.1F1}a=\mathrm{Im}(\langle \Gamma''',\Gamma''\rangle),\quad b=\frac{1}{2}\langle \Gamma'',\Gamma''\rangle,\quad {\mathtt s}= \sqrt[3]{|a|}\end{equation}
and the differential forms
\begin{equation}\label{S2.1F1.bis}{\mathfrak a}=adt^3,\quad {\mathfrak b}=bdt^2,\quad {\mathfrak s}=\mathtt{s}dt,\end{equation}
do not depend on the choice of $\Gamma$. 

\begin{defn}{In analogy with the terminology used  in projective differential geometry \cite{JM}, the smooth differential forms ${\mathfrak b},{\mathfrak a}$ are called the {\it quadratic and the cubic Fubini's forms}. The functions $b$ and $a$ are the corresponding tensor densities. The linear differential form ${\mathfrak s}$ and the function ${\mathtt s}=\sqrt[3]{|a|}$ are said the {\it infinitesimal strain} and the {\it strain density} respectively.}\end{defn}

\begin{remark}{The Fubini's differential forms are the lower order \cite{O1} cr-differential invariants  of a parameterized Legendrian curve. The infinitesimal strain and the strain density are continuous but not necessarily smooth. From the definition it follows that congruent Legendrian curves have the same Fubini's forms and the same infinitesimal strain.}\end{remark}

\begin{prop}\label{S2.1L2}{Let $\gamma : {\rm I}\to  \mathcal{S}$ be a Legendrian curve and $h : {\rm J}\to {\rm I}$ be a change of the parameter. Then, the Fubini's forms and the infinitesimal strain of $\gamma$ and $\widetilde{\gamma}=\gamma\circ h$ satisfy the transformation law
\begin{equation}\label{tldl} \widetilde{{\mathfrak a}}=h^*({\mathfrak a}),\quad 
\widetilde{{\mathfrak b}}=h^*({\mathfrak b})+\mathtt{S}(h),\quad 
\widetilde{{\mathfrak s}}={\rm sign}(h')h^*({\mathfrak s})
\end{equation}
where 
$$\mathtt{S}(h)=\left(\frac{h'''}{h'}-\frac{3}{2}\frac{h''^2}{h'^2}\right)dt^2$$ 
is the Schwartzian derivative of $h$.
}\end{prop}

\begin{proof}{First we prove that a normalized lift satisfies the following identities:
\begin{equation}\label{S2.1F7}\begin{split}
& \langle \Gamma,\Gamma\rangle = \langle \Gamma,\Gamma'\rangle = \langle \Gamma',\Gamma''\rangle=\langle \Gamma,\Gamma'''\rangle = 0,\\
& \langle \Gamma',\Gamma'\rangle = -\langle \Gamma,\Gamma''\rangle = 1.\end{split}
\end{equation}
Differentiating $\mathrm{det}(\Gamma,\Gamma',\Gamma'')=i$ we find $\mathrm{det}(\Gamma,\Gamma',\Gamma''')=0$. Then, $\Gamma'''=p\Gamma + q\Gamma'$ where $p,q$ are smooth functions. We then have
\begin{equation}\label{S2.1F3}\langle \Gamma,\Gamma\rangle = \langle \Gamma,\Gamma'\rangle = \langle \Gamma,\Gamma'''\rangle = 0.\end{equation}
Differentiating $\langle \Gamma,\Gamma'\rangle =0$ we get
\begin{equation}\label{S2.1F4}\langle \Gamma,\Gamma''\rangle + \langle \Gamma',\Gamma'\rangle = 0.\end{equation}
Taking the derivative of (\ref{S2.1F4}) we obtain  $\langle \Gamma,\Gamma'''\rangle +2\langle \Gamma',\Gamma''\rangle + \langle \Gamma'',\Gamma'\rangle =0$. Then, using (\ref{S2.1F3}), we deduce that
\begin{equation}\label{S2.1F5}\langle \Gamma',\Gamma''\rangle = 0.\end{equation}
This implies
\begin{equation}\label{S2.1F6}\langle \Gamma',\Gamma'\rangle = v^2,\quad \langle \Gamma,\Gamma''\rangle =-v^2,\end{equation}
where $v$ is a positive constant. We put
\begin{equation}\label{Wilc}\mathrm{B}_1=\Gamma,\quad \mathrm{B}_2=\frac{1}{v}\Gamma',\quad \mathrm{B}_3=-\frac{i}{v^2}\left(\Gamma''+\frac{1}{2v^2}\langle \Gamma'',\Gamma''\rangle \Gamma\right).\end{equation}
From (\ref{S2.1F3}), (\ref{S2.1F4}),(\ref{S2.1F5}) and (\ref{S2.1F6}) it follows that $\langle \mathrm{B}_j,\mathrm{B}_i\rangle = h_{ji}$. Hence,
$$1=|\mathrm{det}(\mathrm{B}_1,\mathrm{B}_2,\mathrm{B}_3)|=v^{-3}|\mathrm{det}(\Gamma,\Gamma',\Gamma'')|=v^{-3}.$$
Therefore, $v=1$. Putting $v=1$ in (\ref{S2.1F6}), we infer that $\langle \Gamma',\Gamma'\rangle =- \langle \Gamma,\Gamma''\rangle =1$. So, (\ref{S2.1F7}) is proved.
\noindent Now we are in a position to deduce the transformation laws (\ref{tldl}). If $\Gamma$ is a normalized lift of $\gamma$, then
$\widetilde{\Gamma}=h'^{-1}\Gamma\circ h$ is a normalized lift of $\widetilde{\gamma}$. From this we get
\begin{equation}\label{S2.1F8} \widetilde{\Gamma}''=h'\Gamma''\circ h - \frac{h''}{h'}\Gamma'\circ h+\left(2\frac{h''^2}{h'^3}-\frac{h'''}{h'^2}\right)\Gamma\circ h.\end{equation}
Using (\ref{S2.1F8}) and (\ref{S2.1F7}), we have
$$\widetilde{b} = h'^2(b\circ h) +\frac{h'''}{h'}-\frac{3}{2}\frac{h''^2}{h'^2}.$$
Then, $\widetilde{{\mathfrak b}}=h^*({\mathfrak b})+\mathtt{S}(h)$. Differentiating (\ref{S2.1F8}) we obtain
$$\widetilde{\Gamma}'''= h'^2\Gamma^{'''}\circ h +\left(\frac{3h''^2}{h'^2}-\frac{2h'''}{h'}\right)\Gamma'\circ h-
\left(\frac{6h''^3}{h'^4}-\frac{6h'' h'''}{h'^3}+ \frac{h^{(4)}}{h'^2}\right)\Gamma \circ h.
$$
Combining this identity with (\ref{S2.1F7}) we get
$$\widetilde{a}={\rm Im}(\langle \widetilde{\Gamma}''',\widetilde{\Gamma}''\rangle) = h'^3 {\rm Im}(\langle \Gamma'''\circ h,\Gamma''\circ h\rangle) = h'^3(a\circ h).
$$
Then, $\widetilde{{\mathfrak a}}=h^*({\mathfrak a})$. Obviously, this implies $\widetilde{{\mathfrak s}}={\rm sign}(h')h^*({\mathfrak s})$.}\end{proof}

\begin{defn}{Borrowing  the terminology of classical projective differential geometry \cite{Cay,Ca2,Ha,OT,TU}, we say that $\gamma(t_*)$ is a {\it sextactic point} if  ${\mathfrak a}|_{t_*}=0$.  A Legendrian curve with no sextactic points is said {\it generic}. If ${\mathfrak a}=0$, then $\gamma$ is said a {\it Legendrian cycle}.}\end{defn}

\begin{remark}\label{ccy}{A cycle is a trivial Legendrian knot equivalent  to
$t\to [^t(1,t,it^2/2)]$. Its Maslov index is zero and its Bennequin-Thurston invariant is $-1$. Thus, according to the Eliashberg's classification of Legendrian unknots \cite{EF,ET1}, the cycles are representatives of the unique Legendrian isotopy class of Legendrian unknots with Bennequin-Thurston invariant $-1$. The Legendrian isotopy class of any other Legendrian unknot can be represented by a stabilization \cite{ET1} of a cycle.  Let $\gamma$ be a Legendrian curve. Then, for every $t_*\in {\rm I}$, there exist a unique cycle passing through $\gamma(t_*)$ with analytic contact of order $\ge 3$ with $\gamma$ at  $\gamma(t_*)$ (see Figure \ref{FIG1}). The order of contact is exactly $3$ if $\gamma(t_*)$ is not a sextactic point. Otherwise, the order of contact is $>3$. The value ${\mathtt s}(t_*)$ of the strain density at $t=t_*$ is a measure of how much the fourth-order jet of $\gamma$ at $\gamma(t_*)$ differs from that of its osculating cycle at the contact point. We refer to \cite{Ca3,GG,Je,JM,M2} for the notion of analytic contact and the related concept of deformation.}\end{remark}

\begin{figure}[h]
\begin{center}
\includegraphics[height=6.2cm,width=6.2cm]{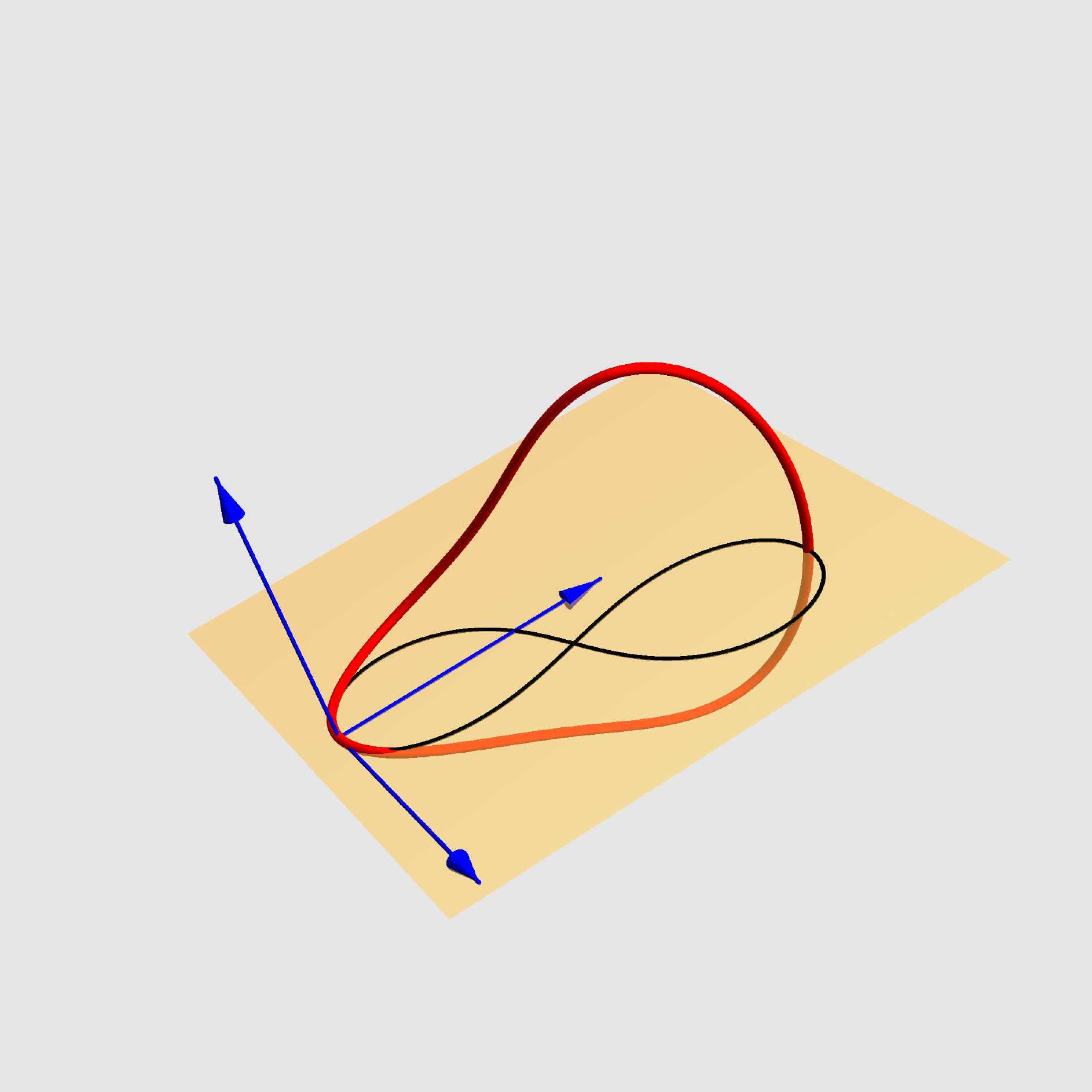}
\includegraphics[height=6.2cm,width=6.2cm]{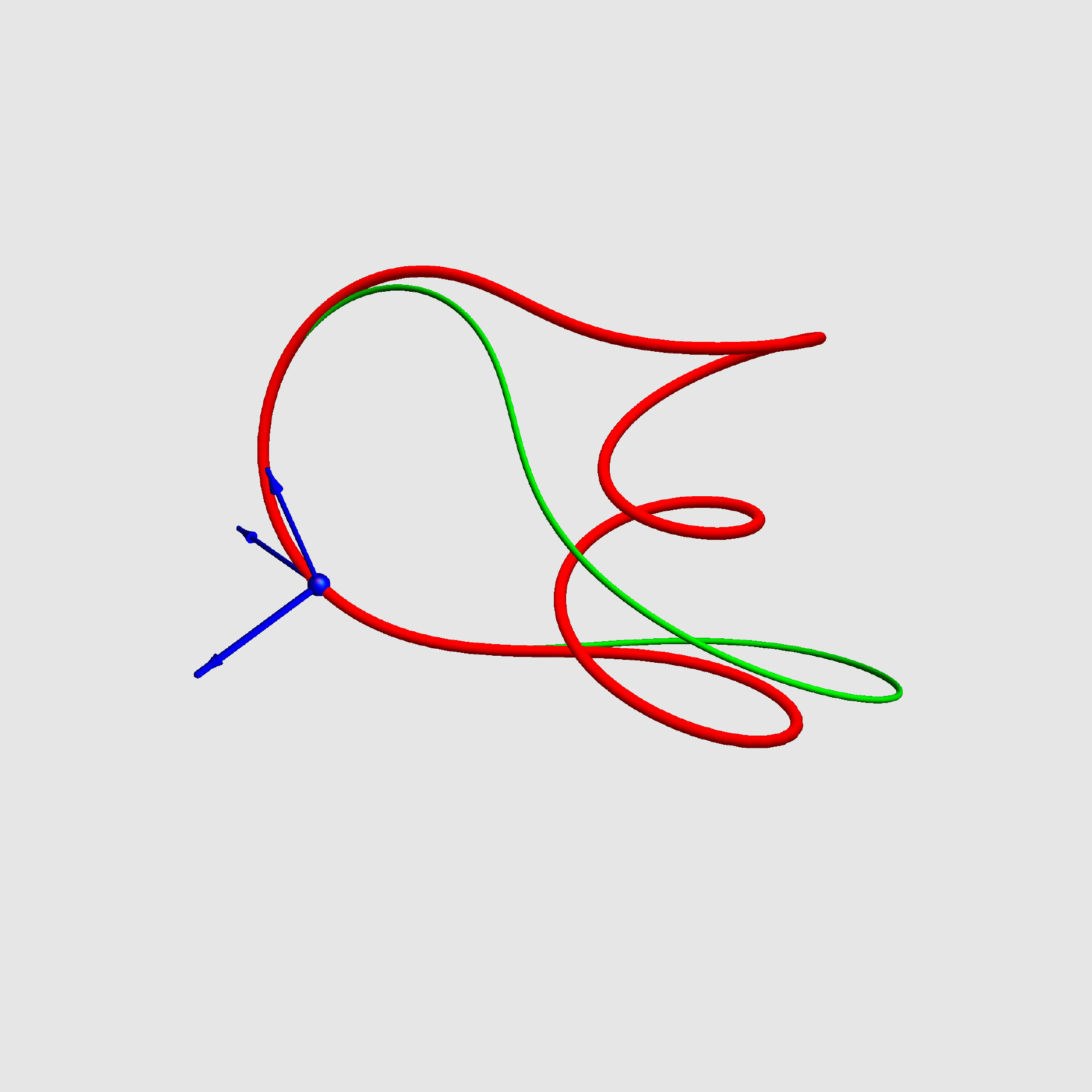}
\caption{\small{A cycle and its Lagrangian projection on the left; a generic curve (red) and one of its osculating cycles (green) on the right.}}\label{FIG1}
\end{center}
\end{figure}

\begin{defn}{A generic Legendrian curve whose strain density is identically equal to $1$ is said {\it parametererized by its natural parameter}. The quadratic Fubini's density of a natural parameterization $\gamma$ is called the {\it cr-curvature} of $\gamma$. We adopt the notation $\kappa$ to denote the cr-curvature.
}\end{defn}

\begin{remark}\label{parameter}{Given a generic Legendrian curve $\gamma : {\rm I}\to {\mathcal S}$, there is a change of parameter $h:{\rm I}\to {\rm J}$ such that $\gamma\circ h^{-1}$ is parameterized by the natural parameter.The natural parameters differ by an additive constant, thus they define a {\it unimodular affine structure},  i.e. an atlas of ${\rm I}$ whose transition functions are special affine transformations. Note that, the natural parameters induce a {\it canonical orientation} on a generic Legendrian curve.}\end{remark}

\begin{defn}{A {\it moving frame} along $\gamma : {\rm I}\to  \mathcal{S}$ is a lift of $\gamma$ to ${\rm G}$, that is a smooth map 
${\bf B} : {\rm I}\to  \rm{G}$ such that $\pi_0\circ {\bf B}=\gamma$. If ${\bf B}$ is a moving frame, any other is given by 
${\bf B}\cdot {\bf Y}(\rho,\phi,z,r)$, where $\rho,\phi,r:{\rm I}\to \R$, $z:{\rm I}\to \C$ are smooth functions and ${\bf Y}(\rho,\phi,z,r):{\rm I}\to {\rm G}_0$ is as in (\ref{gauge}). Given a moving frame ${\bf B}$ we denote by $\mathcal{B}$  the $\mathfrak{g}$-valued smooth function such that
$\mathcal{B}= {\bf B}^{-1}\cdot {\bf B}'$. If ${\bf B}$  and $\widetilde{{\bf B}}$ are two moving frames along $\gamma$ and if  $\widetilde{{\bf B}}= {\bf B} \cdot {\bf Y}(\rho,\phi,z,r)$, then $\widetilde{\mathcal{B}}= {\bf Y}^{-1}\cdot \mathcal{B}\cdot {\bf Y} + {\bf Y}^{-1}{\bf Y}'$.}\end{defn}

\begin{defn}{Let  $\Gamma$ be a normalized lift and ${\rm B}_j:{\rm I}\to \C^{2,1}\setminus \{0\}$, $j=1,2,3$, be defined by $\mathrm{B}_1=\Gamma$, $\mathrm{B}_2=\Gamma'$ and $\mathrm{B}_3=-i\left(\Gamma'' + b \Gamma\right)$ (cfr. (\ref{Wilc})).
From the proof of Proposition \ref{S2.1L2} one sees that $({\rm B}_1,{\rm B}_2,{\rm B}_3)|_t$ is a light-cone basis of $\C^{2,1}$, for every $t\in {\rm I}$. Then,
${\bf B}=({\rm B}_1,{\rm B}_2,{\rm B}_3):{\rm I}\to {\rm G}$
is a moving frame, the {\it Wilczynski frame} along $\gamma$. If $\widetilde{{\bf B}}$ is another Wilczynski frame, then $\widetilde{{\bf B}}=\varepsilon {\bf B}$, where 
$\varepsilon$ is a cubic root of the unity. The map ${\mathcal B}$  of a Wilczynski frame can be written as
\begin{equation}\label{S2.2.F4}{\mathcal B}(a,b)= {\bf E}_2^1+i{\bf E}^2_3+b(i{\bf E}^3_2-{\bf E}^2_1)+a{\bf E}^3_1.\end{equation}}\end{defn}

\begin{defn}{Let $\gamma$ be a Legendrian curve. The function
\begin{equation}\label{strdn}\begin{split}
\mathtt{t}=&\frac{4400}{81} a a'^3 a''+a^2 \left(-\frac{400}{27} b a'^3-\frac{200}{9} a' a''^2-\frac{400}{27}
a'^2 a^{(3)}\right)+\\
&+a^3 \left(\frac{25}{3} a'^2 b'+\frac{50}{3} b a' a''+\frac{50}{9}
a'' a^{(3)}+\frac{25}{9} a' a^{(4)}\right)+\\
&+a^4 \left(-\frac{16}{3} b^2 a'-5 b' a''-3 a' b''-\frac{10}{3} b a^{(3)}-\frac{1}{3}
a^{(5)}\right)+\\
&+a^5 \left(8 b b'+b^{(3)}\right)-\frac{6160}{243} a'^5.
\end{split}
\end{equation}
and the differential form ${\mathfrak t}={\mathtt t}dt^{20}$ are called
the {\it stress density} and the {\it stress tensor} of $\gamma$ respectively.}\end{defn}

\noindent Using Proposition \ref{S2.1L2} and with elementary but tedious computations, one can easily prove that, if $\widetilde{\gamma}=\gamma\circ h$ is a reparameterization of $\gamma$, then $\widetilde{\mathfrak{t}}=h^*(\mathfrak{t})$.

\section{The strain functional}\label{3}
\subsection{Admissible variations}\label{S3.1} An {\it admissible variation} of a Legendrian curve ${\gamma}:{\rm I}\to \mathcal{S}$ is a smooth map ${\bf g}:{\mathcal R}_{\epsilon}\to \mathcal{S}$ defined on an open rectangle ${\mathcal R}_{\epsilon}={\rm I}\times (-\epsilon,\epsilon)$, such that

\noindent $\bullet$ ${\bf g}(t,0)=\gamma(t)$, for every $t\in {\rm I}$;

\noindent $\bullet$ $g_{\tau}: t\in {\rm I}\to {\bf g}(t,\tau)\in \mathcal{S}$ is a Legendrian curve, $\forall \tau \in (-\epsilon,\epsilon)$;

\noindent $\bullet$ the  variational vector field $\mathfrak{v}_{{\bf g}}:t\in {\rm I}\to {\bf g}_*|_{(t,0)}(\partial_{\tau})\in {\rm T}(\mathcal{S})$, is compactly supported; 

\noindent $\bullet$ if ${\mathtt s}_{\tau}$ is the  strain density of $g_{\tau}$ and ${\rm K}_{\bf g}$ is the support of $\mathfrak{v}_{{\bf g}}$, then  
$$\mathfrak{S}_{\bf{g}}:\tau\in (-\epsilon,\epsilon)\to \int_{{\rm K}_{\bf g}}{\mathtt s}_{\tau}dt\in \R$$
is differentiable at $\tau=0$.

\noindent If ${\bf g}$ is an admissible variation, then there exist a smooth map ${\bf B}_{\bf {g}}:  \mathcal{R}_{\epsilon}\to  {\rm G}$ such that ${\bf B}_{\tau}:t\in {\rm I}\to {\bf B}_{\bf {g}}(t,\tau)\in {\rm G}$ is a Wilczynski frame of $g_{\tau}$, for every $\tau\in (-\epsilon,\epsilon)$. We call ${\bf B}_{\bf {g}}$ a {\it Wilczynski frame} along ${\bf g}$. We denote by $a_{\tau}$ and $b_{\tau}$ the Fubini's densities of $g_{\tau}$ and we put $a_{{\bf g}}(t,\tau)=a_{\tau}(t)$, $b_{{\bf g}}(t,\tau)=b_{\tau}(t)$,  ${\mathtt s}_{{\bf g}}(t,\tau)=\mathtt{s}_{\tau}(t)$. Let $\widetilde{{\mathcal B}},\mathcal{V}:\mathcal{R}_{\epsilon}\to \mathfrak{g}$ be defined by
\begin{equation}\label{S3.1F1}{\bf B}_{\bf {g}}^{-1}d{\bf B}_{\bf {g}} = \widetilde{{\mathcal B}}dt + \mathcal{V}d\tau,\end{equation}
$\widetilde{\mathfrak{r}}^h_k+i\widetilde{{\mathfrak s}}^h_k$ be the entries of $\mathcal{V}$ and
$\mathfrak{r}^h_k, {\mathfrak s}^h_k:{\rm I}\to \R$ be given by $\mathfrak{r}^h_k(t)=\widetilde{\mathfrak{r}}^h_k(t,0)$ and by $\mathfrak{s}^h_k(t)=\widetilde{\mathfrak {s}}^h_k(t,0)$. Differentiating (\ref{S3.1F1}) we get $\partial_{\tau}{\mathcal B}-\partial_t\mathcal{V}=[{\mathcal B},\mathcal{V}]$. In turn, this implies
\begin{equation}\label{S3.1F2}\begin{split}
\partial_{\tau} a_{{\bf g}}\big|_{(t,0)}=&\frac{1}{6}\Big(6a'\mathfrak{r}^2_1 +18a
(\mathfrak{r}^2_1)'-(16bb'+2b^{(3)})\mathfrak{r}^3_1-\\
&-(16b^2+9b'')(\mathfrak{r}^3_1)'
-15b'(\mathfrak{r}^3_1)''-10b(\mathfrak{r}^3_1)^{(3)}-(\mathfrak{r}^3_1)^{(5)}\Big)
\end{split}
\end{equation}
\subsection{The strain functional and its critical curves}\label{S3.2}

Let ${\rm J}\subset {\rm I}$ be a closed interval. The integral
$$\mathfrak{S}_J(\gamma)=\int_J {\mathtt s}dt,$$
is the {\it total strain} of the Legendrian arc $\gamma({\rm J})$. It measures  of how much a Lagrangian arc is far from being a cycle. By construction, is invariant by cr-transformations and reparameterizations.

\begin{defn}{A Legendrian curve  $\gamma$ is {\it critical for the total strain functional} if $\mathfrak{S}_{\bf{g}}'|_{0}=0$,
for every admissible variation.}\end{defn}

\noindent {\bf Theorem A.}\hspace{1pt}{\it A Legendrian curve is critical for the total strain functional if and only if its stress tensor is zero. Furthermore, a critical curve is either a cycle or else is generic.
}
\begin{proof}{The proof consists of three steps.

\noindent {\bf Step I}. We show that if $\gamma$ is critical, then its stress tensor vanishes.

\noindent We begin by proving a preliminary result: suppose that $\gamma:{\rm I}\to {\mathcal S}$ is not a cycle. Put ${\rm I}_*=\{t\in {\rm I}: a(t)\neq 0\}$ and let ${\rm K}=[t_0,t_1]\subset {\rm I}_*$ be a closed interval. 
If $w : {\rm I}\to \R$ is a smooth function such that ${\rm supp}(w) \subset (t_0,t_1)$, then there exist an admissible variation ${\bf g}$ such that 
\begin{equation}\label{vd}\mathfrak{S}_{\bf{g}}'|_{0}= \int_{{\rm K}}{\mathtt t}(t)\cdot w(t)dt.\end{equation}
Firstly, we construct the variation. Without loss of generality, we assume ${\rm P}_{\infty}\notin |[\gamma]|$. With a possible change of parameter, we can suppose that $a|_{{\rm K}}>0$. Then, there is a regular plane curve $\alpha:t\in {\rm I}\to x(t)+iy(t)\in \C\cong \R^2$ such that 
$$\gamma(t)=[^t(1,\alpha(t),z(t)+\frac{i}{2}|\alpha(t)|^2)],\quad z(t)=\int_{t_0}^t(x'y-xy')du + c.$$
The constant $c$ can be put equal to $0$. Consider the moving frame ${\bf H}_{\gamma}:{\rm I}\to {\rm G}$ along $\gamma$ defined by ${\bf H}_{\gamma}={\rm Id}_{3\times 3}+\alpha {\bf E}^1_2+i\overline{\alpha}{\bf E}^2_3+(z+\frac{i}{2}|\alpha|^2){\bf E}^1_3$. Let $\rho,\phi,r:{\rm I}\to \R$, $\rho>0$ and ${\rm p}:{\rm I}\to \C$ be smooth functions such that, ${\bf B}_{\gamma}={\bf H}_{\gamma}\cdot {\bf Y}(\rho,\phi,{\rm p},r)$ is a  Wilczynski frame along $\gamma$. We put
$$\eta=\frac{2187}{2\|\alpha'\|^2}\dfrac{d}{dt}\left(\rho^{-2}a^{17/3}w\right).$$
By construction, ${\rm supp}(\eta)\subseteq {\rm supp}(w)$. Then we define
\begin{equation}\label{av}\begin{cases}
\beta(t,\tau)=&\alpha(t)+i\tau\eta(t)\alpha'(t),\\
u(t,\tau)=&z(t)+\tau\left(\frac{2187w(t)\cdot a(t)^{17/3}}{\rho(t)^2}-\eta(t){\rm Re}(\overline{\alpha(t)}\alpha'(t)) \right)+\\
&+ \tau^2\int_{t_0}^t \eta^2(s){\rm Im}(\overline{\alpha''(s)}\alpha'(s))ds
\end{cases}
\end{equation}
Choosing $\epsilon>0$ sufficiently small,  and putting ${\mathcal R}_{\epsilon}={\rm I}\times (-\epsilon,\epsilon)$, the map
$${\bf g}:(t,\tau)\in{\mathcal R}_{\epsilon}\to [(1,\beta(t,\tau),u(t,\tau)+\frac{i}{2}|\beta(t,\tau)|^2)^t]\in {\mathcal S}$$
is an admissible variation of $\gamma$. Without loss of generality we may suppose that $a_{\bf g}$ is strictly positive on ${\rm J}\times (-\epsilon,\epsilon)$, where ${\rm J}$ is an open interval such that ${\rm supp}(w)\subset {\rm J}\subset {\rm K}$. We show that ${\bf g}$ satisfies (\ref{vd}).
Let ${\bf H}_{\bf g}: \rm{I}\times (-\epsilon,\epsilon)\to {\rm G}$ be the moving frame along $\bf{g}$ defined by
${\bf H}_{{\bf g}}={\rm Id}_{3\times 3}+\beta {\bf E}^1_2+i\overline{\beta}{\bf E}^2_3+(u+\frac{i}{2}|\beta|^2){\bf E}^1_3$. From (\ref{av}) it follows that
\begin{equation}\label{S3.2.F1}{\bf H}_{{\bf g}}(t,\tau)={\bf H}_{\gamma}{\it (t)}+\tau {\bf L}_{1}{\it (t)} + \tau^2 {\bf L}_2{\it (t)},\end{equation}
where ${\bf L}_1 = \eta(i\alpha'{\bf E}^1_2+\overline{\alpha'}{\bf E}^2_3)+\lambda {\bf E}^1_3$ and
$\lambda=2187\rho^{-2}w\cdot a^{17/3}-\eta \overline{\alpha}\alpha'$. Then, there exist smooth functions $\widetilde{\rho},\widetilde{\phi},\widetilde{r}:\mathcal{R}_{\epsilon}\to \R$, $\widetilde{\rho}>0$ and $\widetilde{{\rm p}}:\mathcal{R}_{\epsilon}\to \C$ such that
${\bf B}_{{\bf g}}= {\bf H}_{\bf g}\cdot {\bf Y}_(\widetilde{\rho},\widetilde{\phi},\widetilde{{\rm p}},\widetilde{r})$ is a Wilczynski frame along ${\bf g}$ and that ${\bf B}_{{\bf g}}|_{(t,0)}={\bf B}_{\gamma}|_{t}$, for every $t\in {\rm I}$. By construction, ${\bf Y}(\widetilde{\rho},\widetilde{\phi},\widetilde{{\rm p}},\widetilde{r})|_{(t,0)}={\bf Y}_{\gamma}|_{t}$, where ${\bf Y}_{\gamma}={\bf Y}(\rho,\phi,{\rm p},r)$. Using (\ref{S3.2.F1}) we obtain
$${\mathcal V}\big|_{(t,0)} = ({\bf B}_{{\bf g}}^{-1}\partial_{\tau}{\bf B}_{{\bf g}})\big|_{(t,0)} = {\bf Y}_{\gamma}^{-1} \partial_{\tau}{\bf Y}\big|_{(t,0)}+{\bf Y}_{\gamma}^{-1}\cdot {\bf H}_{{\bf \gamma}}^{-1}\cdot {\bf L}_1\cdot {\bf Y}_{\gamma}.
$$
Since ${\bf Y}_{\gamma}(t)\in {\rm G}_0$, for every $t\in {\rm I}$, then $({\bf Y}_{\gamma}^{-1}\partial_{\tau}{\bf Y}|_{(t,0)})^2_1=({\bf Y}_{\gamma}^{-1}\partial_{\tau}{\bf Y}|_{(t,0)})^3_1=0$.
It is now a computational matter to check that
\[
\begin{split}({\bf Y}_{\gamma}^{-1}\cdot  {\bf H}_{{\bf \gamma}}^{-1}\cdot{\bf L}_1\cdot {\bf Y}_{\gamma})^2_1  &= -2187 a^{17/3}we^{2i\theta}{\rm p}+ie^{3i\theta}\rho\eta\alpha',\\
({\bf Y}_{\gamma}^{-1}\cdot  {\bf H}_{{\bf \gamma}}^{-1}\cdot {\bf L}_1\cdot {\bf Y}_{\gamma})^3_1 &=2187 a^{17/3}w.
\end{split}
\]
This implies
\[{\mathcal V}^2_1\big|_{(t,0)}=\mathfrak{r}^2_1+i\mathfrak{s}^2_1=-2187 a^{17/3}we^{2i\theta}{\rm p}+ie^{3i\theta}\rho\eta\alpha',\quad 
{\mathcal V}^3_1\big|_{(t,0)}={\mathfrak r}^3_1=2187 a^{17/3}w.
\]
Using (\ref{S3.1F2}) and proceeding with elementary but rather tedious calculations, we get
$\partial_{\tau}{\mathtt s}_{{\bf g}}|_{{\rm{J}\times\{0\}}}\cong_{d,{\rm K}}({\mathtt t}\cdot w)|_{{\rm J}}$
where $f\cong_{d,{\rm {\rm K}}} g$ means that $f=g+r'$, for some smooth function $r$ such that ${\rm supp}(r)\subset {\rm K}$.
Then,
\[\mathfrak{S}_{\bf{g}}'|_{0}=\partial_{\tau}\left( \int_{{\rm K}_{\bf g}}{\mathtt s}_{{\bf g}}dt\right)\Big|_{\tau=0}=
 \int_{{\rm K}_{\bf g}}\partial_{\tau}{\mathtt s}_{{\bf g}}\big|_{(t,0)}dt =\int_{{\rm K}}{\mathtt t}\cdot w dt.
\]
We are now in a position to conclude the proof of the first step. Suppose that $\gamma$ is a critical curve. If $\gamma$ is a cycle there is nothing to prove. If $\gamma$ is not a cycle we denote by $\rm{I}_a$ be the zero set of $a$ and we put ${\rm I}_*={\rm I}\setminus \rm{I}_a$. Then, our preliminary discussion implies that ${\mathtt t}$ is zero on ${\rm I}_*$. Obviously, ${\mathtt t}$ is zero on the interior of $\rm{I}_a$. Hence, ${\mathtt t}$ is everywhere zero. 
 
\noindent {\bf Step II}. We prove that if $\mathfrak{t}=0$, then $\gamma$ is either a cycle or is generic. Preliminarily we show that for every $t_*\in {\rm I}$ there exist an open interval ${\rm J}\subset {\rm I}$ containing $t_*$ and a smooth, strictly increasing function $h:{\rm J}\to \R$ such that the quadratic Fubini's form  of $\gamma\circ h^{-1}:h({\rm J})\to \mathcal{S}$ is zero. The collection of all such functions defines a {\it projective structure} on ${\rm I}$, ie an atlas $\mathfrak{P}_{\gamma}=\{({\rm J}_{\alpha},h_{\alpha})\}_{\alpha\in {\rm A}}$  whose transition functions are orientation-preserving linear fractional transformations. This assertion can be justified as follows: let $b$ be the quadratic Fubini's density of $\gamma$. For every $t_*\in {\rm I}$ and every $h_0,h_1,h_2\in \R$, $h_1>0$, we consider the solution of the Cauchy problem
\[ \frac{h'''}{h'}-\frac{3}{2}\frac{h''^2}{h'^2}+h'^2(b\circ h) =0,\quad
h(t_*)=h_0,\quad h'(t_*)=h_1>0,\quad h''(t_*)=h_2.\]
Shrinking the interval of definition we assume that $h$ is strictly increasing. Proposition \ref{S2.1L2} implies that the Fubini's quadratic form of $\gamma\circ h^{-1}$ is identically zero. We call $h$ a projective chart. We prove that the family $\mathfrak{P}_{\gamma}=\{({\rm I}_{\alpha},h_{\alpha})\}_{\alpha\in {\rm A}}$ of all projective charts is a projective structure on ${\rm I}$. Let $h_{\alpha} : {\rm I}_{\alpha}\to \R$ and $h_{\beta}:{\rm I}_{\beta}\to \R$ be two projective charts such that 
${\rm I}_{\alpha}\cap {\rm I}_{\beta}\neq \emptyset$ and 
$f_{\alpha}^{\beta}=h_{\alpha}\circ h_{\beta}^{-1}$ be the corresponding transition function. Since the Fubini's quadratic forms of 
$\gamma\circ h_{\beta}^{-1}$ and $\gamma\circ h_{\alpha}^{-1}$ are both identically zero, (\ref{tldl})
implies that ${\mathtt S}(f_{\alpha}^{\beta})=0$. Then, $f_{\alpha}^{\beta}$ is a strictly increasing, linear fractional function. Using this projective structure we show that if the stress tensor is zero and if $\gamma(t_*)$ is a sextactic point, then $a$ and all its derivatives vanish at $t_*$.
From Proposition \ref{S2.1L2} and making use of the projective structure, we may assume ${\mathfrak b}=0$. For every $n\in {\mathbb N}$ we put
\[\begin{split}
c_{1,n}& =\frac{\Gamma(5n+1)}{\Gamma(n+1)^5},\\ 
c_{2,n}&= \frac{\Gamma(5n+1)}{\Gamma(n)\Gamma(n+1)^3\Gamma(n+2)},\hspace{30pt} c_{3,n}=\frac{\Gamma(5n+1)}{\Gamma(n)^2\Gamma(n+1)\Gamma(n+2)^2},\\
c_{4,n}&=\frac{\Gamma(5n+1)}{\Gamma(n-1)\Gamma(n+1)^2\Gamma(n+2)^2},\hspace{14pt} c_{5,n}=\frac{\Gamma(5n+1)}{\Gamma(n-1)\Gamma(n)\Gamma(n+2)^3},\\
c_{6,n}&=\frac{\Gamma(5n+1)}{\Gamma(n-2)\Gamma(n+1)\Gamma(n+2)^3},\hspace{16pt}  c_{7,n}=\frac{\Gamma(5n+1)}{\Gamma(n-3)\Gamma(n+2)^4},
\end{split}
\]
where, in this context, $\Gamma$ is the Euler gamma function. Note that
\[\begin{split}
\mathfrak{c}_n &= 6160 c_{1,n}-13200 c_{2,n}+5400 c_{3,n}+3600 c_{4,n}-1350 c_{5,n}-675 c_{6,n}+81 c_{7,n}\\
&=\frac{4\Gamma(5n+1)(4n^4+76n^3+519n^2+1501n+1540)}{\Gamma(1+n)\Gamma(n+2)^4}>0.
\end{split}\]
Let $f$ be a smooth function and denote by $\equiv_n$  the equality of functions modulo the ideal generated by $f,f',...,f^{(n)}$. Then, proceeding by induction, we see that
\begin{equation}\label{S2.4F2}\begin{cases}
(f'^5)^{(5n)}\equiv_n c_{1,n}(f^{(n+1)})^5,\\
(f f'^3 f'')^{(5n)}\equiv_n c_{2,n}(f^{(n+1)})^5,\\
(f^2f'f''^2)^{(5n)}\equiv_n c_{3,n}(f^{(n+1)})^5,\\
(f^2f'^2f^{(3)})^{(5n)}\equiv_n c_{4,n}(f^{(n+1)})^5,\\
(f^3f''f^{(3)})^{(5n)}\equiv_n c_{5,n}(f^{(n+1)})^5,\\
(f^3f'f^{(4)})^{(5n)}\equiv_n c_{6,n}(f^{(n+1)})^5\\
(f^4f^{(5)})^{(5n)}\equiv_n c_{7,n}(f^{(n+1)})^5.
\end{cases}
\end{equation}
Putting $b=0$ in (\ref{strdn}), the stress density takes the form
\begin{equation}\label{S2.4F3}\begin{split}{\mathtt t} = & -6160a'^5+13220 a a'^3 a''-5400a^2a'a''^2-3600a^2a'^2a^{(3)}+\\
&+1350a^3a''a^{(3)}+675a^3a'a^{(4)}-81a^4a^{(5)}.
\end{split}
\end{equation}
If ${\mathtt t}=0$ and $a|_{t_*}=0$, then (\ref{S2.4F3}) implies $a'|_{t_*}=0$. By induction, suppose that $a^{(k)}|_{t_*}=0$, for every $k=0,...,n$. From (\ref{S2.4F2}) and (\ref{S2.4F3}), we obtain
$$0 = \frac{d^{5n}\mathtt{t}}{dt^{5n}}\Big|_{t_*}=-\mathfrak{c}_n\cdot \big(a^{(n+1)}\big|_{t_*}\big)^5+\mathfrak{r}\big|_{t_*},$$
where $\mathfrak{r}$ belongs to the ideal spanned by $a,a',...,a^{(n)}$. By the inductive hypothesis, $\mathfrak{r}|_{t_*}=0$. Since $\mathfrak{c}_n\neq 0$, we have $a^{(n+1)}|_{t_*}=0$. Thus, $a$ and all its derivatives vanish at $t_*$. We conclude the proof of the second step. By contradiction, suppose that $\mathfrak{t}=0$ and that ${\rm I}_r=\{t\in {\rm I}: a(t)\neq 0\}$ is a non-empty proper subset of ${\rm I}$. Let  
${\rm I}^*_r$ be a connected component of ${\rm I}_r$. There are two possibilities: either ${\rm sup}({\rm I}^*_r)<{\rm sup}({\rm I})$ or ${\rm inf}({\rm I}^*_r)>{\rm inf}({\rm I})$. Consider the first case, ie $t_*={\rm sup}({\rm I}^*_r)<{\rm sup}({\rm I})$. Take $\epsilon >0$ such that
${\rm J}=(-\epsilon+t_*,\epsilon+t_*)\subset {\rm I}$ and that $a(t)\neq 0$, for every $t\in (-\epsilon+t_*,t_*)$.
We may assume that ${\rm J}$ is the domain of definition of a chart $\phi:{\rm J}\to \R$ of the projective atlas ${\mathfrak P}_{\gamma}$ such that $\phi(t_*)=0$.  We put ${\rm J}'=(-\epsilon',\epsilon'):=\phi({\rm J})$. Then, $\widetilde{\gamma}=\gamma\circ \phi^{-1}: {\rm J}'\to {\mathcal S}$ is a Legendrian curve with zero quadratic differential and zero stress tensor. In addition, the cubic density $\widetilde{a}$ of $\widetilde{\gamma}$ vanishes at $t=0$ and $\widetilde{a}|_{t}\neq 0$, for every $t\in (-\epsilon',0)$. By our previous discussion, we know that $\widetilde{a}^{(n)}|_{t=0}=0$, for every $n\in {\mathbb N}$. Denote by $\widetilde{{\bf B}}:{\rm J}'\to {\rm G}$ a Wilczynski frame along $\widetilde{\gamma}$. 
Without loss of generality, we may assume that $\widetilde{{\bf B}}|_0={\rm Id}_{3\times 3}$. Let 
$\widehat{a}: {\rm J}'\to \R$ be the smooth function defined by 
\[\begin{cases}\widehat{a}|_t&=\widetilde{a}|_t,\quad {\rm if}\hspace{2pt} t\in (-\epsilon,0),\\
\widehat{a}|_t & =0,\hspace{17pt} {\rm if}\hspace{2pt} t\in [0,\epsilon).\end{cases}\]
Retaining the notation (\ref{S2.2.F4}), we put $\widehat{\mathcal {B}}=\mathcal {B}_{\widehat{a},0}$. Denote by $\widehat{{\bf B}}: {\rm J}'\to {\rm G}$ the solution of the linear system  
\begin{equation}\label{S2.4F4}\widehat{{\bf B}}^{-1}\widehat{{\bf B}}'= \widehat{\mathcal {B}},\quad   \widehat{\bf {B}}|_0= {\rm Id}_{3\times 3}.\end{equation}
Then, $\widehat{\gamma}:{\rm J}'\to [\widehat{{\rm B}}_1]\in {\mathcal S}$ is a Legendrian curve and
$\widehat{{\bf B}}$ is a Wilczynski frame along $\widehat{\gamma}$. By the Cartan-Darboux congruence Theorem \cite{JMN}, we have $\widehat{{\bf B}}_{(-\epsilon,0]}=\widetilde{{\bf B}}_{(-\epsilon,0]}$. Note that $\widehat{a}$ is the cubic density of $\widehat{\gamma}$ and that $\widehat{{\mathfrak b}}=0$  and  $\widehat{{\mathfrak t}}=0$. We put
$$k=\frac{7\widehat{a}'^2-6\widehat{a}\widehat{a}''}{18\widehat{a}^{8/3}},\quad \dot{k}=\frac{1}{\sqrt[3]{\widehat{a}}}k',\quad \ddot{k}=\frac{1}{\sqrt[3]{\widehat{a}}}\dot{k}'
$$
and we define ${\bf Y}:(-\epsilon',0)\to {\rm G}_0$ and ${\rm H}: (-\epsilon',0)\to {\mathfrak h}$ by
$${\bf Y}={\bf Y}\left(\frac{1}{\sqrt[3]{|\widehat{a}|}},\frac{1-sign(\widehat{a})}{2}\pi,i\frac{\widehat{a}'}{3\widehat{a}},0\right)$$
and by
$${\rm H}=2\left(i({\bf E}^1_3+{\bf E}^2_1)+{\bf E}^3_2\right)+2k\big(\frac{1}{3}({\bf E}_1^1-2{\bf E}_2^2+{\bf E}^3_3)-ik{\bf E}^3_1\big)+\frac{2}{3}\dot{k}({\bf E}^2_1+i{\bf E}_2^3)-\frac{2}{3}i\ddot{k}{\bf E}^3_1.
$$
Let $\Lambda : (-\epsilon',0)\to {\mathfrak h}$ be given by 
$\Lambda={\bf Y}^{-1}\cdot {\rm H}\cdot \bf{Y}$.
A direct computation shows that $\Lambda^3_1=2i\widehat{a}^{-2/3}$. Then, $\Lambda$ can't be extend smoothly on the whole interval ${\rm J}'$. On the other hand, $\widehat{{\mathfrak t}}=0$ implies  
$(\Lambda' +[\widehat{\mathcal{B}},\Lambda])|_{(-\epsilon',0)}=0$. Consequently,
$(\widehat{\bf{B}}\cdot \Lambda\cdot \widehat{\bf{B}}^{-1})|_{(-\epsilon',0)}={\mathfrak m}$,
where $\mathfrak{m}$ is a fixed element of ${\mathfrak h}$.
Then, $\widehat{\bf{B}}^{-1}\cdot {\mathfrak m}\cdot \widehat{\bf{B}}$  is a smooth extension of $\Lambda$ on ${\rm J}'$. We have thus come to a contradiction. If ${\rm inf}({\rm I}^*_r>{\rm inf}({\rm I})$ we can use similar arguments, coming to the same conclusion.

\noindent {\bf Step III}. We prove that if ${\mathfrak t}=0$, then $\gamma$ is critical. By the second step, if ${\mathfrak t}=0$, then either $\gamma$ is a cycle or else is generic. In the first case $\gamma$ is obviously critical. Assume that $\gamma$ is a natural parameterization of a generic Legendrian curve with ${\mathfrak t}=0$. Let ${\bf g}$ be an admissible variation defined on the open rectangle ${\mathcal R}_{\epsilon}$. Since $a_{\bf g}(t,0)=1$, for every $t$, then $a_{\bf g}$ is strictly positive on an open neighborhood of ${\rm I}\times \{0\}$. From (\ref{strdn}) and (\ref{S3.1F2}) we have 
\[
\partial_{\tau}{\mathtt s}_{{\bf g}}\big|_{(t,0)}\cong_{{\rm d},{\rm K}_{\bf g}}-\frac{1}{9}(\kappa^{(3)}+8\kappa \kappa')\big|_t\cdot \mathfrak{r}^3_1\big|_t= -\frac{1}{9}{\mathtt t}\mathfrak{r}^3_1 = 0.
\]
Then, $\mathfrak{S}_{\bf{g}}'|_{0}=0$. Consequently, a generic Legendrian curve with zero stress tensor is critical. This concludes the proof of the Theorem.}\end{proof}

\begin{remark}\label{S2.4R1}{Leaving aside the cycles, a Legendrian curve such that ${\mathfrak t}=0$ is generic. Thus, it can be parameterized by the natural parameter. Putting $a=1$ in (\ref{strdn}) one sees that the cr-curvature $\kappa$ is a solution of the third-order ode
\begin{equation}\label{S2.5F}\kappa'''+8\kappa\kappa'=0.\end{equation}
Therefore, either $\kappa$ is constant or else $-2\kappa/3$ is a real form of a Weierstrass $\wp$-function. In the first case $\gamma$ is an orbit of a $1$-parameter group of cr-transformations. If $\kappa$ is non-constant, the natural parameterization of the curve can be found by solving a system of linear ode whose coefficients are real forms of elliptic functions. This is a classical problem already studied by Picard \cite{Pi} at the turn of the nineteen century. It is known that the solutions can be written in terms of elliptic functions and incomplete elliptic integrals. In the last section we will explicitly address this problem when $\kappa$ is a periodic solution of 
(\ref{S2.5F}). }\end{remark}

\begin{defn}\label{S2.4R2}{Let $\gamma:{\rm I}\to {\mathcal S}$ be a natural parameterization of a generic curve,
$\bf{B}$ be a Wilczynski along $\gamma$ and $\Lambda:{\rm I}\to {\mathfrak h}$ be defined by 
\begin{equation}\label{Lambda}\Lambda=2\left(i({\bf E}^1_3+{\bf E}^2_1)+{\bf E}^3_2\right)+2\kappa\big(\frac{1}{3}({\bf E}_1^1-2{\bf E}_2^2+{\bf E}^3_3)-i\kappa{\bf E}^3_1\big)+\frac{2}{3}\kappa'({\bf E}^2_1+i{\bf E}_2^3)-\frac{2}{3}i\kappa''{\bf E}^3_1.\end{equation}
If ${\mathfrak t}=0$, then (\ref{S2.5F}) implies
\begin{equation}\label{cl}\bf{B}\cdot \Lambda \cdot \bf{B}^{-1}={\mathfrak m},\end{equation}
where ${\mathfrak m}$ is a fixed element of ${\mathfrak h}$, the {\it momentum} of $\gamma$.
}\end{defn}

\begin{remark}{The conservation law (\ref{cl}) has the following theoretical explanation. Put $\mathrm{Z}=\mathrm{G}\times \R^3$, denote by $k,\dot{k},\ddot{k}$ the fiber coordinates and by $(\alpha^i_j+i\beta^i_j)_{1\le,i,j\le 3}$ the pull-back on ${\rm Z}$ of the Maurer-Cartan form of ${\rm G}$. Let $\gamma : {\rm I}\to {\mathcal S}$ be a natural parameterization of a generic Legendrian curve with cr-curvature $\kappa$, zero stress tensor and Wilczynski frame ${\bf B}$. Then, $s\to ({\bf B},\kappa,\kappa',\kappa'')|_s\in {\rm Z}$ is a lift of $\gamma$ to ${\rm Z}$, the {\it prolongation} of $\gamma$. The prolongations are integral curves of  the Reeb vector field $\mathfrak{X}$ of the contact form
$\chi = \alpha^2_1+(\alpha^1_3-\alpha^2_1)/3-(\ddot{k}+3k^2)\alpha^3_1/9+
2\dot{k}\beta^2_1/9+2\beta^1_1/3$.
The action of ${\rm G}$ on the left of ${\rm Z}$ is Hamiltonian and co-isotropic \cite{GM}. Using the pairing induced by the Killing form, the {\it contact momentum map} \cite{Ba,OR} of the action is given by
$({\bf B},k,\dot{k},\ddot{k})\in {\rm Z}\to i{\bf B}\cdot \Lambda(k,\dot{k},\ddot{k})\cdot {\bf B}^{-1}$. Hence, (\ref{cl}) is a consequence of the N\"other conservation theorem, i.e. that the momentum map is constant along the integral curves of $\mathfrak{X}$. The manifold ${\rm Z}$ and the contact form $\chi$ are build via the Griffith's approach to the calculus of variations \cite{Gr,GM}. Since the action is co-isotropic, then $\mathfrak{X}$ is collective completely integrable and, a fortiori, Liouville-integrable \cite{M1}. Hence, in principle, its integral curves can be found by quadratures linearizing the restriction of $\mathfrak{X}$ on the fibers of the momentum map.}\end{remark}

\section{Closed critical curves}\label{4}

\subsection{Closed critical curves with constant curvature}\label{4S1}
For every $q>1$ we put
\begin{equation}\label{rq}\begin{cases}{\mathtt r}(q)=\sqrt{2+4q-4\sqrt{q(1+q)}},\\
{\mathtt c}(q)= \dfrac{3(16+56{\mathtt r}(q)^4+{\mathtt r}(q)^8)}{2\left(2(-64+528{\mathtt r}(q)^4+132{\mathtt r}(q)^8-{\mathtt r}(q)^{12})\right)^{2/3}}.
\end{cases}\end{equation}
Then, ${\mathtt c}:(1,+\infty)\to (\frac{3}{2\sqrt[3]{4}},+\infty)$ is a smooth diffeomorphism.

\begin{thm}\label{S3.6.P1}{A generic Legendrian curve with constant cr curvature $\kappa=c$ is closed if and only if $c={\mathtt c}(q)$, where $q=m/n>1$ is rational number. Using the Heisenberg picture, such a curve is cr-equivalent to the solenoidal torus knot (see Figure \ref{FIG2}) of type $(-m,n)$ parameterized by
$$\widetilde{\gamma}_q : t\in \R\to {\rm R}_{Oz}(mt/n)\eta_{{\mathtt r}(q)}(t),$$
where ${\rm R}_{Oz}(\theta)$ is the rotation of an angle $\theta$ around the $Oz$-axis and $\eta_{{\mathtt r}(q)}$ is the parameterization of the elliptical profile of the standard Heisenberg Cyclide with parameter ${\mathtt r}(q)$, defined as in (\ref{elp}).
}\end{thm}

\begin{figure}[h]
\begin{center}
\includegraphics[height=6.2cm,width=6.2cm]{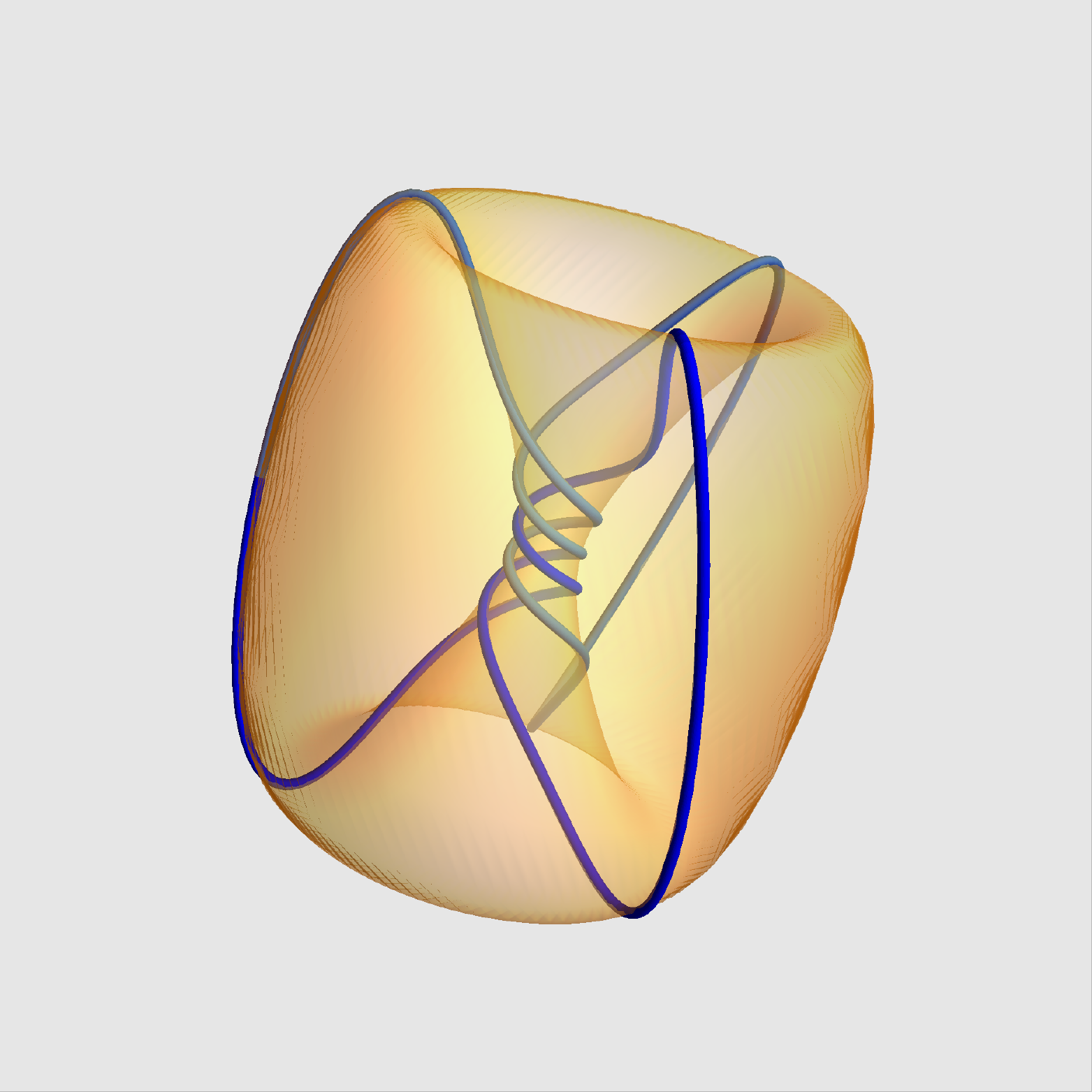}
\includegraphics[height=6.2cm,width=6.2cm]{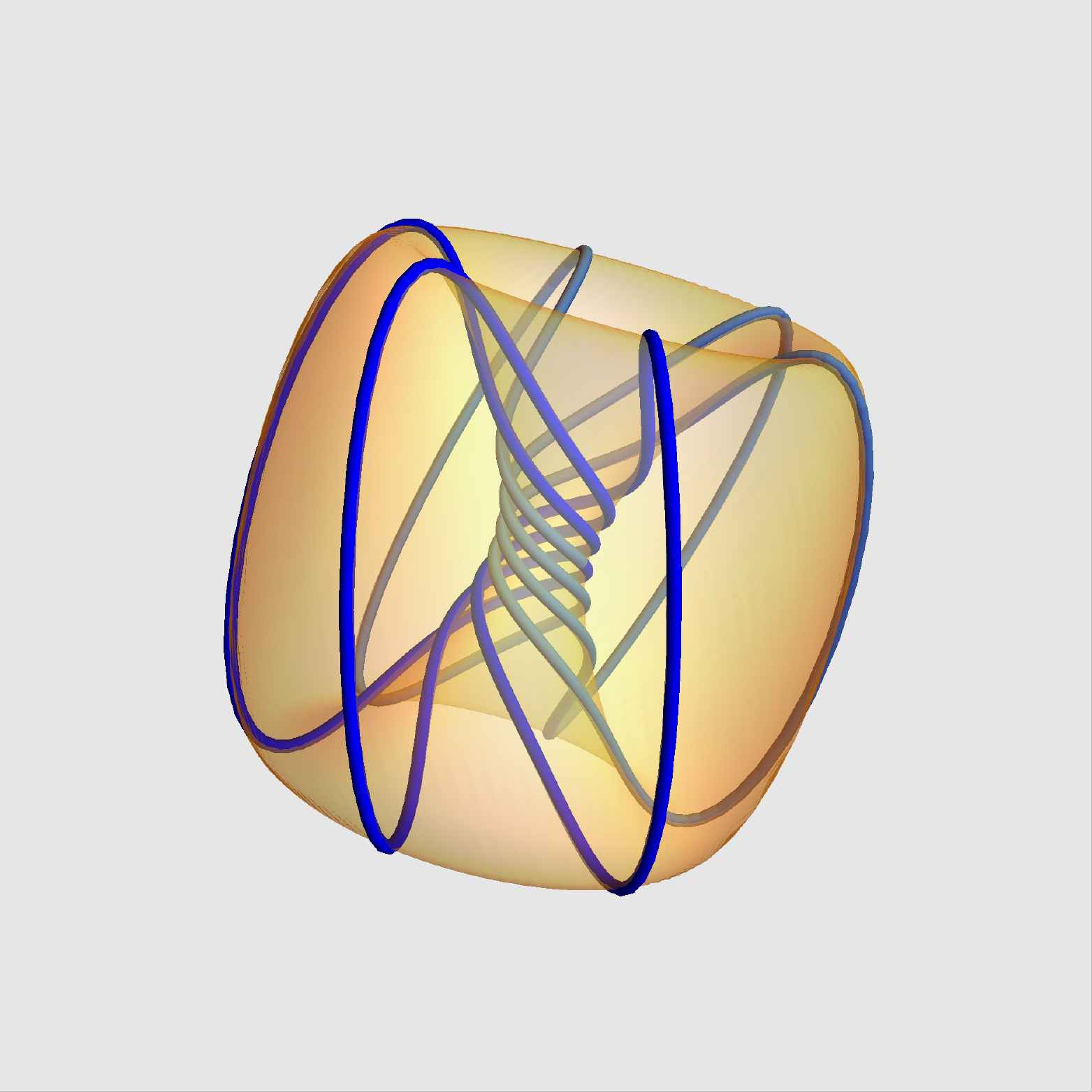}
\caption{\small{Generic Legendrian curves with constant curvature $c(5/3)\approx 1.69321$ (left) and $c(7/6)\approx 3.63111$ (right). }}\label{FIG2}
\end{center}
\end{figure}

\begin{proof}{A Legendrian curve with constant pseudo-conformal curvature $c$ is congruent to $\gamma_{c}:s\in \R\to {\rm Exp}(s{\mathcal K}_{c})\cdot {\rm P}_0$, where ${\mathcal K}_{c} = ({\bf E}^1_2+{\bf E}^3_1+i{\bf E}^2_3)-c({\bf E}^2_1-i{\bf E}^3_2)$. Thus, its trajectory can be closed if and only if $\mathcal{K}(c)$ has three purely imaginary roots, that is if and only if the discriminant $\Delta_{{\mathcal K}_{c}}$ of the characteristic polynomial of $\mathcal{K}_c$ is negative. It is an easy matter to check that $\Delta_{{\mathcal K}_{c}}<0$ if and only if $c>3/2\sqrt[3]{4}$, ie. if and only if $c=c(q)$, for a unique $q\in \R$, $q>1$. With an elementary calculation we see that
$\widetilde{\gamma}_q$ is a Legendrian curve of $\R^3$ and that the Fubini's densities of $\gamma_{q}=p_{h}^{-1}\circ \widetilde{\gamma}_q$ are given by
$$a= -\frac{r^{12}-132r^8-528r^4+64}{6912r^6},\quad b=\frac{r^8+56r^4+16}{384r^4},$$
where $r={\mathtt r}(q)$. Hence in view of (\ref{rq}), $0<r< 2-\sqrt{2}$. This implies $\kappa=b/\sqrt[3]{a^2}=c(q)$. To conclude the proof it suffices to note that two generic Legendrian curves with the same cr-curvature are equivalent each other.
}\end{proof}

\begin{remark}\label{S3.6.R1}{We briefly comment on the topological structure of the torus knots constructed in the Theorem above. It is known that the contact isotopy class of a Legendrian torus knot is uniquely determined by the tours knot type, by the Maslov index and the Bennequin-Thurston invariant \cite{GE}. In addition, if the torus knot is negative, of type $(m,-n)$, with $m>n$, then its Bennequin-Thurston invariant $\mathfrak{tb}$ is less or equal than $-mn$. If $\mathfrak{tb}=-mn$ then its Maslov index is
in the range $\{\pm(m-n-2nk) : k\in \Z, 0\le k\le (m-n)/n\}$. It can be shown that the Maslov index and the Thruston-Bennequin invariant of $\widetilde{\gamma}_q$, $q=m/n>1$, are $m-n$ and $-mn$ respectively. This can be verified with elementary techniques, although the proof is non trivial from a computational viewpoint. Thus, each isotopy class of a negative torus knot with maximal Maslov index and maximal Thurston-Bennequin invariant can be represented by a Legendrian curve with constant cr-curvture $c(q)$.}\end{remark}

\subsection{Critical curves with non-constant periodic curvature}\label{4S.2} Now we focus on  cr-${\it strings}$, ie generic Legendrian curves parameterized by the natural parameter, with non constant periodic cr-curvature and zero stress tensor. 
The equation of motion $\kappa'''+2\kappa\kappa'=0$ implies the existence of two constants $m_2,m_3\in \R$ such that
\begin{equation}\label{4S.2.F1}\begin{cases}&{\kappa}''+4{\kappa}^{2}-\frac{3}{8}m_2=0,\\
&({\kappa}')^2+\frac{8}{3}{\kappa}^3-\frac{3}{4}m_2 {\kappa}+9(1+\frac{m_3}{8}) =0.\end{cases}
\end{equation}
Periodic solutions of (\ref{4S.2.F1}) do exist if and only if 
$m_2^3-54(m_3+8)^2>0$. 
Let $\ell>0$ and $m\in (0,1)$ be  defined by
\begin{equation}\label{4S.2.F2}\begin{cases}m_2=\frac{8}{3}(1-m+m^2)\ell^4,\\ 
m_3=\frac{8}{27}\left(\ell^6(2m^3-3m^2-3m+2)-27\right).\end{cases}
\end{equation}
Modulo a possible translation of the independent variable, the periodic solutions of $(\ref{4S.2.F1})$ can be written as
\begin{equation}\label{4S.2.F3}
\kappa_{m,\ell}(s)=\dfrac{3}{2}\ell^2\left(\dfrac{m+1}{3}-m\ {\rm sn}^2(\ell s,m)\right),
\end{equation}
where ${\rm sn}(-,m)$ is the Jacobi's elliptic sine with parameter\footnote{The parameter is the square of the modulus.} $m\in (0,1)$. Note that $\kappa_{m,\ell}$ is an even periodic function with least period\footnote{${\rm K}$ is the complete elliptic integral of the first kind.} $\omega_{m,\ell}=2{\rm K}(m)/\ell$. If $\gamma$ is a string whit cr-curvature $\kappa_{m,\ell}$, then $(m,\ell)$ are said the {\it characters} of $\gamma$. Let ${\mathfrak m}$ be the momentum of a string with characters $(m,\ell)$. The discriminant ${\mathtt p}(m,\ell)$ and the spectrum  $\{\lambda_j\}_{j=1,2,3}$ of ${\mathfrak m}$ are given by
\begin{equation}\label{dfnL}\begin{cases}
{\mathtt p}(m,\ell)&=64\left(m^2(m-1)^2\ell^{12}+2 (m-2) (1 + m) (2 m-1) \ell^6-27\right),\\
\lambda_1(m,\ell)&=-\lambda_2(m,\ell)-\lambda_3(m,\ell),\\
\lambda_2(m,\ell)&=-\frac{4}{3}\sqrt{1+m(m-1)}\ \ell^2\sin\left(\frac{\arcsin(\widetilde{{\mathtt p}}(m,\ell))}{3}\right),\\
\lambda_3(m,\ell)&=\frac{2}{3}\sqrt{1+m(m-1)}\ \ell^2\sin\left(\frac{\arcsin(\widetilde{{\mathtt p}}(m,\ell))}{3}\right)+\\
 &\quad+\frac{2}{\sqrt{3}}\sqrt{1+m(m-1)}\ \ell^2\cos\left(\frac{\arcsin(\widetilde{{\mathtt p}}(m,\ell))}{3}\right),\\
\widetilde{{\mathtt p}}(m,\ell)&=\frac{(m-2) (1 + m) (2 m-1) \ell^6-27}{2 \left(1 + (m-1) m\right)^{3/2} \ell^6}.
\end{cases}
\end{equation}
The eigenvalues are sorted as follows:
$$\begin{cases}
\lambda_1(m,\ell)<\lambda_2(m,\ell)<\lambda_3(m,\ell),\hspace{113pt} {\rm if}\hspace{3pt} {\mathtt p}(m,\ell)>0,\\
\lambda_1(m,\ell)\in \R,\hspace{2pt}\overline{\lambda_2(m,\ell)} = \lambda_3(m,\ell)\in \C,\hspace{2pt}
{\rm Im}(\lambda_3(m,\ell))>0,\hspace{2pt} {\rm if}\hspace{2pt} {\mathtt p}(m,\ell)<0,\\
\lambda_1(m,\ell)=-2\lambda_2(m,\ell),\hspace{2pt}\lambda_2(m,\ell)=\lambda_3(m,\ell)>0,\hspace{41pt} {\rm if}\hspace{3pt}{\mathtt p}(m,\ell)=0.
\end{cases}
$$
For every $(m,\ell)\in (0,1)\times \R^{+}$ we put
\begin{equation}\label{4S.2.F13.tris}
\Theta_j(m,\ell)=\frac{6 \Pi\left(\frac{6m\ell^2}{2(1+m)\ell^2+3\lambda_j(m,\ell)},m \right)}{\pi\ell \left(2 (1 + m) \ell^2 + 3 \lambda_j(m,\ell)\right)},\quad j=1,2,3,
\end{equation}
where $\Pi(n,m)$ is the complete integral of the third kind.

\begin{thm}\label{closureconditions}{A cr-string with characters $(m,\ell)$ is closed if and only if ${\mathtt p}(m,\ell)>0$ and $\Theta_2(m,\ell),\Theta_3(m,\ell)\in {\mathbb Q}$.}\end{thm}

\begin{proof}{First we analyze eigenvectors and generalized eigenvectors of the momentum $\mathfrak{m}$ of a string with characters $(m,\ell)$. Denote by $\Lambda$ the map defined as in (\ref{Lambda}), with $\kappa=\kappa_{m,\ell}$. Let $\lambda$ be an eigenvalue of ${\mathfrak m}$. Then, $\lambda$ is an eigenvalue of $\Lambda|_s$, for every $s\in \R$. From (\ref{Lambda}) it follows that
\begin{equation}\label{4S.2.F6}{\rm U}_{\lambda}(s) =\hspace{1pt} ^t\left( \frac{i}{2}\big(4\kappa|_s+3\lambda\big)\big(\frac{2}{3}\kappa|_s-\lambda\big),6\big(1+\frac{i}{3}\kappa'|_s\big),
4\kappa|_s+3\lambda\right),
\end{equation}
generates the $\lambda$-eigenspace ${\mathbb L}_{\lambda}|_s$ of $\Lambda|_s$. Hence, ${\mathbb L}_{\lambda}|_s$ is $1$-dimensional. If ${\bf B}$ is a  Wilczynski frame and if we put
\begin{equation}\label{4S.2.F7} {\bf U}_{\lambda}={\bf B}\cdot {\rm U}_{\lambda},\end{equation}
then, ${\bf U}_{\lambda}(s)$ belongs to the $1$-dimensional $\lambda$-eigenspace ${\mathbb M}_{\lambda}$ of 
${\mathfrak m}$, for every $s\in \R$. Hence, ${\bf U}_{\lambda}'=\varrho {\bf U}_{\lambda}$, where $\varrho$ is a complex-valued function. From this we deduce that
\begin{equation}\label{4S.2.F7Bis}{\rm U}_{\lambda}' +({\mathcal B}-\varrho {\rm Id}_{3\times 3}){\rm U}_{\lambda} = 0,\end{equation}
where ${\mathcal B}$ is as in (\ref{S2.2.F4}), with $a=1$ and $b=\kappa$.
The third component of the left hand side of (\ref{4S.2.F7Bis}) is equal to $\frac{1}{3}(2\kappa'-\varrho(4\kappa+3\lambda)+6i)$. Since $\kappa$ is real-valued, the equation $2\kappa'-\varrho(4\kappa+3\lambda)+6i=0$ implies
\begin{equation}\label{4S.2.F8} 4\kappa(s)+3\lambda\neq 0,\hspace{3pt} \forall s\in \R,\hspace{3pt} {\rm and}\hspace{3pt}  \varrho =\frac{2\kappa'+6i}{4\kappa+3\lambda}.\end{equation}
Let $\sqrt{4\kappa+3\lambda}$ be a continuous determination of the square root of  $4\kappa+3\lambda$ and ${\rm V}_{\lambda},{\rm W}_{\lambda}:\R\to \C^{2,1}$ be defined by
\begin{equation}\label{4S.2.F9}
{\rm V}_{\lambda} = \dfrac{1}{\sqrt{4\kappa+3\lambda}}{\rm U}_{\lambda},\quad 
{\rm W}_{\lambda}=e^{-6i \int \frac{ds}{4\kappa+3\lambda}}{\rm V}_{\lambda}.
\end{equation}
Then
\begin{equation}\label{4S.2.F10}\mathfrak{W}_{\lambda}={\bf B}\cdot {\rm W}_{\lambda}=
e^{-6i \int_0^s \frac{dt}{4\kappa+3\lambda}} {\bf B}\cdot {\rm V}_{\lambda}
\end{equation}
is a {\it constant} eigenvector of ${\mathfrak m}$. If ${\mathtt p}(m,\ell)=0$, the characteristic polynomial of ${\mathfrak m}$ has a double positive real root $\lambda$. We put
\begin{equation}\label{4S.2.F11}\begin{cases}
\widetilde{{\rm V}}_{\lambda}=\sqrt{4\kappa+
3\lambda}\hspace{5 pt}^t\left(-\frac{2(\kappa+3\lambda)^2}{3(4\kappa+3\lambda)},
\frac{(2\kappa-3\lambda)(\kappa+3\lambda)}{3(\kappa'+3i)},\frac{i(2\kappa-3\lambda)}{4\kappa+3\lambda}\right),\\
\widetilde{{\rm W}}_{\lambda}=e^{-6i \int_0^s \frac{dt}{4\kappa+3\lambda}}\hspace{5pt}^t\left(\widetilde{{\rm V}}_{\lambda}+
162\lambda\int_0^s\frac{dt}{(4\kappa+3\lambda)^2}{\rm V}_{\lambda}\right),\end{cases}
\end{equation}
where ${\rm V}_{\lambda}$ is as in (\ref{4S.2.F9}). Then,  $\widetilde{{\rm W}}_{\lambda}(s)$ is a rank two generalized eigenvector\footnote{that is $(\Lambda(s)-\lambda{\rm Id})\widetilde{{\rm W}}_{\lambda}(s)\neq 0$ and $(\Lambda(s)-\lambda{\rm Id})^2\widetilde{{\rm W}}_{\lambda}(s)= 0$.} of $\Lambda(s)$, for each $s\in \R$. Since 
$\widetilde{{\rm W}}_{\lambda}$ satisfies $\widetilde{{\rm W}}_{\lambda}'+\mathcal{B}\cdot \widetilde{{\rm W}}_{\lambda}=0$, we infer that
\begin{equation}\label{4S.2.F12}\widetilde{\mathfrak{W}}_{\lambda}={\bf B}\cdot \widetilde{{\rm W}}_{\lambda}=
e^{-6i \int_0^s \frac{dt}{4\kappa+3\lambda}} {\bf B}\cdot \left(\widetilde{{\rm V}}_{\lambda}+162\lambda\int_0^s\frac{dt}{(4\kappa+3\lambda)^2}{\rm V}_{\lambda}\right)
\end{equation}
is a {\it constant} rank-two generalized eigenvector of ${\mathfrak m}$. From this we can deduce the following conclusion: if $\mathtt{p}(m,\ell) \neq 0$ then there exist ${\bf C}\in {\rm GL}(3,\C)$, a periodic map $\bf{P}:\R\to \rm{GL}(3,\C)$ with least period $\omega_{m,\ell}$ such that
\begin{equation}\label{4S.2.F13}\begin{split}&{\bf C}={\bf B}(s)\cdot {\bf P}(s)\cdot {\bf D}(s),\\ 
&{\bf D}(s)=e^{-6i \int_0^s \frac{dt}{4\kappa+3\lambda_1(m,\ell)}}{\bf E}^1_1+e^{-6i \int_0^s \frac{dt}{4\kappa+3\lambda_2(m,\ell)}}{\bf E}^2_2+e^{-6i \int_0^s \frac{dt}{4\kappa+3\lambda_3(m,\ell)}}{\bf E}^3_3.\end{split}\end{equation}
Similarly, if $\mathtt{p}(m,\ell) = 0$, there exist ${\bf C}\in {\rm GL}(3,\C)$ and a periodic map $\bf{P}:\R\to \rm{GL}(3,\C)$ with least period $\omega_{m,\ell}$ such that
\begin{equation}\label{4S.2.F13.bis}\begin{split}
& {\bf C}={\bf B}(s)\cdot {\bf P}(s)\cdot {\bf D}(s)\cdot {\bf T}(s),\\
&{\bf T}(s)={\rm Id}_{3\times 3}+\left(\int_0^s\frac{162\lambda_2(m,\ell)}{(4\kappa+3\lambda_2(m,\ell))^2}dt\right){\bf E}^3_2.\end{split}\end{equation}
We now prove that, if $\mathtt{p}(m,\ell) \le 0$, then there are no closed stings with characters $(m,\ell)$. By contradiction, suppose that $\gamma$ is periodic. Then, ${\bf B}$ is a periodic map too. Hence, its Hilbert-Schmidt norm $\|{\bf B}\|_{\rm {HS}}$ is a bounded function. Assume $\mathtt{p}(m,\ell) < 0$. Then,  
$$\lambda_1(m,\ell)=2\lambda_{m,\ell},\quad \lambda_2(m,\ell)=-\lambda_{m,\ell}-i\tau_{m,\ell},\quad 
\lambda_3(m,\ell)=-\lambda_{m,\ell}+i\tau_{m,\ell},
$$
where $\lambda_{m,\ell},\tau_{m,\ell}\in \R$ and $\tau_{m,\ell}>0$. Observing that
$$\int_0^s \frac{dt}{4\kappa+3\lambda_2(m,\ell)}=\int_0^s \frac{(4\kappa-3\lambda_{m,\ell})dt}{(4\kappa-3\lambda_{m,\ell})^2+9\tau_{m,\ell}^2}
+3i\tau_{m,\ell}\int_0^s \frac{dt}{(4\kappa-3\lambda_{m,\ell})^2+9\tau_{m,\ell}^2}
$$
and using (\ref{4S.2.F13}), we get\footnote{in this context, $\|\cdot \|$ and $(\cdot, \cdot)$ are the standard norm and hermitian inner product of $\C^3$.}
\begin{multline}\|{\bf B}\|_{\rm {HS}}^2 \|{\bf P}\|_{\rm {HS}}^2\ge \|{\bf B}\cdot {\bf P}\|_{\rm {HS}}^2=\\=
\|{\rm C}_1\|^2+e^{36\tau_{m,\ell}\int_0^s \frac{dt}{(4\kappa-3\lambda_{m,\ell})^2+9\tau_{m,\ell}^2}} \|{\rm C}_2\|^2+e^{-36\tau_{m,\ell}\int_0^s \frac{dt}{(4\kappa-3\lambda_{m,\ell})^2+9\tau_{m,\ell}^2}} \|{\rm C}_3\|^2,\nonumber
\end{multline}
where ${\rm C}_j$ are the column vectors of ${\bf C}$. This implies that $\|{\bf B}\|_{\rm {HS}}$ is unbounded. Similarly, if $\gamma$ is periodic and $\mathtt{p}(m,\ell) = 0$, then $\lambda_1(m,\ell)=-2\lambda_{m,\ell}$, 
$\lambda_2(m,\ell)=\lambda_3(m,\ell)=\lambda_{m,\ell}$ where $\lambda_{m,\ell}>0$. From (\ref{4S.2.F13.bis}) we obtain
\begin{multline}
\|{\bf B}\|_{\rm {HS}}^2 \|{\bf P}\|_{\rm {HS}}^2\ge \|{\bf B}\cdot {\bf P}\|_{\rm {HS}}^2=\\=
\|{\bf C}\|_{\rm {HS}}^2
+ \int_0^s\frac{162\lambda_{m,\ell}\ dt}{(4\kappa+3\lambda_{m,\ell} )^2}\left( \int_0^s\frac{162\lambda_{m,\ell}\ dt}{(4\kappa+3\lambda_{m,\ell} )^2}\|{\rm C}_2\|^2-2{\rm Re}(({\rm C}_2,{\rm C}_3)) \right).\nonumber
\end{multline}
Hence, $\|{\bf B}\|_{\rm {HS}}$ is unbounded. So even in this case there are no closed strings.

\noindent Suppose $\mathtt{p}(m,\ell) > 0$. Keeping in mind (\ref{4S.2.F13}) and taking into account that ${\bf P}$ is periodic with least period $\omega_{m,\ell}$, then ${\bf B}$ is periodic 
if and only if
\begin{equation}\label{periodic}
\frac{3}{\pi}\int_0^{\omega_{m,\ell}} \frac{ds}{4\kappa_{m,\ell}(s) + 3\lambda_j (m,\ell)}\in {\mathbb Q},
\end{equation}
$j=1,2,3$. From (\ref{4S.2.F3}) it follows that the above integral is equal to $\Theta_j(m,\ell)$. This proves that a string with $\mathtt{p}(m,\ell) > 0$ is closed if and only if $\Theta_j(m,\ell)\in {\mathbb Q}$, $j=1,2,3$. From (\ref{4S.2.F13}) we have
\begin{equation}\label{4S.2.F13.tetra}{\rm det}({\bf C}^{-1}\cdot {\bf P})=e^{6i\sum_{j=1}^{3}\int_0^s \frac{dt}{4\kappa+3\lambda_j(m,\ell)}}.\end{equation}
Since ${\rm det}({\bf P})$ is a periodic function with period $\omega_{m,\ell}$, then (\ref{4S.2.F13.tetra}) 
implies that $\Theta_1(m,\ell)+\Theta_2(m,\ell)+\Theta_3(m,\ell)\cong_{\Z} 0$. Therefore, if $\Theta_2(m,\ell),\Theta_3(m,\ell)\in {\mathbb Q}$, then $\Theta_1(m,\ell)\in {\mathbb Q}$. This concludes the proof of the Theorem.
}\end{proof}

\section{The period map}\label{5}
\noindent Denote by ${\mathcal D}$ (see Figure \ref{FIG3}) the planar domain $\{(m,\ell)\in (0,1)\times \R^+/ \ell>{\mathfrak l}(m)\}$, where 
\begin{equation}\label{5S.1.F1Bis}\mathfrak{l}(m)=\sqrt[6]{\frac{27}{(m -2) (1 + m) (2 m-1) + 2 \left(1 + (m -1) m\right)^{3/2}}}.
\end{equation}

\begin{defn}\label{periodM}{Let $\Theta_2$ and $\Theta_3$ be the real-analytic functions defined in (\ref{4S.2.F13.tris}). We call $\Theta : (m,\ell)\in {\mathcal D}\to (\Theta_2(m,\ell),\Theta_3(m,\ell))\in \R^2$  the {\it period map} of the total strain functional. By construction, $\Theta$ is real-analytic and non-constant.}\end{defn}

\begin{figure}[h]
\begin{center}
\includegraphics[height=6.2cm,width=6.2cm]{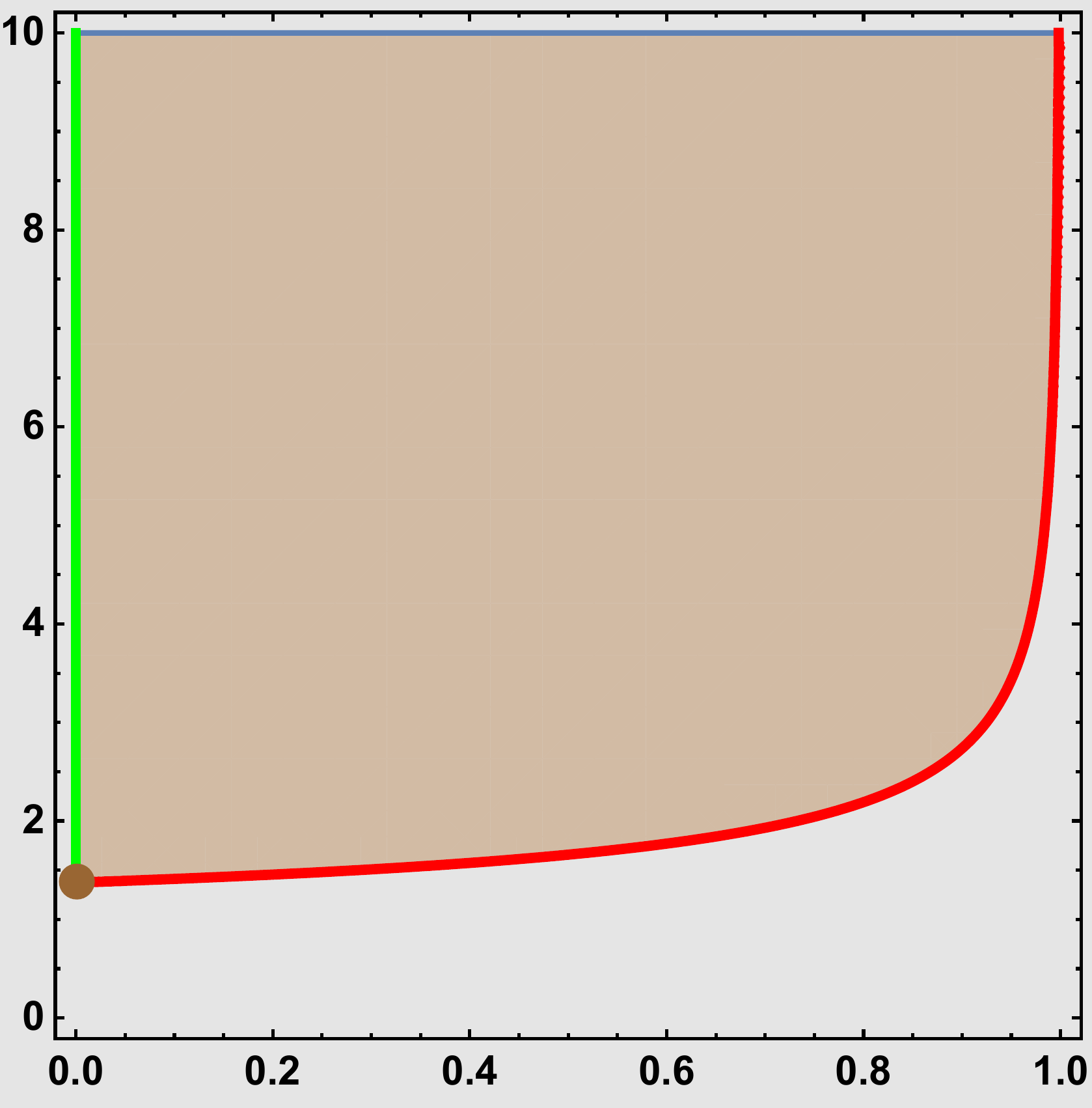}
\includegraphics[height=6.2cm,width=6.2cm]{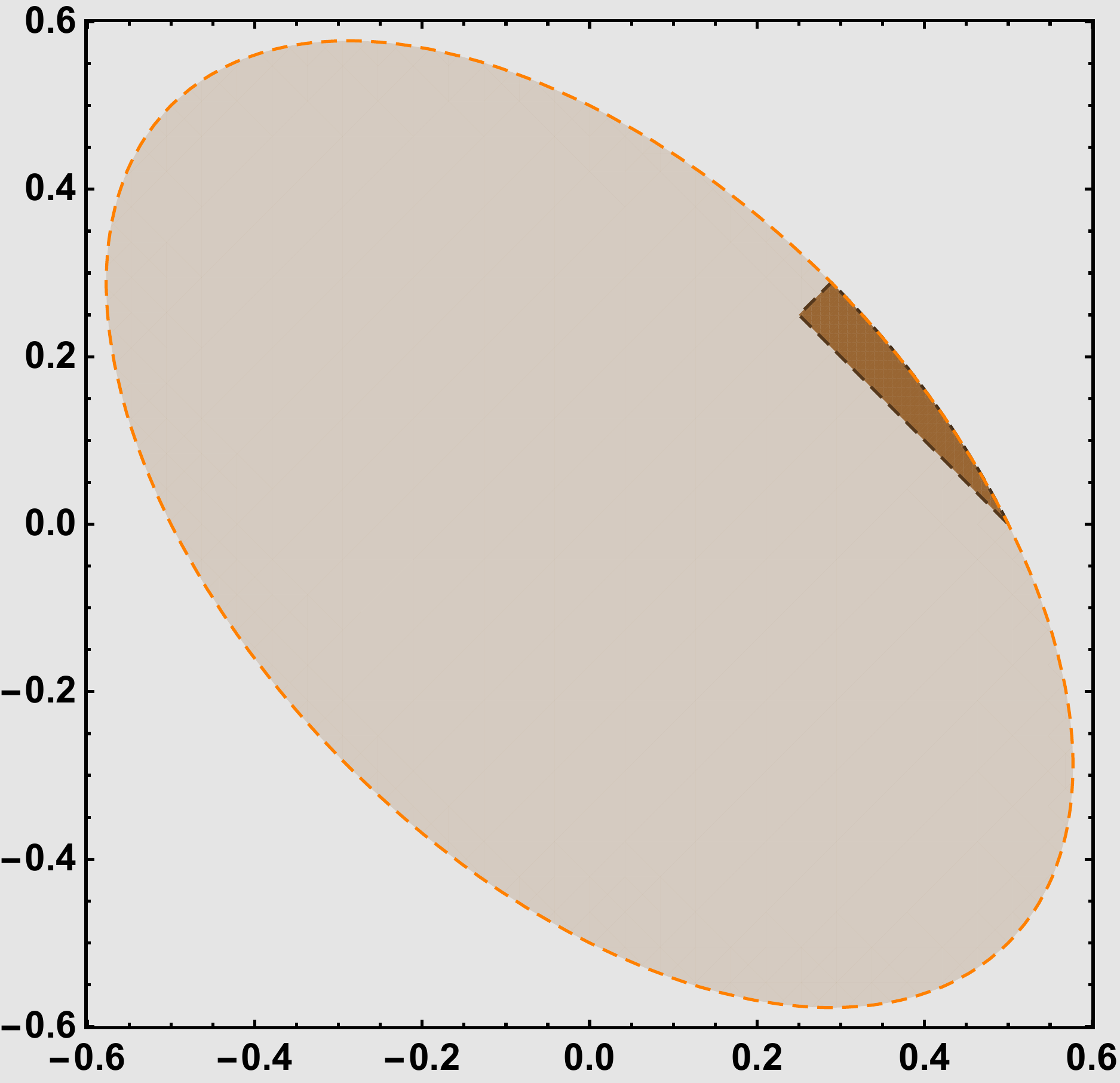}
\caption{\small{The domains ${\mathcal D}$ (left) and ${\mathcal M}$ (rigth), the dark portion of the ellipse.}}\label{FIG3}
\end{center}
\end{figure}

\begin{defn}\label{DMN}{The {\it monodromic domain} is the planar domain (see Figure \ref{FIG3}) defined by
$${\mathcal M}=\{(x,y)\in \R^2 : x^2+xy+y^2<1/4,\hspace{2pt} x-y>0,\hspace{2pt} x+y>1/2 \}.$$}\end{defn}

\begin{remark}\label{arc}{The boundary of the monodromic domain (see Figure \ref{FIG4}) consists of three vertices
$P_1(1/4,1/4)$, $P_2(1/2\sqrt{3},1/2\sqrt{3})$ and $P_3(1/2,0)$, the segments $\sigma_{1,2}=[{\rm P}_1,{\rm P}_2]$ and $\sigma_{1,3}=[{\rm P}_1,{\rm P}_3]$ and the arc $\sigma_{2,3}$ of the ellipse $x^2+xy+y^2=1/4$ connecting ${\rm P}_2$ and ${\rm P}_3$ parameterized by $t\in [0,1]\to (x(t),y(t))$, where
\begin{equation}\label{DMN.F1}\begin{cases}
x(t)&=\frac{\sqrt{1 - t}}{\sqrt{3}\left(1 + 2\sin\left(\frac{\arcsin(1 - 2 t)}{3}\right)\right)},\\
y(t)&=\frac{\sqrt{1 - t}}{3\cos\left(\frac{\arcsin(1 - 2 t)}{3}\right)+
\sqrt{3}\left(1 -\sin\left(\frac{\arcsin(1 - 2 t)}{3}\right)\right)}.\end{cases} 
\end{equation}}\end{remark}

\noindent The content of Theorem \ref{closureconditions} can be rephrased as follows: 

\begin{cor}{There is a one-to-one correspondence between the equivalence classes of closed strings and the set 
${\mathcal D}_{*}$ of all $(m,\ell)\in {\mathcal D}$ such that $\Theta(m,\ell)\in {\mathbb Q}\times {\mathbb Q}$.}\end{cor}

\noindent Thus, Theorem {\rm B} can be reformulated as follows.

\begin{thm}{The period map is a diffeomorphism of ${\mathcal D}$ onto ${\mathcal M}$.}\end{thm}

\noindent Since the proof is rather technical, we split the discussion into three parts (Propositions \ref{PI}, \ref{PII} and \ref{PIII}). In the first one we prove that $\Theta$ is a local diffeomorphism. In the second part we prove the injectivity of $\Theta$ and in the third part we show that $\Theta({\mathcal D})= {\mathcal M}$.

\begin{prop}\label{PI}{The determinant of the Jacobian matrix $J(\Theta)$ of the period map is strictly positive.}\end{prop}
\begin{proof}{
Let $\Phi^a_{j,b}:\mathcal{D}\to \R$, $j=2,3$, $a,b=1,2$ be the real-analytic functions
\[
\begin{cases}
\Phi^1_{j,1}(m,\ell)&=\frac{9\lambda_j(m, \ell) + 6 (m - 1) \ell^2}{\pi m \ell \left(4 ((m-1) m+1) \ell^4-9\lambda_j(m,\ell)^2\right)},\\
\Phi^1_{j,2}(m,\ell)&=\frac{9\lambda_j(m,\ell)+6 (2 m-1) \ell^2}{\pi(m-1) m \ell \left(4 ((m-1) m+1) \ell^4-9\lambda_j(m,\ell)^2\right)},\\
\Phi^2_{j,1}(m,\ell)&=\frac{18\lambda_j(m,\ell)+12 (m-2) \ell^2}{\pi \ell^2 (4 ((m-1) m+1)\ell^4-9 \lambda_j(m,\ell)^2)},\\
\Phi^2_{j,2}(m,\ell)&=\frac{36}{\pi (4 ((m-1) m+1)\ell^4-9 \lambda_j(m,\ell)^2)}.
\end{cases}\]
We prove that
\begin{equation}\label{PDT}
\begin{cases}
\partial_m \Theta_j\big|_{(m,\ell)}&=\Phi^1_{j,1}(m,\ell){\rm K}(m)+\Phi^1_{j,2}(m,\ell){\rm E}(m),\\
\partial_{\ell} \Theta_j\big|_{(m,\ell)}&=\Phi^2_{j,1}(m,\ell){\rm K}(m)+\Phi^2_{j,2}(m,\ell){\rm E}(m),
\end{cases}
\end{equation}
where ${\rm K}$ and ${\rm E}$ are the complete elliptic integrals of the first and second kind respectively. In fact, $\lambda_2$ and $\lambda_3$ are solutions of the overdetermined system of PDE
\begin{equation}\label{ODS}
\begin{cases}
\partial_m f\big|_{(m,\ell)}&=\frac{4}{3}\left(\frac{2(1 - 2 (m-1) m) \ell^6 + 3 (1 - 2 m) \ell^4 f(m,\ell)}
{4 (1 + (m -1) m) \ell^4 - 9 f(m,\ell)^2}\right),\\
\partial_{\ell} f\big|_{(m,\ell)}&=-\frac{16}{3}\left(\frac{(m-2) (1 + m) (2 m-1) \ell^5 + 3 (1 + (m -1) m) 
\ell^3 f(m, \ell)}{4 (1 + (m-1) m) \ell^4 - 9 f(m, \ell)^2}\right).
\end{cases}
\end{equation}
and the partial derivatives of the complete integral of the third kind are given by
\begin{equation}\label{PDP}
\begin{cases}
\partial_n \Pi\big|_{(n,m)}=&\frac{n {\rm E}(m) + (m - n){\rm K}(m) + (n^2-m) \Pi(n,m)}{2 (m - n) (n-1) n},\\
\partial_m \Pi\big|_{(n,m)}=&\frac{{\rm E}(m)}{2 (m-1) (n-m)}+\frac{\Pi(n,m)}{2 (n-m)}.
\end{cases}
\end{equation}
Then, (\ref{PDT}) follows immediately from (\ref{4S.2.F13.tris}), (\ref{ODS}) and (\ref{PDP}). Using (\ref{PDT}) we obtain
\[\begin{cases}{\rm det}(J(\Theta))\big|_{(m,\ell)}=\frac{108}{\varrho(m,\ell)}
\left(3{\rm E}(m)^2+(1-m){\rm K}(m)^2+2(m-2){\rm E}(m){\rm K}(m)\right),\\
\varrho(m,\ell)=\frac{\pi^2\ell m(m-1)}{\lambda_2(m,\ell)-\lambda_3(m,\ell)}\prod_{j=2,3} \left(4 (1 - m + m^2) \ell^4 - 9\lambda_j(m,\ell)^2\right).\end{cases}\]
Since $\varrho(m,\ell)<0$ and $3{\rm E}(m)^2+(1-m){\rm K}(m)^2+2(m-2){\rm E}(m){\rm K}(m)<0$, for every $(m,\ell)\in {\mathcal D}$, then ${\rm det}(J(\Theta))$ is strictly positive on ${\mathcal D}$, as claimed.
}\end{proof}

\begin{prop}\label{PII}{The period map is injective.}\end{prop}
\begin{proof}{The proof is organized in six steps, a comment and a conclusion. 

\noindent {\bf Step I}. In the first step we prove that
\begin{equation}\label{IN}
\partial_m \Theta_2<0,\quad \partial_{\ell} \Theta_2>0.
\end{equation}
Since $4 (1+(m-1) m)\ell^4 -9\lambda_2(m,\ell)^2 >0$, then (\ref{PDT}) implies that $\partial_m \Theta_2<0$ if and only if
\begin{equation}\label{f1}\left(6 (2 m-1) \ell^2 + 9 \lambda_2(m,\ell)\right){\rm E}(m)+\left(6 (m-1)^2 \ell^2 + 
      9 (m-1)\lambda_2(m,\ell)\right){\rm K}(m)>0
\end{equation}
and that $\partial_{\ell} \Theta_2>0$ if and only if
\begin{equation}\label{f2}36 \ell^2 {\rm E}(m) + \left(12 (m-2) \ell^2 + 18\lambda_2(m,\ell)\right){\rm K}(m)>0.\end{equation}
We prove (\ref{f1}): let $g(m,\ell)$ be the the left hand side of (\ref{f1}). We claim that,
for every $m\in (0,1)$, the function  $f_{m} : \ell\in ({\mathfrak l}(m),+\infty)\to g(m,\ell)/ 6\ell^2$ is strictly decreasing.
From (\ref{ODS}) and keeping in mind that
\begin{equation}\label{dEK}{\rm K}'|_m=\frac{{\rm E}(m)-(1-m){\rm K}(m)}{2 (1 - m) m},\quad
{\rm E}'|_m=\frac{{\rm E}(m)-{\rm K}(m)}{2m}\end{equation}
we obtain
$$f_{m}'|_{\ell}=-96\sqrt{3}\frac{\sqrt{1 - m + m^2}\cos\left(\frac{1}{3}\arcsin\left(\widetilde{{\mathtt p}}(m,\ell)\right)\right)\left(\rm{E}(m)+(m-1)\rm{K}(m)\right)}{\ell\sqrt{{\mathtt p}(m,\ell)}},
$$
where ${\mathtt p}$ and $\widetilde{{\mathtt p}}$ are as in (\ref{dfnL}). Since
\begin{equation}\label{dEK1}\rm{E}(m) + (m-1) \rm{K}(m)>0,\quad \forall m\in (0,1)\end{equation}
and 
\begin{equation}\label{f3}-1<\widetilde{{\mathtt p}}(m,\ell)<1,\quad \forall (m,\ell)\in {\mathcal D},\end{equation}
we infer that $f_{m}'|_{\ell}<0$, for every $\ell\in ({\mathfrak l}(m),+\infty)$, as claimed. Then,
$$\frac{g(m,\ell)}{6\ell^2}=f_{m}(\ell) > \lim \limits_{\ell\rightarrow
+\infty} f_{m}(\ell) =2(2 m-1) \rm{E}(m) + (m-1) (3 m-2) \rm{K}(m)>0.$$
Now we prove (\ref{f2}). The reasoning is similar to the previous one. Let $\widetilde{g}(m,\ell)$ be the left hand side of (\ref{f2}). Then,
${\widetilde g}(m,\ell)/12\ell^2=3 \rm{E}(m) +\rm{K}(m)r(m,\ell)$ where 
$r(m,\ell)=(m-2 - 2 \sqrt{1 - m + m^2}\sin(\arcsin(\widetilde{{\mathtt p}}(m,\ell))/3)$.
We claim that, for every $m\in (0,1)$, the function ${\widetilde f}_{m} : \ell \to r(m,\ell)$ is strictly decreasing. From (\ref{ODS}) and (\ref{f3}) we obtain
$${\widetilde f}_{m}'|_{\ell}= - 96\sqrt{3}\frac{\sqrt{1 - m + m^2}\cos\left(\frac{1}{3}\arcsin\left(\widetilde{{\mathtt p}}(m,\ell)\right)\right)}{\ell\sqrt{{\mathtt p}(m,\ell)}}<0.$$
Then,
$$\frac{{\widetilde g}(m,\ell)}{12\ell^2}=3 {\rm E}(m) +{\rm K}(m){\widetilde f}_{m}(\ell)>
3 {\rm E}(m) +{\rm K}(m)\lim \limits_{\ell\rightarrow
+\infty} {\widetilde f}_{m}=3 \rm{E}(m) + 3(m-1)\rm{K}(m)>0.
$$

\noindent {\bf Step II}. We prove that
\begin{equation}\label{IN2}
\partial_m \Theta_3>0,\quad \partial_{\ell} \Theta_3<0.
\end{equation}
Since $4 (1+(m-1) m)\ell^4 -9\lambda_3(m,\ell)^2 <0$, then $\partial_m \Theta_3>0$ if and only if
\begin{equation}\label{f4}\left(6 (2m-1) \ell^2 + 9 \lambda_3(m,\ell)\right){\rm E}(m)+\left(6 (m-1)^2 \ell^2 + 
      9 (m-1)\lambda_3(m,\ell)\right){\rm K}(m)>0
\end{equation}
and  $\partial_{\ell} \Theta_3<0$ if and only if
\begin{equation}\label{f5}36 \ell^2 {\rm E}(m) +\left (12 (-2 + m) \ell^2 + 18\lambda_3(m,\ell)\right){\rm K}(m)>0.\end{equation}
We prove (\ref{f4}): denote by ${\mathtt g}(m,\ell)$ the left hand side of (\ref{f4}). Given $m\in (0,1)$, we put ${\mathtt f}_{m} : \ell\in ({\mathfrak l}(m),+\infty)\to {\mathtt g}(m,\ell)/6\ell^2$.
From (\ref{ODS}) and (\ref{dEK}) we obtain
$${\mathtt f}_m'|_{\ell}=
48\frac{\sqrt{1 - m + m^2}\left({\rm E}(m)+(m-1){\rm K}(m)\right)}{\ell\sqrt{\mathtt{p}(m,\ell)}}{\mathtt h}_m(\ell),
$$
where 
\begin{equation}\label{hm}
{\mathtt h}_m(\ell)=\sqrt{3}\cos\left(\frac{1}{3}\arcsin\Big(\widetilde{{\mathtt p}}(m,\ell)\Big)\right)-3\sin\left(\frac{1}{3}\arcsin\Big(\widetilde{{\mathtt p}}(m,\ell)\Big)\right).
\end{equation}
 Then, (\ref{dEK1}) and (\ref{f3}) imply that ${\mathtt f}_m'>0$. Hence,
\begin{multline}\frac{{\mathtt g}(m,\ell)}{6\ell^2}={\mathtt f}(\ell)>\lim \limits_{\ell\rightarrow
{\mathfrak l}(m)} {\mathtt f} = m{\rm E}(m)+\\+\big(-1 +  m +\sqrt{1 - m + m^2}\big)\big({\rm E}(m)
+(m-1){\rm K}(m)\big)>0.\nonumber
\end{multline}
Now we prove
(\ref{f5}). Let $\widetilde{{\mathtt g}}(m,\ell)$ be the left hand side of (\ref{f5}). Given $m\in (0,1)$, we consider the function 
$\widetilde{{\mathtt f}}_m : \ell\in ({\mathfrak l}(m),+\infty)\to \widetilde{\mathtt{g}}(m,\ell)/6\ell^2$.
Proceeding as above, we get
$$\widetilde{{\mathtt f}}_m'|_{\ell}= \frac{96\sqrt{1 - m + m^2}}{\ell\sqrt{\mathtt{p}(m,\ell)}} {\rm K}(m)\mathtt{h}_m(\ell)>0,$$
where $\mathtt{h}_m(\ell)$ is defined as in (\ref{hm}).
Then,
$$\frac{{\mathtt g}(m,\ell)}{6\ell^2}=\widetilde{{\mathtt f}}(\ell)>\lim \limits_{\ell\rightarrow
{\mathfrak l}(m)} \widetilde{\mathtt{f}} = 6 {\rm E}(m) + 2 (-2 + m + \sqrt{1 - m + m^2}) {\rm K}(m)>0.$$

\noindent {\bf Step III}. 
We show that $\Theta_2(\mathcal{D}) = (1/4,1/2)$. Since $\Theta_2$ is strictly increasing with respect to the second variable, to verify that ${\rm sup}(\Theta_2(\mathcal{D}))= 1/2$ it suffices to show that, for each $m\in (0,1)$, the limit of $\Theta_2(m,\ell)$ as $\ell\to +\infty$ is equal to $1/2$. To this end we observe that
\[
\lim \limits_{n\rightarrow
1^-}\left(\Pi(n,m)\sqrt{1-n}\right)=\frac{\pi}{2\sqrt{1-m}},\quad
\lim \limits_{\ell\rightarrow
+\infty}\frac{6\ell^2}{2\lambda_2(m,\ell)+2(1+m)\ell^2}=1\]
This implies
\begin{equation}\label{PL5.F1}\lim \limits_{\ell\rightarrow
+\infty}\Theta_2(m,\ell)= \lim \limits_{\ell\rightarrow
+\infty} \frac{3}{\sqrt{{\rm F}(m,\ell)}}= \lim \limits_{r\rightarrow
0^+} \frac{3}{\sqrt{{\rm F}(m,1/r)}}
\end{equation}
where
$${\rm F}(m,\ell)=\ell^2 (1 - m) (3\lambda_2(m,\ell) + 
    2 (1 + m) \ell^2) (3\lambda_2(m,\ell) + 2 (1 + m) \ell^2 - 6 m \ell^2)).$$
Note that ${\rm F}(m,1/r)=(1-m)r^{-6}{\rm A}(m,r){\rm B}(m,r)$ where
\[\begin{cases}{\rm A}(m,r)&=2(1+m)-4\sqrt{1 + (-1 + m) m}\sin\left(\alpha(m,r)\right),\\
{\rm B}(m,r)&=-6m+2(1+m)-4\sqrt{1 + (-1 + m) m}\sin\left(\alpha(m,r)\right),\\
\alpha(m,r)&=\frac{1}{3}\arcsin\left(\frac{-27 r^6 + (-2 + m) (1 + m) (-1 + 2 m)}{2 (1 + (-1 + m) m)^{3/2}}\right).
\end{cases}\]
We fix $m\in (0,1)$. Taking the Taylor expansions of ${\rm A}(m,r)$ and ${\rm B}(m,r)$ at $r=0$, we obtain
\[\begin{cases}{\rm A}(m,r)=&2(1+m)-4\sqrt{1 + (-1 + m) m}\sin\left(\alpha(m,0)\right)+\\
&+\frac{4\sqrt{3}\sqrt{1 + (-1 + m) m}}{m(1-m)}\cos\left(\alpha(m,0)\right)r^6+O(r^9),\\
{\rm B}(m,r)=&2(1-2m)-4\sqrt{1 + (-1 + m) m}\sin\left(\alpha(m,0)\right)+\\
&+\frac{4\sqrt{3}\sqrt{1 + (-1 + m) m}}{m(1-m)}\cos\left(\alpha(m,0)\right)r^6+O(r^9).
\end{cases}
\]
Observing that
$$\sin\left(\alpha(m,0)\right)=\frac{1-2m}{2\sqrt{1+(m-1)m}},\quad 
\cos\left(\alpha(m,0)\right)=\frac{\sqrt{3}}{2\sqrt{1+(m-1)m}},
$$
we have
\[{\rm A}(m,r) \sim 6\left(m+\frac{r^6}{m(1-m)} \right),\hspace{2pt}
{\rm B}(m,r)\sim \frac{6 r^6}{m(1-m)},\hspace{3pt} {\rm as}\hspace{3pt} r\to 0.
\]
Then, $\lim \limits_{r\rightarrow
0^+}{\rm F}(m,1/r)=36$. This implies $\lim \limits_{\ell\rightarrow
\infty}\Theta_2(m,\ell)=1/2$, for every $m\in (0,1)$, as claimed. Next we prove that ${\rm inf}(\Theta_2({\mathcal D}))=1/4$. Preliminarily we observe that
\begin{equation}\label{FF0}\vartheta(m):=\lim \limits_{\ell\rightarrow
{\mathfrak l}(m)} \Theta_2(m,\ell)=\frac{3(1-m)m\ \Pi\left(m+1-\sqrt{1-m(1-m)},m\right)}{\pi(m+1+\sqrt{1-m(1-m)}) \phi(m)}
\end{equation}
where $\phi(m)$ is the positive square root of
\[ m\left(m (2\sqrt{(m - 1) m + 1}+3-2 m) -2\sqrt{(m - 1) m + 1} + 3\right)+2(\sqrt{(m-1) m+1}-1).\]
Then, $\vartheta$ is strictly decreasing and $\lim \limits_{m\rightarrow
1^-}\vartheta(m)=1/4$. On the other hand, $\Theta_2(m,\ell)$ is strictly increasing with respect to the variable $\ell$ and
$$\lim \limits_{\ell\rightarrow
{\mathfrak l}(m)} \Theta_2(m,\ell)=\vartheta(m)>\lim \limits_{m\rightarrow 1^-}\vartheta(m) = 1/4.$$
This implies that ${\rm inf}(\Theta_2({\mathcal D}))=1/4$.

\noindent {\bf Step IV}. We prove that $\Theta_3(\mathcal{D}) = (0,1/2\sqrt{3})$. $\Theta_3$ is strictly decreasing with respect to the variable $\ell$. Then ${\rm inf}(\Theta_3({\mathcal D}))=0$ if and only if
$\lim \limits_{\ell\rightarrow
+\infty}\Theta_3(m,\ell)=0$, $\forall m\in (0,1)$. Indeed, from
\[\begin{split}
& \lim \limits_{\ell\rightarrow
+\infty}\frac{1}{3\lambda_3(m,\ell)\ell+2(1+m)\ell^3}= \frac{1}{6\ell^3},\\
&\lim \limits_{\ell\rightarrow
+\infty} \Pi\left(\frac{6m\ell^2}{2(1+m)\ell^2+3\lambda_3(m,\ell)},m \right)=\frac{{\rm E}(m)}{1-m}
\end{split}\]
we have
$$\lim \limits_{\ell\rightarrow
+\infty} \Theta_3(m,\ell) =  \lim \limits_{\ell\rightarrow
+\infty}\frac{6 \Pi\left(\frac{6m\ell^2}{2(1+m)\ell^2+3\lambda_3(m,\ell)},m \right)}{\pi\ell (2 (1 + m) \ell^2 + 3 \lambda_3(m,\ell))}=\frac{{\rm E}(m)}{\pi(1-m)}\lim \limits_{\ell\rightarrow
+\infty}\frac{1}{\ell^3}=0.$$
The functions $\vartheta$ and $\ell\to \Theta_3(m,\ell)$ are strictly decreasing and, in addition
\begin{equation}\label{PL5Bis.F3}
\lim \limits_{\ell\rightarrow
{\mathfrak l}(m)} \Theta_3(m,\ell)=\vartheta(m)>\lim \limits_{m\rightarrow 0^+}\vartheta(m) = 1/2\sqrt{3}.
\end{equation}
Then, ${\rm sup}(\Theta_3({\mathcal D}))=1/2\sqrt{3}$.

\noindent {\bf Comment}.
\begin{figure}[h]
\begin{center}
\includegraphics[height=6.2cm,width=6.2cm]{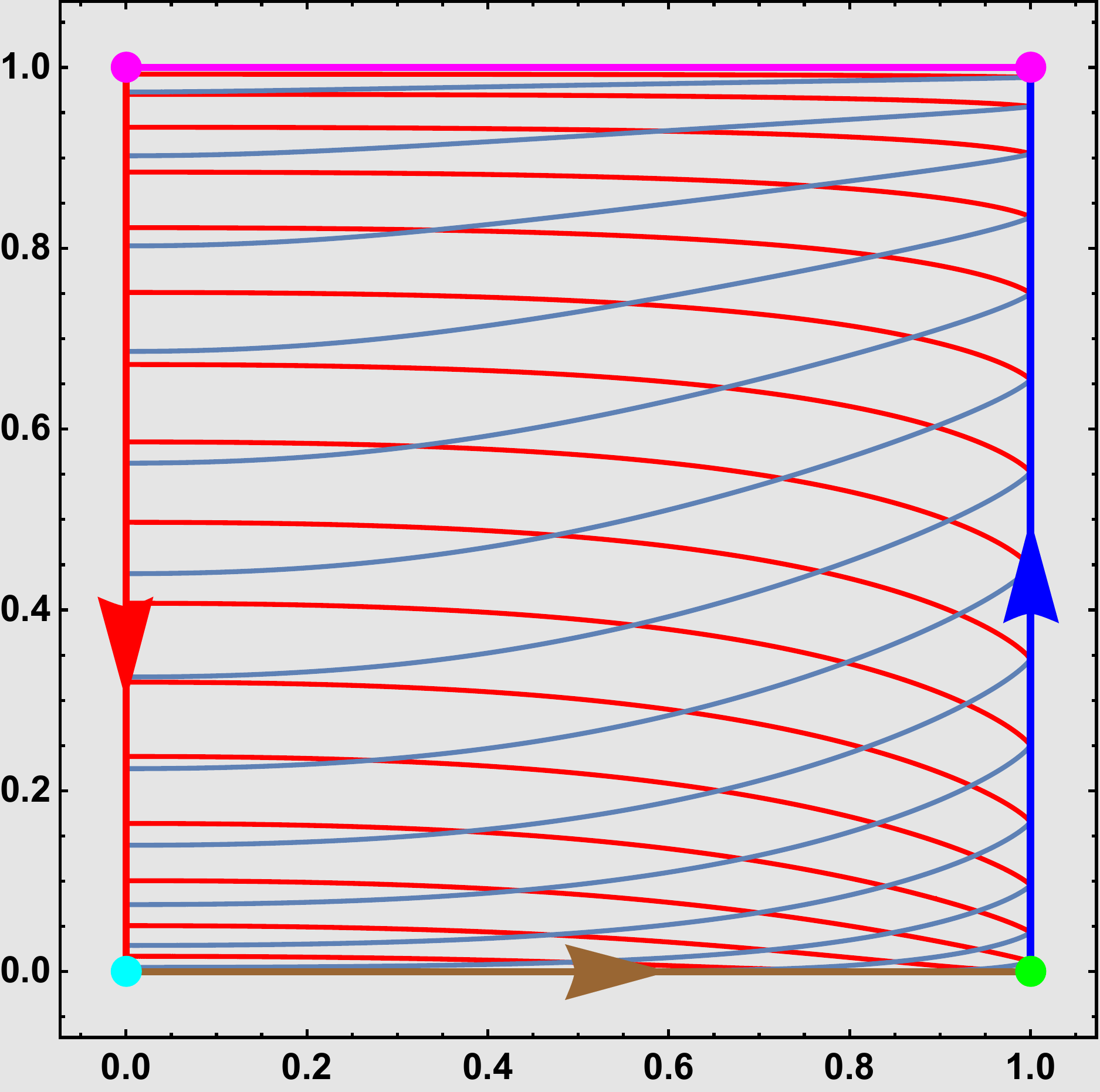}
\includegraphics[height=6.2cm,width=6.2cm]{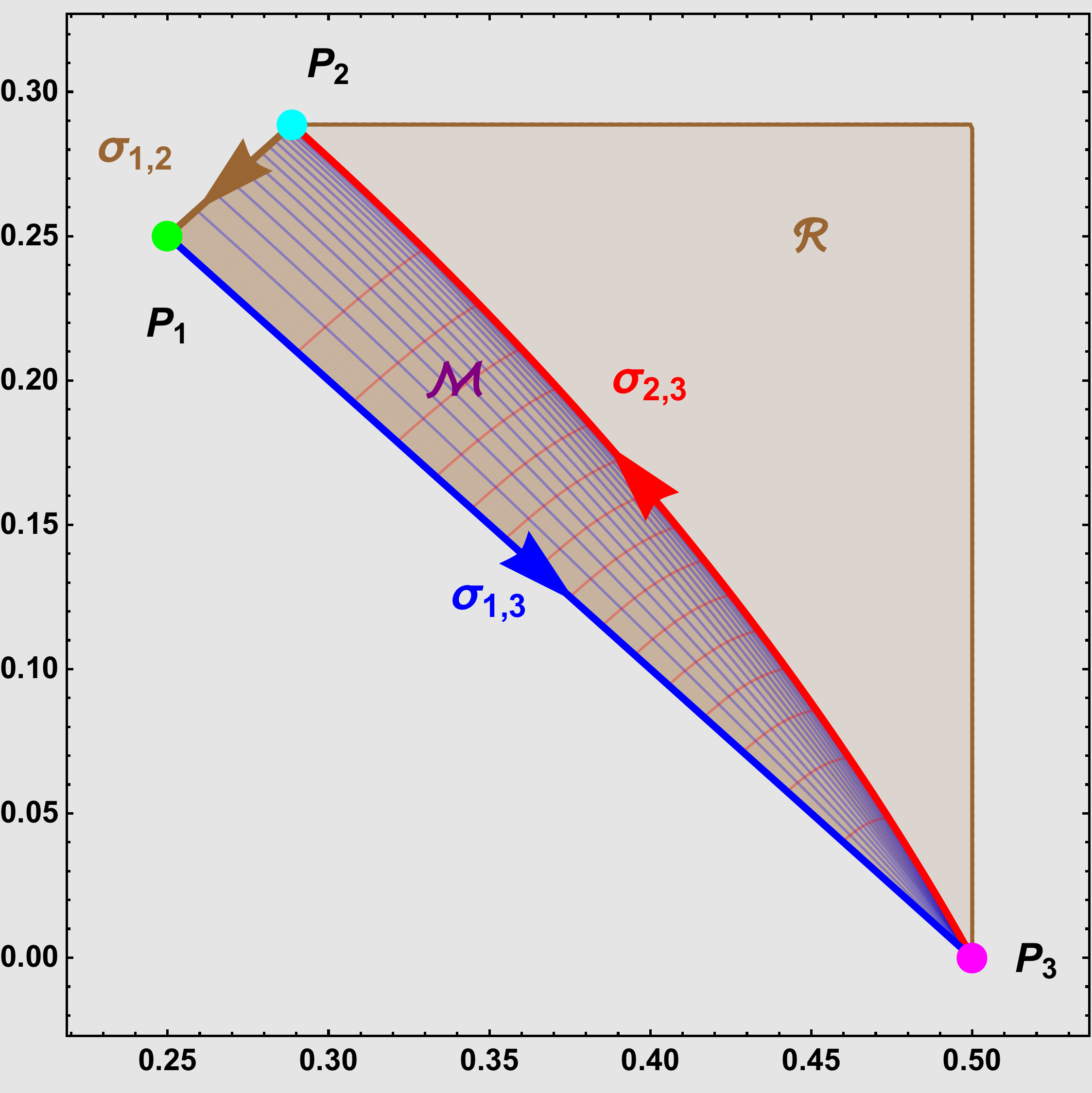}
\caption{\small{Left: the level curves of the components of $\widetilde{\Theta}_2$ (blue) and $\widetilde{\Theta}_3$ (red) of the modified period map $\widetilde{\Theta}$. Right: plot of the modified period map, the monodromic domain $ {\mathcal M}$ and the polygonal region $\mathcal{R}$.
 }}\label{FIG4}
\end{center}
\end{figure}
 For every ${\mathtt x}_2\in (1/4,1/2)$ and every ${\mathtt x}_3\in (0,1/2\sqrt{3})$, we denote by $\mathcal{C}_{2}({\mathtt x}_2)$ and $\mathcal{C}_{3}({\mathtt x}_3)$ the level curves
\[
\mathcal{C}_{2}({\mathtt x}_2)=\{(m,\ell)\in {\mathcal D}: \Theta_2(m,\ell)={\mathtt x}_2\},\quad
\mathcal{C}_{3}({\mathtt x}_3)=\{(m,\ell)\in {\mathcal D}: \Theta_3(m,\ell)={\mathtt x}_3\}.\]
Since the partial derivatives of $\Theta_2$ and $\Theta_3$ are non-zero at each point $(m,\ell)\in \mathcal{D}$, then $\mathcal{C}_{2}({\mathtt x}_2)$ and $\mathcal{C}_{3}({\mathtt x}_3)$ are smooth embedded curves, for every ${\mathtt x}_2$ and ${\mathtt x}_3$. To prove the injectivity of $\Theta$ we show that $\mathcal{C}_{2}({\mathtt x}_2)$ and $\mathcal{C}_{3}({\mathtt x}_3)$ are either disjoint or have only one point of intersection. This follows from  the next two steps.

\noindent {\bf Step V}. We claim that there exist a function ${\mathfrak m}_2:(1/4,1/2)\to (0,1)$ such that, for every ${\mathtt x}_2\in (1/4,1/2)$, the curve $\mathcal{C}_{2}({\mathtt x}_2)$ is the graph of a strictly increasing differentiable function
$\varphi_{{\mathtt x}_2}:({\mathfrak m}_2({\mathtt x}_2),1)\to \R$ satisfying
\begin{equation}\label{PL6.F1}{\mathfrak l}(m)<\varphi_{{\mathtt x}_2}(m),\quad \varphi_{{\mathtt x}_2}'\big|_m=-\frac{\partial_m\Theta_2}{\partial_{\ell}\Theta_2}\Big|_{(m,{\varphi_{{\mathtt x}_2}}(m))}.
\end{equation}
The function $\vartheta$ is continuous and strictly decreasing on $[0,1]$, is differentiable on $(0,1)$ and ${\rm Im}(\vartheta)=[1/4,1/2\sqrt{3}]$. Its inverse $\vartheta^{-1}:[1/4,1/2\sqrt{3}]\to [0,1]$ is continuous, strictly decreasing and differentiable on $(1/4,1/2\sqrt{3})$. If $c\in (1/4,1/2\sqrt{3})$, then  $m_c=\vartheta^{-1}(c)$ is the unique element of $(0,1)$ such that
$\lim \limits_{\ell\rightarrow
{\mathfrak l}(m_c)}\Theta_2(m_c,\ell)=\vartheta|_{m_c}=c$.
If ${\mathtt x}_2\ge 1/2\sqrt{3}$, we put $\mathfrak{m}_2({\mathtt x}_2)=0$. We prove that, for every $m\in (0,1)$, the equation $\Theta_2(m,\ell)= {\mathtt x}_2$ has a unique solution $\varphi_{{\mathtt x}_2}(m)\in (\mathfrak{l}(m),+\infty)$. Indeed,  $f_m : \ell \in (\mathfrak{l}(m),+\infty)\to \Theta_2(m,\ell)$ is a differentiable, strictly increasing function satisfying
$\lim \limits_{\ell\rightarrow
{\mathfrak l}(m)} f_m(\ell)=\vartheta(m)<1/2\sqrt{3}$ and $\lim \limits_{\ell\rightarrow +\infty} f_m(\ell) =1/2$.
Consequently, there exist a unique $\varphi_{{\mathtt x}_2}(m)\in (\mathfrak{l}(m),+\infty)$ such that 
$\Theta_2(m,\varphi_{{\mathtt x}_2}(m)) = {\mathtt x}_2$. Then, $\mathcal{C}_{2}({\mathtt x}_2)$ is the graph of the function $\varphi_{{\mathtt x}_2} : m\to \varphi_{{\mathtt x}_2}(m)$. We prove that $\varphi_{{\mathtt x}_2}$ is differentiable and that its derivative is as in (\ref{PL6.F1}). Since $\mathcal{C}_{2}({\mathtt x}_2)$ is smooth, for every $m_*\in (0,1)$ there exist a smooth embedding $\beta=(\beta_2,\beta_3):(-\epsilon,\epsilon)\to {\mathcal D}$ such that $\beta(0)=(m_*,\varphi_{{\mathtt x}_2}(m_*))$ and that $\beta((-\epsilon,\epsilon))\subset \mathcal{C}_{2}({\mathtt x}_2)$. Since the partial derivatives of $\Theta_2$ are never zero, the derivatives of $\beta_2$ and $\beta_3$ are non-zero, for each $t\in (-\epsilon,\epsilon)$. Hence, $\beta_2$ and $\beta_3$ are invertible and, by construction, $\phi_{{\mathtt x}_2}=\beta_3\circ \beta_2^{-1}$. This implies the differentiability of $\phi_{{\mathtt x}_2}$. Differentiation of $\Theta_2(m,\phi_{{\mathtt x}_2}(m))={\mathtt x}_2$ with respect to $m$ implies that the derivative of $\phi_{{\mathtt x}_2}$ is as in  (\ref{PL6.F1}).  If $1/4<{\mathtt x}_2< 1/2\sqrt{3}$, we put $\mathfrak{m}_2({\mathtt x}_2)=\vartheta^{-1}({\mathtt x}_2)$. Let $(m_*,\ell_*)$ be a point of  $\mathcal{C}_{2}({\mathtt x}_2)$. We prove that $m_*>{\mathfrak m}_2({\mathtt x}_2)$. By contradiction, suppose that $m_*\le {\mathfrak m}_2({\mathtt x}_2)$ and $\Theta_2(m_*,\ell_*)={\mathtt x}_2$. Since $\Theta_2(m,\ell)$ is strictly increasing with respect to the variable $\ell$ and $\vartheta$ is strictly decreasing, then
${\mathtt x}_2=\Theta_2(m_*,\ell_*)>\lim \limits_{\ell\rightarrow
{\mathfrak l}(m)} \Theta_2(m_*,\ell) = \vartheta(m_*)\ge \vartheta({\mathfrak m}_2({\mathtt x}_2))= {\mathtt x}_2
$.  Next we show that, for every $m>\mathfrak{m}_2({\mathtt x}_2)$, the equation $\Theta_2(m,\ell)={\mathtt x}_2$ has a unique solution $\varphi_{{\mathtt x}_2}(m)\in (\mathfrak{l}(m),+\infty)$. In fact, 
$\vartheta$ is strictly decreasing and satisfies
$$\vartheta(m)= \lim \limits_{\ell\rightarrow
{\mathfrak l}(m)^+} \Theta_2(m,\ell)< \vartheta(\mathfrak{m}_2({\mathtt x}_2))={\mathtt x}_2,\quad 
$$
while $f_m:\ell \to \Theta_2(m,\ell)$ is strictly increasing and satisfies
$$\lim \limits_{\ell\rightarrow
{\mathfrak l}(m)^+}f_m(\ell)=\vartheta(m)< {\mathtt x}_2< \frac{1}{2\sqrt{3}} < \frac{1}{2}=\lim \limits_{\ell\rightarrow
+\infty}f_m(\ell).
$$
Then, there is a unique $\varphi_{{\mathtt x}_2}(m)\in (\mathfrak{l}(m),+\infty)$ such that
$\Theta_2(m,f_{{\mathtt x}_2}(m))={\mathtt x}_2$, as claimed. Hence, $\mathcal{C}_{2}({\mathtt x}_2)$ is the graph of the function 
$\varphi_{{\mathtt x}_2}:m\in (\mathfrak{m}_2(m),1)\to \varphi_{{\mathtt x}_2}(m)\in (\mathfrak{l}(m),+\infty)$. Using the same arguments as above it is shown that $\varphi_{{\mathtt x}_2}$ is differentiable and that its derivative is as in (\ref{PL6.F1}).

\noindent {\bf Step VI}. We prove the existence of a  function ${\mathfrak m}_3:(0,1/2\sqrt{3})\to (0,1)$ such that, for every 
${\mathtt x}_3\in (0,1/2\sqrt{3})$, the curve $\mathcal{C}_{3}({\mathtt x}_3)$ is the graph of a stricly 
increasing differentiable function $\psi_{{\mathtt x}_3}:(0,{\mathfrak m}_3)\to \R$ such that
\begin{equation}\label{PL7.F1}{\mathfrak l}(m)< \psi_{{\mathtt x}_3}(m)\quad 
\psi_{{\mathtt x}_3}\big|_m=-\frac{\partial_m \Theta_3}{\partial_{\ell} \Theta_3}\Big|_{(m,{\psi_{{\mathtt x}_3}}(m))}.
\end{equation}
If ${\mathtt x}_3\in (0,1/4]$, we put ${\mathfrak m}_3({\mathtt x}_3)=0$. We show that for every $m\in (0,1)$ the equation  $\Theta_3(m,\ell)={\mathtt x}_3$ has a unique solution $\psi_{{\mathtt x}_3}(m)\in (\mathfrak{l}(m),+\infty)$. Indeed, it suffices to note that for every $m\in (0,1)$, the function $\mathtt{f}_m:\ell\to \Theta_3(m,\ell)$ is strictly decreasing and satisfies
$\lim \limits_{\ell\rightarrow
{\mathfrak l}(m)^+} \mathtt{f}_m(\ell)=1/4$, $\lim \limits_{\ell\rightarrow+\infty} \mathtt{f}_m(\ell)=0$. Consequently, $\mathcal{C}_{3}({\mathtt x}_3)$ is the graph of the function $\psi_{{\mathtt x}_3}:m\in (0,1)\to \psi_{{\mathtt x}_3}(m)$. Reasoning as in the previous step, one sees that $\psi_{{\mathtt x}_3}$ is differentiable and that its derivative is as in (\ref{PL7.F1}). If ${\mathtt x}_3\in (1/4,1/2\sqrt{3})$ we put $\mathfrak{m}_3({\mathtt x}_3)=\vartheta^{-1}({\mathtt x}_3)$. We prove that, if $(m_*,\ell_*)\in \mathcal{C}_{3}({\mathtt x}_3)$, then $m_*<\mathfrak{m}_3({\mathtt x}_3)$. By contradiction: suppose $m_*\ge \mathfrak{m}_3({\mathtt x}_3)$. Since $\Theta_3$ is strictly decreasing with respect to the variable $\ell$ and $\vartheta$ is strictly decreasing, then
$${\mathtt x}_3=\Theta_3(m_*,\ell_*)< \lim \limits_{\ell\rightarrow {\mathfrak l}(m)}\Theta_3(m_*,\ell)=
\vartheta(m_*)\le \vartheta(\mathfrak{m}_3({\mathtt x}_3))= {\mathtt x}_3.
$$
Finally, we prove that, for every $m\in (0,\mathfrak{m}_3({\mathtt x}_3))$, the equation $\Theta_3(m,\ell)={\mathtt x}_3$ has a unique solution $\psi_{{\mathtt x}_3}(m)\in (\mathfrak{l}(m),+\infty)$. In fact, $\Theta_3$ is strictly decreasing in the second variable, $\vartheta$ is strictly decreasing and
$$\lim \limits_{\ell\rightarrow {\mathfrak l}(m)^+}\Theta_3(m,\ell)=\vartheta(m)>\vartheta(\mathfrak{m}_3({\mathtt x}_3))={\mathtt x}_3,\quad 
\lim \limits_{\ell\rightarrow +\infty}\Theta_3(m,\ell)=0.$$
Then, $\mathcal{C}_{3}({\mathtt x}_3)$ is the graph of the function $\psi_{{\mathtt x}_3}:m\in (0,\mathfrak{m}_3({\mathtt x}_3))\to \psi_{{\mathtt x}_3}(m)$. Reasoning as in the previous cases one proves that $\psi_{{\mathtt x}_3}$ is differentiable and that its derivative is as in 
(\ref{PL7.F1}).

\noindent {\bf Conclusion}. We conclude the proof showing that $\mathcal{C}_{2}({\mathtt x}_2)$ and 
$\mathcal{C}_{3}({\mathtt x}_3)$ are either disjoint or else have a single point of intersection. 
If
$\mathcal{C}_{2}({\mathtt x}_2)\cap \mathcal{C}_{3}({\mathtt x}_3)\neq \emptyset$, then $({\mathfrak m}_2({\mathtt x}_2),1)\cap (0,{\mathfrak m}_3({\mathtt x}_3))$ is a non empty open interval ${\rm I}_{{\mathtt x}_2,{\mathtt x}_3}\subset (0,1)$ and $(m_*,\ell_*)\in \mathcal{C}_{2}({\mathtt x}_2)\cap \mathcal{C}_{3}({\mathtt x}_3)$ if and only if $m_*\in {\rm I}_{{\mathtt x}_2,{\mathtt x}_3}$ and
$(m_*,\ell_*) = (m_*,\varphi_{{\mathtt x}_2}(m_*))=(m_*,\psi_{{\mathtt x}_3}(m_*))$.
From (\ref{PL6.F1}) and (\ref{PL7.F1}) and keeping in mind that
$\partial_{\ell}\Theta_2>0$, $\partial_{\ell}\Theta_3<0$ and that
$\partial_m\Theta_2 \partial_{\ell}\Theta_3-\partial_{\ell}\Theta_2 \partial_m\Theta_3>0$,
we have
$$(\varphi_{{\mathtt x}_2}-\psi_{{\mathtt x}_3})'\big|_{m_*}=-
\frac{\partial_m\Theta_2 \partial_{\ell}\Theta_3-\partial_{\ell}\Theta_2 \partial_m\Theta_3}{\partial_{\ell}\Theta_2 \partial_{\ell}\Theta_3}\Big|_{(m_*,\ell_*)}>0.
$$
Then, $\varphi_{{\mathtt x}_2}-\psi_{{\mathtt x}_3}$ vanishes at $m_*$ and its derivative is strictly positive at its zeroes. So, $m_*$ is its only zero.
}\end{proof}

\begin{prop}\label{PIII}{The image of $\Theta$ coincides with the monodromic domain
}\end{prop}
\begin{proof}{The proof is subdivided into four intermediate steps and a conclusion.

\noindent {\bf Step I}. The image of the period  map is contained in the polygonal region (see Figure ~\ref{FIG4})
$$\mathcal{R}=\{(x,y)\in \R^2: x-y>0,x+y>1/2, 1/4<x<1/2,0<y<1/2\sqrt{3}\}.
$$
The inequalities $1/4<\Theta_2<1/2$ and $0<\Theta_3<1/2\sqrt{3}$ have been verified in the proof of the previous proposition. For every $m\in (0,1)$, the function $\ell\to \Theta_3(m,\ell)$ is strictly decreasing and $\ell \to \Theta_2(m,\ell)$ is strictly increasing. Thus 
$$\ell\in ({\mathfrak l}(m),+\infty)\to \Theta_2(m,\ell)-\Theta_3(m,\ell)\in \R$$ 
is strictly increasing, for every $m\in (0,1)$.  Since
$$\lim \limits_{\ell\rightarrow {\mathfrak l}(m)^+} \Theta_2(m,\ell) = \lim \limits_{\ell\rightarrow {\mathfrak l}(m)^+} \Theta_3(m,\ell) =\vartheta(m),\ \forall m\in (0,1)$$
then, $\Theta_2-\Theta_3$ is strictly positive on ${\mathcal D}$. From (\ref{PDT}) it follows that $\partial_{\ell}(\Theta_2+\Theta_3)$ is strictly negative on ${\mathcal D}$. Then,
$\ell\in (\mathfrak{l}(m),+\infty)\to \Theta_2(m,\ell)+\Theta_3(m,\ell)\in \R$
is strictly decreasing, for every $m\in (0,1)$. In the proof of Proposition \ref{PII} we showed that
$\lim \limits_{\ell\rightarrow +\infty}\Theta_2(m,\ell)=1/2$ and that $\lim \limits_{\ell\rightarrow +\infty}\Theta_3(m,\ell)=0$, for every $m\in (0,1)$. Hence
$$\Theta_2(m,\ell)+\Theta_3(m,\ell) > \lim \limits_{\ell\rightarrow +\infty}(\Theta_2(m,\ell)+\Theta_3(m,\ell))=1/2.
$$

\noindent {\bf Step II}. We prove that $\Theta({\mathcal D})\subseteq {\mathcal M}$. The arc $\sigma_{2,3}$ (see Remark \ref{arc} and Figure \ref{FIG4}) divides the interior of $\mathcal{R}$ into two disjoint connected sub-domains: $\mathcal{M}$ and the region above $\sigma_{2,3}$. Since 
$\Theta(\mathcal{D})\cap \mathcal{M}\neq \emptyset$, it suffices to check that $\Theta(\mathcal{D})\cap \sigma_{2,3}=\emptyset$. To this end, we consider the reparametrization of $\mathcal{D}$ defined by
\begin{equation}\label{Frep}
{\rm F}:(m,h)\in {\mathcal Q}:=(0,1)\times (0,1)\to (m, (1-h)^{-1/6}\mathfrak{l}(m))\in 
\mathcal{D}.\end{equation}
Let $\widetilde{\Theta}$ be the {\it modified period map}, defined by $\widetilde{\Theta}=\Theta\circ {\rm F}$ (see Figure \ref{FIG4}). From (\ref{PDT}) and (\ref{Frep}) we obtain
\begin{equation}\label{FP.F1}
\partial_m\widetilde{\Theta}_2|_{(m,h)}
=\frac{\sqrt{1-h}{\mathtt a}(m)}{{\mathtt r}_2(m,h){\mathtt b}(m)},\quad 
\partial_m\widetilde{\Theta}_3|_{(m,h)}
=-\frac{\sqrt{1-h}{\mathtt a}(m){\mathtt r}_3(m,h)}{{\mathtt b}(m)}
\end{equation}
where
\[\begin{cases}
{\mathtt b}(m)=2\pi m(m-1)(1 - m + m^2) \sqrt{3 (m-2) (1 + m) (2 m-1) + 6 (1 + (m-1) m)^{3/2}}\\ 
{\mathtt a}(m) =\alpha_{{\rm E}}(m){\rm E}(m)+\alpha_{{\rm K}}(m){\rm K}(m),\\
\alpha_{{\rm E}}(m) =-2 (1 + \sqrt{1 + (m-1) m}) + 
m (4 + 3 \sqrt{1 + (m-1) m}) +\\
\qquad\qquad\hspace{3pt} +3m^2 ( \sqrt{1 + (m-1) m}-2) -2 m^3 (m-2 + \sqrt{1 + (m-1) m}),\\
\alpha_{{\rm K}}(m) =(1 - m) \left(2(1 + \sqrt{1 + (m - 1) m})-m (3 + 2 \sqrt{1 + (m-1) m})\right)+\\
\qquad\qquad\hspace{3pt} + m ^2 (1 - m) (3 - m - \sqrt{1 + (m - 1) m}),
\end{cases}
\]
and
\[\begin{cases}
{\mathtt r}_2(m,h)= 1 + 2\sin(\frac{1}{3}\arcsin({\mathtt q}_2(m,h))),\\
{\mathtt r}_3(m,h)=\frac{1-\sqrt{3}\cos\left(\frac{1}{3}\arcsin({\mathtt q}_3(m,h))\right)+\sin\left(\frac{1}{3}\arcsin({\mathtt q}_3(m,h))\right)}{1+\cos\left(\frac{2}{3}\arcsin({\mathtt q}_3(m,h))\right)-\sqrt{3}\sin\left(\frac{2}{3}\arcsin({\mathtt q}_3(m,h))\right)},\\
{\mathtt q}_2(m,h) = 1-h\frac{2 m^3 + 2 (1 + \sqrt{1 - m + m^2}) + 
    m^2 (2 \sqrt{1 - m + m^2}-3)  - m (3 + 2 \sqrt{1 - m + m^2})}{2(1 - m + m^2)^{3/2}},\\
    {\mathtt q}_3(m,h) =  \frac{-2 h + 3 h m + 3 h m^2 - 2 h m^3 
    +2 (1 + (-1 + m) m)^{3/2} - 2 h (1 + (-1 + m) m)^{3/2}}{2(1 - m + m^2)^{3/2}}.
\end{cases}
\]
Observing that ${\mathtt a}$ and ${\mathtt r}_2$ are strictly positive and ${\mathtt b}$ and ${\mathtt r}_3$ are strictly negative, one sees $\partial_m \widetilde{\Theta}_2$ and $\partial_m \widetilde{\Theta}_3$ are strictly negative on ${\mathcal Q}$. Thus, $\widetilde{\Theta}_2$ and $\widetilde{\Theta}_3$ are strictly decreasing functions with respect to the variable $m$.
On the other hand, the arc $\sigma_{2,3}$ is parameterized by
$$\beta : h\in (0,1)\to \lim \limits_{m\rightarrow 0} \widetilde{\Theta}(m,h),$$ whose components $(\beta_2,\beta_3)$ are respectively given by $x(h)$ and $y(h)$ as defined in (\ref{DMN.F1}).
Note that $\beta_2$ is stricly increasing and that $\beta_3$ is strictly decreasing. By contradiction, suppose that ${\rm Im}(\Theta)\cap \sigma_{2,3}\neq \emptyset$. Then, there exist $(m_*,h_*)\in {\mathcal Q}$ and $k_*\in (0,1)$ such that $\widetilde{\Theta}(m_*,h_*)=\beta(k_*)$. Consequently, we have
\[\begin{cases}\beta_2(h_*)=\lim \limits_{m\rightarrow 0} \widetilde{\Theta}_2(m,h_*)>\widetilde{\Theta}_2(m_*,h_*)=\beta_2(k_*)=\lim \limits_{m\rightarrow 0} \widetilde{\Theta}_2(m,k_*),\\
\beta_3(h_*)=\lim \limits_{m\rightarrow 0} \widetilde{\Theta}_3(m,h_*)>\widetilde{\Theta}_3(m_*,h_*)=\beta_3(k_*)=\lim \limits_{m\rightarrow 0} \widetilde{\Theta}_3(m,k_*).
\end{cases}\]
Since $\beta_2$ is strictly increasing and $\beta_3$ is strictly decreasing, we get $h_*>k_*$ and $k_*<h_*$. We have thus found a contradiction. 

\noindent {\bf Step III}. Note that
$$\left|\frac{{\mathtt a}(m)}{{\mathtt b}(m)}\right|\le {\rm C}_1\left|\log(1-m)\right|,\quad
0\le \frac{\sqrt{1-h}}{{\mathtt r}_2(m,h)}\le \frac{\sqrt{3}}{2},\quad
0\le -\sqrt{1-h}\ {\mathtt r}_3(m,h)\le 1/2,
$$
where ${\rm C}_1$ is a positive constant and $(m,h)\in [0,1)\times [0,1]$. From these bounds and using (\ref{FP.F1}) we obtain \begin{equation}\label{FP.F6}
\left|\partial_m\widetilde{\Theta}_2|_{(m,h)}\right|< \widetilde{{\rm C}}_1\left|\log(1-m)\right|,\quad
\left|\partial_m\widetilde{\Theta}_3|_{(m,h)}\right|< \widetilde{{\rm C}}_1\left|\log(1-m)\right|,
\end{equation}
for every $(m,h)\in [0,1)\times [0,1]$ and some positive constant $\widetilde{{\rm C}}_1$.

\noindent {\bf Step IV}. Now prove that
\begin{equation}\label{FP.F8}
\begin{split}
&\left |\partial_h\widetilde{\Theta}_j|_{(m,h)}\right|\le  \frac{\widetilde{{\rm C}}_2}{(1-m)^2\sqrt{h}},\hspace{30pt}  j=2,3,\quad 
\forall (m,h)\in [0,1)\times(0,1/2),\\
&\left|\partial_h\widetilde{\Theta}_j|_{(m,h)}\right|\le  \frac{\widetilde{{\rm C}}_2}{m(1-m)\sqrt{1-h}},\quad  j=2,3,\quad 
\forall (m,h)\in (0,1)\times(1/2,1),\end{split}
\end{equation}
for some positive constant $\widetilde{{\rm C}}_2$ .
From (\ref{PDT}) and (\ref{Frep}) we obtain
\begin{equation}\label{FP.F7}
\partial_h\widetilde{\Theta}_2|_{(m,h)}
=\frac{{\mathtt g}(m){\mathtt s}_2(m,h){\mathtt t}_2(m,h)}{\pi\sqrt{3(1-h)}},\hspace{2pt}
\partial_h\widetilde{\Theta}_3|_{(m,h)}
=\frac{{\mathtt g}(m){\mathtt s}_3(m,h){\mathtt t}_3(m,h)}{\pi\sqrt{3(1-h)}}
\end{equation}
where
\[\begin{cases}
&{\mathtt g}(m)=\frac{\sqrt{(m-2) (1 + m) (2 m-1) + 2 (1 + (m-1) m)^{3/2}}}{6 (1 - m + m^2)},\\
&{\mathtt s}_2(m,h)=3{\rm E}(m)+ \left(m - 2 + 2 \sqrt{1 - m + m^2} \sin\big(\frac{\arcsin({\mathtt q}_2(m,h))}{3}\big)\right){\rm K}(m),\\
&{\mathtt s}_3(m,h)=3{\rm E}(m)+\Big(m - 2 + 2 \sqrt{1 - m + m^2}\Big(\sqrt{3}\cos(\frac{\arcsin({\mathtt q}_2(m,h))}{3})-\\&\hspace{50pt} -\sin(\frac{\arcsin({\mathtt q}_2(m,h))}{3})\Big)\Big){\rm K}(m),\\
&{\mathtt t}_2(m,h)=\frac{-1}{1-2\cos\left(\frac{2\arcsin({\mathtt q}_2(m,h))}{3}\right)},\\
&{\mathtt t}_3(m,h)=\frac{-1}{1+\cos\left(\frac{2\arcsin({\mathtt q}_2(m,h))}{3}\right)-\sqrt{3}\sin\left(\frac{2\arcsin({\mathtt q}_2(m,h))}{3}\right)}
\end{cases}
\]
The function ${\mathtt g}$ is positive on $[0,1)$ and bounded above by a positive constant ${\rm C}$ on $[0,1]$. 
The function  ${\mathtt t}_2$ satisfies
\begin{equation}\label{FP.F10}
\left|{\mathtt t}_2(m,h)\right| (1-m^2)\sqrt{h}\le \sqrt{1-h},
\end{equation}
for every $(m,h)\in [0,1)\times(0,1/2)$. Similarly, ${\mathtt s}_3$ and ${\mathtt t}_3$ satisfy
\begin{equation}\label{FP.F11}
0<{\rm C}'\le 3 {\rm E}(m) + \big(m-2 + \sqrt{1 - m + m^2}\big) {\rm K}(m)\le {\mathtt s}_3(m,h) \le 3{\rm E}(m)\le \frac{3\pi}{2}\end{equation}
for every $(m,h)\in [0,1]\times [0,1]$, and
\begin{equation}\label{FP.F12}
-\sqrt{1-h}\le 2{\mathtt t}_3(m,h)(1-m^2)\sqrt{h}\le 0,
\end{equation}
for every $(m,h)\in [0,1)\times (0,1/2)$. Since ${\mathtt s}_2$ is non-negative and bounded above, the bounds in the first line of (\ref{FP.F8}) follow from (\ref{FP.F10})-(\ref{FP.F12}). The functions ${\mathtt t}_2$ and ${\mathtt t}_3$
satify $$0\le  3\sqrt{3}\ {\mathtt t}_2(m,h)m(1-m)\le \sqrt{2}$$ and $$-\sqrt{2}/3\le {\mathtt t}_3(m,h)(1-m)m\le 0,$$
for every $(m,h)\in (0,1)\times (1/2,1)$. Then,
using  (\ref{FP.F11}) and recalling that ${\mathtt g}$ and ${\mathtt s}_2$ are non-negative and bounded above, we see that $\partial_h\widetilde{\Theta}_j$, $j=2,3$, fulfill the bounds in the second line of (\ref{FP.F8}).

\noindent {\bf Conclusion}. We are now in a position to prove that $\widetilde{\Theta}({\mathcal D})=\mathcal{M}$. Preliminarily, we observe that
\[\begin{cases}
\tau_{0,2}(h):=\lim \limits_{m\rightarrow 0^+} \widetilde{\Theta}_2(m,h)=\frac{1-h}{\sqrt{3}\left(1 + 2 \sin(\arcsin(1-2h)/3)\right)},\\
\tau_{0,3}(h):=\lim \limits_{m\rightarrow 0^+} \widetilde{\Theta}_3(m,h)=\frac{1-h}{\sqrt{3} (1 + \sqrt{3}\cos(\arcsin(1-2h)/3) - \sin(\arcsin(1-2h)/3))},\\
\tau_{1,2}(h):=\lim \limits_{m\rightarrow 1^-} \widetilde{\Theta}_2(m,h)=\frac{3}{8}-\frac{1}{4\pi}\arcsin(1-2h),
\\
\tau_{1,3}(h):=\lim \limits_{m\rightarrow 1^-} \widetilde{\Theta}_3(m,h)=\frac{1}{8}+\frac{1}{4\pi}\arcsin(1-2h),\\
\widehat{\tau}_{0,2}(m):=\lim \limits_{h\rightarrow 0^+} \widetilde{\Theta}_2(m,h)=\vartheta(m),\\
\widehat{\tau}_{0,3}(m):=\lim \limits_{h\rightarrow 0^+} \widetilde{\Theta}_3(m,h)=\vartheta(m),\\
\widehat{\tau}_{1,2}(m):=\lim \limits_{h\rightarrow 1^-} \widetilde{\Theta}_2(m,h)=1/2,\\
\widehat{\tau}_{1,3}(m):=\lim \limits_{h\rightarrow 1^-} \widetilde{\Theta}_3(m,h)=0,
\end{cases}
\]
Then,

\noindent $\bullet$ $\tau_0:h\in [0,1]\to (\tau_{0,2}(h),\tau_{0,3}(h))$ is a parametrization of $\sigma_{2,3}$,\\
\noindent $\bullet$ $\tau_1:h\in [0,1]\to (\tau_{1,2}(h),\tau_{1,3}(h))$ is a parametrization of the segment $\sigma_{1,3}$,\\
\noindent $\bullet$  $\widehat{\tau}_0:h\in [0,1]\to (\widehat{\tau}_{0,2}(h),\widehat{\tau}_{0,3}(h))$ is a parametrization of the segment $\sigma_{1,2}$,\\
where $\sigma_{2,3}$, $\sigma_{1,3}$ and $\sigma_{1,2}$ are defined as in Remark \ref{arc} (see also Figure \ref{FIG4}).

\noindent By contradiction, suppose that ${\rm Im}(\widetilde{\Theta})$ is properly contained in ${\mathcal M}$. Then, there exist ${\mathtt q}\in \mathcal{M}$ such that  ${\mathtt q}\in \partial\big({\rm Im}(\widetilde{\Theta})\big)$. Let $\{{\mathtt q}_n\}_{n\in \mathbb{N}}\subset {\rm Im}(\widetilde{\Theta})$ be a sequence converging to ${\mathtt q}$. For each $n\in {\mathbb N}$, we choose ${\mathtt d}_n=(m_n,h_n)\in {\mathcal Q}$ such that $\widetilde{\Theta}({\mathtt d}_n)= {\mathtt q}_n$. Without loss of generality, $\{{\mathtt d}_n\}_{n\in \mathbb{N}}$ converges to ${\mathtt d}_*=(m_*,h_*)\in \overline{{\mathcal Q}}$.  Since ${\mathtt q}\notin {\rm Im}(\Theta)$, the point
${\mathtt d}_*$ belongs to $\partial{\mathcal Q}$. There are four possible cases: $m_*=1$ and $h\in [0,1]$, or $m_*=0$ and $h\in [0,1]$, or  $h_*=0$ and $m_*\neq 0,1$ or else $h_*=1$ and $m_*\neq 0,1$.

\noindent {\it Case $1$}: $m_*=1$ and $h\in [0,1]$. From (\ref{FP.F6}) we have
$$\sqrt{\left(\partial_m\widetilde{\Theta}_2|_{(m,h)}\right)^2+\left(\partial_m\widetilde{\Theta}_3|_{(m,h)}\right)^2}<{\rm C}|\log(1-m)|,$$
for some positive constant ${\rm C}$. This implies
\begin{multline}{\rm d}\Big(\widetilde{\Theta}({\mathtt d}_n),\tau_1(h_n)\Big)\le  \int_{m_n}^{1}
\sqrt{\left(\partial_m\widetilde{\Theta}_2|_{(m,h)}\right)^2+\left(\partial_m\widetilde{\Theta}_3|_{(m,h)}\right)^2}dm\le \\
 \le {\rm C}\int_{m_n}^{1} \left|\log(1-m)\right| dm = {\rm C}\big(1  - m_n-\log(1-m_n)+m_n\log(1-m_n)\big).\nonumber
\end{multline}
Hence,
\[ \lim \limits_{n\rightarrow \infty} {\rm d}\Big(\widetilde{\Theta}({\mathtt d}_n),\tau_1(h_*)\Big)\le
 \lim \limits_{n\rightarrow \infty}{\rm d}\Big(\widetilde{\Theta}({\mathtt d}_n),\tau_1(h_n)\Big)+
 \lim \limits_{n\rightarrow \infty}{\rm d}\Big(\tau_1(h_n),\tau_1(h_*)\Big)=0.\]
Thus, ${\mathtt q}=\tau_1(h_*)\in \partial {\mathcal M}$, contradicting the hypothesis that ${\mathtt q}\in {\mathcal M}$.

\noindent {\it Case $2$}: $m_*=0$ and $h\in [0,1]$. We assume $m_n<1/2$, for every $n$. From (\ref{FP.F6}) we infer that
$\big(\partial_m\widetilde{\Theta}_2|_{(m,h)}\big)^2+\big(\partial_m\widetilde{\Theta}_3|_{(m,h)}\big)^2$ is bounded on $(0,1/2)\times (0,1)$. Reasoning as above, we obtain $\lim \limits_{n\rightarrow \infty} {\rm d}\big(\widetilde{\Theta}({\mathtt d}_n),\tau_1(h_*)\big)=0$. Hence, ${\mathtt q}=\tau_1(h_*)\in \partial {\mathcal M}$. So, even in this case, we have come to a contradiction.

\noindent {\it Case $3$}: $m_*\neq 0,1$ $h_*=0$. We assume $h_n<1/2$ and $m_n<m_{**}<1$, for every $n$. Then, (\ref{FP.F8}) implies that
$\big(\partial_h\widetilde{\Theta}_2|_{(m,h)}\big)^2+\big(\partial_h\widetilde{\Theta}_3|_{(m,h)}\big)^2$ is bounded from above by ${\rm C}^2/h$ on $(0,m_{**})\times (0,1/2)$, for some positive constant ${\rm C}$. Hence,
\[{\rm d}\Big(\widetilde{\Theta}({\mathtt d}_n),\widehat{\tau}_0(m_*)\Big)\le \int_0^{h_n} \sqrt{\left(\partial_h\widetilde{\Theta}_2|_{(m,h)}\right)^2+\left(\partial_h\widetilde{\Theta}_3|_{(m,h)}\right)^2} dh\le 2{\rm C}\sqrt{h_n}.
\]
Reasoning as above, we deduce that ${\mathtt q}=\widehat{\tau}_0(m_*)\in \partial {\mathcal M}$. So even in this third case we have reached a contradiction.

\noindent {\it Case $4$}: $h_*=1$ and $0<m_*< 1$. We may assume $h_n\in (1/2,1)$ and
$0<m_{**}<m_n<m_{***}<1$. From (\ref{FP.F8}) we infer that
$\big(\partial_h\widetilde{\Theta}_2|_{(m,h)}\big)^2+\big(\partial_h\widetilde{\Theta}_3|_{(m,h)}\big)^2$ is bounded from above by ${\rm C}^2/(1-h)$ on $[m_{**},m_{***}]\times (1/2,1)$, for some positive constant ${\rm C}$. Proceeding as before, this implies that  ${\rm d}\big(\widetilde{\Theta}({\mathtt d}_n), \widehat{\tau}_1(m_*)\big)$ tends to $0$ as $n\to \infty$. Then,
${\mathtt q}=\widehat{\tau}_1(m_*)\in \partial {\mathcal M}$. So even in the last case we reached a contradiction.}\end{proof}

\section{Quantization}\label{6}
\subsection{Characteristic numbers}\label{S6.1}
 
\begin{defn}{Let ${\mathcal D}_*=\{(x,y)\in {\mathcal D}\hspace{1pt} : \hspace{1pt} \Theta(x,y)\in {\mathbb Q}^2\}$ and $\gamma$ be the canonical parameterization of a closed string  with characters $(m,\ell)\in {\mathcal D}_*$. 
We  call $(q_2,q_3)=\Theta(m,\ell)$ the {\it modulus} of $\gamma$. The positive integers ${\mathtt h}_j,{\mathtt k}_j$, $j=1,2$, such that ${\mathtt h}_1/{\mathtt k}_1=2q_2+q_3$, ${\mathtt h}_2/{\mathtt k}_2=q_3-q_2$ and  that ${\rm gcd}({\mathtt h}_1,{\mathtt k}1)={\rm gcd}({\mathtt h}_2,{\mathtt k}_2)=1$
are said the {\it characteristic numbers} of $\gamma$. The integer ${\mathtt n}={\rm lcm}({\mathtt k}_1,{\mathtt k}_2)$ is the {\it wave number} of $\gamma$. A {\it symmetry} of $\gamma$ is an element $[{\bf A}] \in \widehat{{\rm G}}$, such that $[{\bf A}] \cdot |[\gamma]| = |[\gamma]|$. The set of all symmetries of $\gamma$ is a subgroup $\widehat{{\rm G}}_{\gamma}$ of $\widehat{{\rm G}}$.}\end{defn}

\begin{remark}{The cr-curvature $\kappa_{m,\ell}$ of $\gamma$ is the periodic function with least period $\omega_{m,\ell}=2{\rm K}(m)/\ell$ defined in (\ref{4S.2.F3}). From (\ref{4S.2.F13}) and (\ref{periodic}) it follows that $\gamma$ is periodic, with least period $\mathtt{n}\omega_{m,\ell}$. Its trajectory decomposes as the disjoint union of ${\mathtt n}$-fundamental arcs $|[\gamma_n]|=\gamma([(n-1)\omega_{m,\ell}, n\omega_{m,\ell}))$, $n=1,...,{\mathtt n}$, referred to as the {\it indecomposable waves}. The indecomposable waves are congruent each other and their total strain is $\omega_{m,\ell}$. We may think of $\omega_{m,\ell}$ to as the {\it wavelength} of $\gamma$.
The total strain of $\gamma$ is ${\mathtt n}\omega_{m,\ell}$.}\end{remark}

 
\begin{defn}{The stabilizer of the momentum ${\mathfrak m}_{\gamma}$ of $\gamma$ is a maximal compact Abelian subgroup ${\rm T}^2_{\gamma}\subset  \widehat{{\rm G}}$. The singular orbits $\mathcal{O}^1_\gamma$ and $\mathcal{O}^2_\gamma$ of the action of ${\rm T}^2_{\gamma}$ on ${\mathcal S}$ are said the {\it axes of symmetry} of $\gamma$. Let ${\bf B}_{\gamma}$ is a Wilczysnki frame, then ${\mathcal R}_{\gamma}=[{\bf B}_{\gamma}(\omega_{m,\ell})\cdot {\bf B}_{\gamma}(0)^{-1}]\in {\rm T}^2_{\gamma}$ is called {\it monodromy} of $\gamma$. 
}\end{defn}

 

\begin{defn}{A closed string $\gamma$ is said in a {\it symmetrical configuration} if ${\rm T}^2_{\gamma}={\rm T}^2$, where  
${\rm T}^2$ is the maximal torus defined in (\ref{torus}). Every closed string is congruent to a symmetrical configuration. The axes of symmetry of a symmetrical configuration are the chains ${\mathcal O}_1$ and ${\mathcal O}_2$ considered in Definition \ref{AS}. If $\gamma$ is a symmetrical configuration and if ${\bf L}\in {\rm G}$ is as in
(\ref{dL}), then, $\gamma^{\sharp}={\bf L}\cdot \gamma$ is another symmetrical configuration, the {\it dual of} 
$\gamma$.}\end{defn}

\subsection{The proof of Theorem C}\label{S6.2} We now prove the third main result of the paper.
\vspace{0.5pt}

\noindent {\bf Theorem C.}{\it \hspace{1pt} Let $\gamma$ be a closed string with characteristic numbers $({\mathtt h}_1,{\mathtt k}_1,{\mathtt h}_2,{\mathtt k}_2)$. Then, $\widehat{{\rm G}}_{\gamma}$ is a non-trivial subgroup of order ${\mathtt n}$ contained in a unique maximal torus ${\rm T}^2_{\gamma}$ and, in addition, $|[\gamma]|$ doesn't intersect its axes of symmetry and the integers ${\mathtt l}_1={\mathtt n}{\mathtt h}_2/{\mathtt k}_2$, ${\mathtt l}_2=-{\mathtt n}{\mathtt h}_1/{\mathtt k}_1$ are the linking numbers of $\gamma$ with the symmetry axes.}
\begin{proof}{The proof is organized into in five parts.

\noindent {\bf Part I.}   We build, for every $(m,\ell)\in {\mathcal D}_*$, a natural parameterization $\gamma_{m,\ell}$ of a closed string with characters $(m,\ell)$. Our construction is based on what has been shown in the subsection \ref{4S.2} (particularly in the proof of Theorem \ref{closureconditions}), therefore we are going to adopt a notation consistent with that one already used.\\Denote by $\Phi_j$, $j=1,2,3$, the {\it angular functions}
\begin{equation}\label{angffun}\Phi_j(s|m,\ell)= \int_0^s \frac{-6du}{4\kappa_{m,\ell}(u)+3\lambda_j(m,\ell)}=\frac{-6\Pi\left(\frac{6m\ell^2}{2(1+m)\ell^2+3\lambda_j(m,\ell)},{\rm am}(\ell s,m),m \right)}{\ell \left(2 (1 + m) \ell^2 + 3 \lambda_j(m,\ell)\right)}.
\end{equation}
We put
\[\begin{cases}
r_1(m,\ell)&=\frac{\sqrt{6}}{\sqrt{(\lambda_3(m,\ell)-\lambda_1(m,\ell))(\lambda_3(m,\ell)-\lambda_2(m,\ell))}}\\
r_2(m,\ell)&=\frac{\sqrt{6}}{\sqrt{(\lambda_3(m,\ell)-\lambda_2(m,\ell))(\lambda_2(m,\ell)-\lambda_1(m,\ell))}},\\
r_3(m,\ell)&=\frac{\sqrt{6}}{\sqrt{(\lambda_3(m,\ell)-\lambda_1(m,\ell))(\lambda_2(m,\ell)-\lambda_1(m,\ell))}}
\end{cases}
\]
and we define
\begin{equation}\label{cp}
\begin{cases}
z_1(s|m,\ell)=r_1(m,\ell)\sqrt{\lambda_2(m,\ell)-\lambda_1(m,\ell)}\sqrt{4\kappa_{m,\ell}(s)+3\lambda_3(m,\ell)}\ e^{-i\Phi_3(s|m,\ell)}\\
z_2(s|m,\ell)=r_2(m,\ell)\sqrt{4\kappa_{m,\ell}(s)+3\lambda_2(m,\ell)}\ e^{-i\Phi_2(s|m,\ell)},\\
z_3(s|m,\ell)=r_3(m,\ell)\sqrt{\lambda_3(m,\ell)-\lambda_2(m,\ell)}\sqrt{4\kappa_{m,\ell}(s)+3\lambda_1(m,\ell)}\ e^{-i\Phi_1(s|m,\ell)}
\end{cases}
\end{equation}
Let ${\mathfrak U}\in {\rm GL}(3,\C)$ be as in (\ref{U}) and $\gamma_{m,\ell}:\R\to {\mathcal S}$ be defined by
$$\gamma_{m,\ell}:s\to [{\mathfrak U}\cdot {}^t\left(z_1(s|m,\ell),z_2(s|m,\ell),z_3(s|m,\ell)\right)].$$
We prove that $\gamma_{m,\ell}$ is a natural parameterization of a closed string with characters $(m,\ell)$.
To this end we consider any natural parameterization $\gamma$ of a closed string with characters $(m,\ell)$ and Wilczynski frame ${\bf B}$. Let ${\rm Z}_j:\R\to \C^{2,1}-\{0\}$, $j=1,2,3$, be defined by
\begin{equation}\label{F4.P1.S6.2} 
{\rm Z}_1=r_3(m,\ell){\rm W}_{\lambda_{3}(m,\ell)},\quad
{\rm Z}_2=r_2(m,\ell){\rm W}_{\lambda_{2}(m,\ell)},\quad
{\rm Z}_3=r_1(m,\ell){\rm W}_{\lambda_{1}(m,\ell)},
\end{equation}
where ${\rm W}_{\lambda_{j}(m,\ell)}$, $j=1,2,3$, are as in (\ref{4S.2.F9}).
From the proof of Theorem \ref{closureconditions} it follows that ${\bf B}\cdot {\rm Z}_{1}$, ${\bf B}\cdot {\rm Z}_{2}$ and ${\bf B}\cdot {\rm Z}_{3}$ are constant eigenvectors of the momentum, paired with the eigenvalues $\lambda_3(m,\ell)$, $\lambda_2(m,\ell)$ and $\lambda_1(m,\ell)$ respectively. It is a computational matter to check that $\langle {\rm Z}_{i},{\rm Z}_{j}\rangle = \varepsilon_i\delta_{ij}$, $\varepsilon_1=\varepsilon_2=1$, $\varepsilon_3=-1$ and that $\Omega({\rm Z}_{1},{\rm Z}_{2},{\rm Z}_{3})=1$. Then, ${\bf Z}|_s=({\rm Z}_1(s),{\rm Z}_2(s),{\rm Z}_3(s))$ is a unimodular, pseudo-unitary basis of $\C^{2,1}$ and
${\bf B}\cdot {\bf Z}={\mathfrak C}$, where ${\mathfrak C}$ is a constant unimodular, pseudo-unitary basis of $\C^{2,1}$. Let ${\bf M}$ be the unique element of ${\rm G}$ such that ${\bf M}\cdot {\mathfrak C} = {\mathfrak U}$. By construction, the first column vector of ${\mathfrak U}\cdot {\bf Z}^{-1}$ is a normalized lift of $\gamma_{m,\ell}$ and ${\bf B}_{m,\ell}={\mathfrak U}\cdot {\bf Z}^{-1}$ is a Wilczynski frame along $\gamma_{m,\ell}$. Since ${\bf B}_{m,\ell}={\bf M}\cdot {\bf B}$, then $\gamma_{m,\ell}$ and $\gamma$ are congruent with each other. This shows that $\gamma_{m,\ell}$ is a natural parameteriziation of a closed string with characters $(m,\ell)$.

\noindent {\bf Part II.} We prove that $\gamma_{m,\ell}$ is a symmetrical configuration, the {\it standard symmetrical configuration}  with characters $(m,\ell)$. By construction, 
\begin{equation}\label{Z}{\bf Z}^{-1}|_s=(e^{-i\Phi_3(s|m,\ell)}{\bf E}_1^1+e^{-i\Phi_2(s|m,\ell)}{\bf E}_2^2+e^{i(\Phi_2(s|m,\ell)+\Phi_3(s|m,\ell))}{\bf E}_3^3)\cdot {\bf P}|_s,\end{equation}
where ${\bf P}$ is periodic, with least period $\omega_{m,\ell}$. Then, (\ref{Z}) and (\ref{angffun}) imply that the monodromy ${\mathcal R}_{m,\ell}$ of $\gamma_{m,\ell}$ is given by
\begin{equation}\label{F5.P1.S6.2}{\mathcal R}_{m,\ell}=[{\mathfrak U}\cdot (e^{2i\pi q_3}{\bf E}^1_1+e^{2\pi i q_2}{\bf E}^2_2+e^{-2\pi i (q_3+q_2)}{\bf E}^3_3)\cdot {\mathfrak U}^{-1}],
\end{equation}
where $(q_2,q_3)$ is the modus of $\gamma_{m.\ell}$. Hence, ${\mathcal R}_{m,\ell}\in {\rm T}^2$.\\To conclude the reasoning we show that ${\rm T}^2$ is the stabilizer of the momentum.  From (\ref{F5.P1.S6.2}) we have 
\begin{equation}\label{fact}\begin{cases}{\mathcal R}_{m,\ell}={\mathcal R}'_{m,\ell}{\mathcal R}''_{m,\ell},\\
{\mathcal R}'_{m,\ell}=[e^{-\frac{\pi i{\mathtt h}_1}{3{\mathtt k}_1}}\big(\cos(\pi\frac{•{\mathtt h}_1}{{\mathtt k}_1})({\bf E}^1_1+{\bf E}^3_3)+\sin(\pi\frac{•{\mathtt h}_1}{{\mathtt k}_1})({\bf E}^3_1-{\bf E}^1_3)\big)+
e^{\frac{2\pi i{\mathtt h}_1}{3{\mathtt k}_1}}{\bf E}^2_2],\\{\mathcal R}''_{m,\ell}=[e^{-\frac{2\pi i{\mathtt h}_2}{3{\mathtt k}_2}}({\bf E}^1_1+{\bf E}^3_3)+e^{\frac{4\pi i{\mathtt h}_2}{3{\mathtt k}_2}}{\bf E}^2_2].
\end{cases}\end{equation}
Then, 
$p_h\circ {\mathcal R}'_{m,\ell}\circ p_h^{-1}={\rm R}_{Oz}(2\pi {\mathtt h}_1/{\mathtt k}_1)$ and
$p_h\circ[{\bf L}]\cdot {\mathcal R}''_{m,\ell}\cdot [{\bf L}]^{-1}\circ p_h^{-1}={\rm R}_{Oz}(-2\pi {\mathtt h}_2/{\mathtt k}_2)$,
where  ${\rm R}_{Oz}(\theta)$ is the rotation of an angle $\theta$ around the $Oz$-axis of $\R^3$. Hence, ${\mathcal R}'_{m,\ell}$ has order ${\mathtt k}_1$ and ${\mathcal R}''_{m,\ell}$ has order ${\mathtt k}_2$.
Consequently, ${\mathcal R}_{m,\ell}$ is an element of order ${\mathtt n}>1$ belonging to ${\rm T}^2$ and stabilizing the momentum ${\mathfrak m}_{m,\ell}$ of $\gamma_{m,\ell}$. This implies that ${\rm T}^2$ is the stabilizer of ${\mathfrak m}_{m,\ell}$ and that $\gamma_{m,\ell}$ is a symmetrical configuration.

\medskip
\noindent Clearly, it suffices to prove the Theorem in the case of the standard symmetrical configurations
\medskip

\noindent {\bf Part III.} We show that the symmetry group $\widehat{{\rm G}}_{m,\ell}$ of $\gamma_{m,\ell}$ is generated by ${\mathcal R}_{m,\ell}$. Since ${\bf B}_{m,\ell}={\mathfrak U}\cdot {\bf Z}^{-1}$ and ${\mathcal R}_{m,\ell}=[{\bf B}_{m,\ell}(\omega_{m,\ell})\cdot {\bf B}_{m,\ell}(0)^{-1}]$ then, using (\ref{Z}) and (\ref{angffun}), it follows that $\gamma_{m,\ell}(s+\omega_{m,\ell})={\mathcal R}_{m,\ell}\cdot \gamma_{m,\ell}(s)$. Hence, ${\mathcal R}_{m,\ell}\in \widehat{{\rm G}}_{m,\ell}$. Let $[{\bf C}]$ be a symmetry of $\gamma_{m,\ell}$ then, for every $s_*\in \R$, there exist an open interval ${\rm I}$ containing $s_*$ and a strictly monotonic differentiable function $f:{\rm I}\to \R$ such that ${\bf C}\cdot \gamma_{m,\ell}(s)=\gamma_{m,\ell}(f(s))$, for every $s\in {\rm I}$. In particular,  
$\gamma_{m,\ell}\circ f$ and $\gamma_{m,\ell}$ are both natural parameterizations. From this we infer that $f(s)=s+c$, for some constant $c$ (see Remark \ref{parameter}). Thus, there exist a sequence $\{c_{n}\}_{n\in {\mathbb N}}$ and a covering $\{\rm {I}_{n}\}_{n\in {\mathbb N}}$ of $\R$ by open intervals such that $\gamma_{m,\ell}|_{{\rm I}_{n}}$ is injective and that ${\bf C}\cdot \gamma_{m,\ell}(s)=\gamma_{m,\ell}(s+c_{n})$, for every $s\in {\rm I}_{n}$. This implies 
that $c_n=c_m=c$, for every $n,m\in {\mathbb N}$. The constant $c$ is a period of $\kappa_{m,\ell}$, 
ie $c={\rm p}\omega_{m,\ell}$, for some ${\rm p}\in \Z$. Therefore, we have ${\bf C}={\mathcal R}_{m,\ell}^{{\rm p}}$.

\noindent {\bf Part IV.} We prove that $|[\gamma_{m,\ell}]|\cap {\mathcal O}_1=|[\gamma_{m,\ell}]|\cap {\mathcal O}_2=\emptyset$. The chain ${\mathcal O}_1$ is contained in the complex line $z^2=0$ of $\mathbb{CP}^2$. Since $z_2(s|m,\ell)\neq 0$, for every $s\in \R$, then 
$|[\gamma_{m,\ell}]|\cap {\mathcal O}_1=\emptyset$. Keeping in mind that ${\bf L}$  interchanges the role of ${\mathcal O}_1$ and ${\mathcal O}_2$ then, $\gamma_{m,\ell}\cap {\mathcal O}_2=\emptyset$ if and only if $\gamma^{\sharp}_{m,\ell}\cap {\mathcal O}_1=\emptyset$. 
By using (\ref{dL}) and (\ref{cp}), we get that the second homogeneous component $z_2^{\sharp}$  of $\gamma^{\sharp}_{m,\ell}$ is given by 
$$2ir_2(m,\ell)r_3(m,\ell)\big(\lambda_2(m,\ell)-\lambda_3(m,\ell)\big)\sqrt{4\kappa_{m,\ell}-3\lambda_1(m,\ell)}\ e^{i\Phi_3(-|m,\ell)}.$$ Then, $z_2^{\sharp}(s)\neq 0$ for every $s\in \R$. This implies that $\gamma^{\sharp}_{m,\ell}\cap {\mathcal O}_1=\emptyset$.

\noindent {\bf Part V.} Let ${\rm lk}_1$ and ${\rm lk}_2$ be the linking numbers of $\gamma_{m,\ell}$ with its symmetry axes ${\mathcal O}_1$ and ${\mathcal O}_2$ respectively. We show that ${\rm lk}_1={\mathtt l}_1$ and ${\rm lk}_2={\mathtt l}_2$. To this end we consider the Legendrian curve $\widetilde{\gamma}_{m,\ell}=p_{h}\circ \gamma_{m,\ell}:\R\to \R^3$ and its Lagrangian projection $\alpha_{m,\ell}:\R\to \R^2$ (ie the projection of $\widetilde{\gamma}_{m,\ell}$ onto the $Oxy$-plane). By construction, $\widetilde{\gamma}_{m,\ell}$ is periodic with least period ${\mathtt n}\omega_{m,\ell}$ and $|[\widetilde{\gamma}_{m,\ell}]|$ doesn't intersect the $Oz$-axis. The plane curve $\alpha_{m,\ell}$ is periodic, ${\mathtt n}\omega_{m,\ell}$ is one of its periods, and 
$|[\alpha _{m,\ell}]|$ does not pass  through the origin. Consequently, the components $x_{m,\ell}$ and $y_{m,\ell}$ of  $\alpha_{m,\ell}$ can be written as  $x_{m,\ell}=\varrho_{m,\ell}\cos(\vartheta_{m,\ell})$, $y_{m,\ell}=\varrho_{m,\ell}\sin(\vartheta_{m,\ell})$, where $\varrho_{m,\ell}:\R\to \R^+$  and $\vartheta_{m,\ell} : \R\to {\rm S}^1\cong \R/2\pi\Z$ are smooth functions. Let $\tau_1$ and $\tau_2$ be the integers defined by ${\mathtt n}=\tau_1{\mathtt k}_1$ and by ${\mathtt n}=\tau_2{\mathtt k}_2$.
Using (\ref{F5.P1.S6.2}) and (\ref{fact}) we get
$$p_{h}\circ ({\mathcal R}_{m,\ell})^{{\mathtt k}_1} = {\rm R}_{Oz}\left(2\pi \frac{{\mathtt k}_1 {\mathtt h}_2}{{\mathtt k}_2}\right)\circ p_{h}.$$
This implies
\[\begin{split}\widetilde{\gamma}_{m,\ell}(s+{\mathtt k}_1\omega_{m,\ell})&=
{\rm R}_{Oz}(2\pi {\mathtt k}_1 {\mathtt h}_2/{\mathtt k}_2)\widetilde{\gamma}_{m,\ell}(s),\\
\alpha_{m,\ell}(s+{\mathtt k}_1\omega_{m,\ell})&= {\rm R}_{O}(2\pi {\mathtt k}_1 {\mathtt h}_2/{\mathtt k}_2)\alpha_{m,\ell}(s),
\end{split}\]
where ${\rm R}_{O}(\theta)$ is the rotation of an angle $\theta$ around the origin of $\R^2$. Thus, $\varrho_{m,\ell}$ is periodic and ${\mathtt k}_1\omega_{m,\ell}$ is one of its period while $\vartheta_{m,\ell}$ is a quasi-periodic function such that
\begin{equation}\label{qpp}\vartheta_{m,\ell}(s+{\mathtt k}_1\omega_{m,\ell})=\vartheta_{m,\ell}(s)+2\pi \frac{{\mathtt k}_1 {\mathtt h}_2}{{\mathtt k}_2}.\end{equation}
Since ${\rm lk}_1=\rm{lk}(\widetilde{\gamma}_{m,\ell},Oz)$ and expressing $\rm{lk}(\widetilde{\gamma}_{m,\ell},Oz)$ via the Gaussian linking integral \cite{RN}, we get
$${\rm lk}_1=\frac{1}{4\pi}\int_0^{{\mathtt n} \omega_{m,\ell}}\left(\int_{-\infty}^{+\infty}\frac{(\widetilde{\gamma}_{m,\ell}(s)-t\vec{{\mathtt k}})\cdot (\widetilde{\gamma}_{m,\ell}'(s)\times \vec{{\mathtt k}})}{\|\widetilde{\gamma}_{m,\ell}(s)-t\vec{{\mathtt k}}\|^3}dt
\right)ds.
$$
On  the other hand, from (\ref{cp}) we have
\[\begin{split}
& \int_{-\infty}^{+\infty}\frac{(\widetilde{\gamma}_{m,\ell}(s)-t\vec{{\mathtt k}})\cdot (\widetilde{\gamma}_{m,\ell}'(s)\times \vec{{\mathtt k}})}{\|\widetilde{\gamma}_{m,\ell}(s)-t\vec{{\mathtt k}}\|^3}dt=\\
=&\int_{-\infty}^{+\infty}\frac{x_{m,\ell}(s)y_{m,\ell}'(s)-x_{m,\ell}'(s)y_{m,\ell}(s)}{(x_{m,\ell}(s)^2+y_{m,\ell}(s)^2+z_{m,\ell}(s)^2+t^2-2tz_{m,\ell}(s))^{3/2}}dt=2\vartheta_{m,\ell}'(s)
\end{split}\]
Using (\ref{qpp}), we obtain
\[
{\rm lk}_1=\frac{1}{2\pi}\int_0^{{\mathtt n} \omega_{m,\ell}}\vartheta_{m,\ell}'(s)ds=
\frac{1}{2\pi}\int_0^{\tau_1 {\mathtt k}_1 \omega_{m,\ell}}\vartheta_{m,\ell}'(s)ds
=\frac{\tau_1 {\mathtt k}_1 {\mathtt h}_2}{{\mathtt k}_2} = {\mathtt n}\frac{ {\mathtt h}_2}{{\mathtt k}_2} = {\mathtt l}_1.
\]
To prove that ${\rm lk}_2={\mathtt l}_2$ we consider the dual configuration $\gamma^{\sharp}_{m,\ell}={\bf L}\cdot \gamma_{m,\ell}$. Since
${\rm lk}(\gamma_{m,\ell},{\mathcal O}_2)= {\rm lk}({\bf L}^{-1}\cdot\gamma^{\sharp}_{m,\ell},{\bf L}^{-1}\cdot{\mathcal O}_1)={\rm lk}(\gamma^{\sharp}_{m,\ell},{\mathcal O}_1),
$
it suffices to prove that ${\rm lk}(\gamma^{\sharp}_{m,\ell},{\mathcal O}_1)={\mathtt l}_2$. The monodromy of $\gamma^{\sharp}_{m,\ell}$ is given by
$${\mathcal R}^{\sharp}_{m,\ell}=[{\bf L}\cdot {\mathcal R}_{m,\ell}\cdot {\bf L}^{-1}]=[{\mathfrak U}(e^{2i\pi q_3}{\bf E}^1_1+e^{-2\pi i (q_3+q_2)}{\bf E}^2-2+e^{2\pi i q_2}{\bf E}^3_3)\cdot {\mathfrak U}^{-1}].$$
Then
\begin{equation}\label{pk}p_{h}\circ ({\mathcal R}^{\sharp}_{m,\ell})^{{\mathtt k}_2} = {\rm R}_{Oz}(-2\pi 
{\mathtt k}_2 {\mathtt h}_2/{\mathtt k}_1)\circ p_{h}.\end{equation}
Let $\widetilde{\gamma}^{\sharp}_{m,\ell}=p_{h}\circ \gamma^{\sharp}_{m,\ell}$ and  $\alpha^{\sharp}_{m,\ell}$ be the Lagrangian projection of $\widetilde{\gamma}^{\sharp}_{m,\ell}$. Denote by $\varrho^{\sharp}$ and by $\vartheta^{\sharp}$ the radial and the angular functions of $\alpha^{\sharp}_{m,\ell}$. From (\ref{pk}) we get
\[\begin{cases}\widetilde{\gamma}^{\sharp}_{m,\ell}(s+{\mathtt k}_2\omega_{m,\ell})=
{\rm R}_{Oz}(-2\pi {\mathtt k}_2 {\mathtt h}_1/{\mathtt k}_1)\widetilde{\gamma}^{\sharp}_{m,\ell}(s),\\
\alpha^{\sharp}_{m,\ell}(s+{\mathtt k}_2\omega_{m,\ell})= {\rm R}_{O}(-2\pi {\mathtt k}_2 {\mathtt h}_1/{\mathtt k}_1)\alpha^{\sharp}_{m,\ell}(s).
\end{cases}\]
The radial function $\varrho^{\sharp}$ is periodic and ${\mathtt k}_2\omega_{m,\ell}$ is one of its periods while 
\begin{equation}\label{qpp2}\vartheta^{\sharp}(s+{\mathtt k}_2\omega_{m,\ell})=\vartheta^{\sharp}(s)-2\pi {\mathtt k}_2 {\mathtt h}_1/{\mathtt k}_1.\end{equation}
Then,
\[\begin{split}{\rm lk}_2=& {\rm lk}(\gamma^{\sharp}_{m,\ell},{\mathcal O}_1)=
\frac{1}{4\pi}\int_0^{{\mathtt n} \omega_{m,\ell}}\left(\int_{-\infty}^{+\infty}\frac{(\widetilde{\gamma}^{\sharp}_{m,\ell}(s)-t\vec{{\mathtt k}})\cdot (\widetilde{\gamma}^{\sharp}_{m,\ell}{}^{'}(s)\times \vec{{\mathtt k}})}{\|\widetilde{\gamma}^{\sharp}_{m,\ell}(s)-t\vec{{\mathtt k}}\|^3}dt
\right)ds=\\
=&\frac{1}{2\pi}\int_0^{{\mathtt n} \omega_{m,\ell}}\vartheta^{\sharp}{}'(s)ds=
\frac{1}{2\pi}\int_0^{\tau_2 {\mathtt k}_2 \omega_{m,\ell}}\vartheta^{\sharp}{}'(s)ds
=-\frac{\tau_2 {\mathtt k}_2 {\mathtt h}_1}{{\mathtt k}_1} = -{\mathtt n}\frac{ {\mathtt k}_1}{{\mathtt h}_1} = {\mathtt l}_2.
\end{split}\]
}\end{proof}

\subsection{Examples}\label{S6.3} We use the notation $|{\mathtt n},{\mathtt l}_1,{\mathtt l}_2 >$ for the standard symmetrical configuration with wave number ${\mathtt n}$ and linking numbers ${\mathtt l}_1$ and ${\mathtt l}_2$.  The Maslov index of $|{\mathtt n},{\mathtt l}_1,{\mathtt l}_2 >$ is equal to ${\mathtt l}_2+{\mathtt l}_1$. The cardinality $\varrho({\mathtt n})$ of the set ${\mathcal C}_{\mathtt {n}}$ of the equivalence classes of closed string with symmetry group of order ${\mathtt n}$ exhibits a quadratic growth (see Figure \ref{FIG5}). 
\begin{figure}[h]
\begin{center}
\includegraphics[height=4.1cm,width=6.2cm]{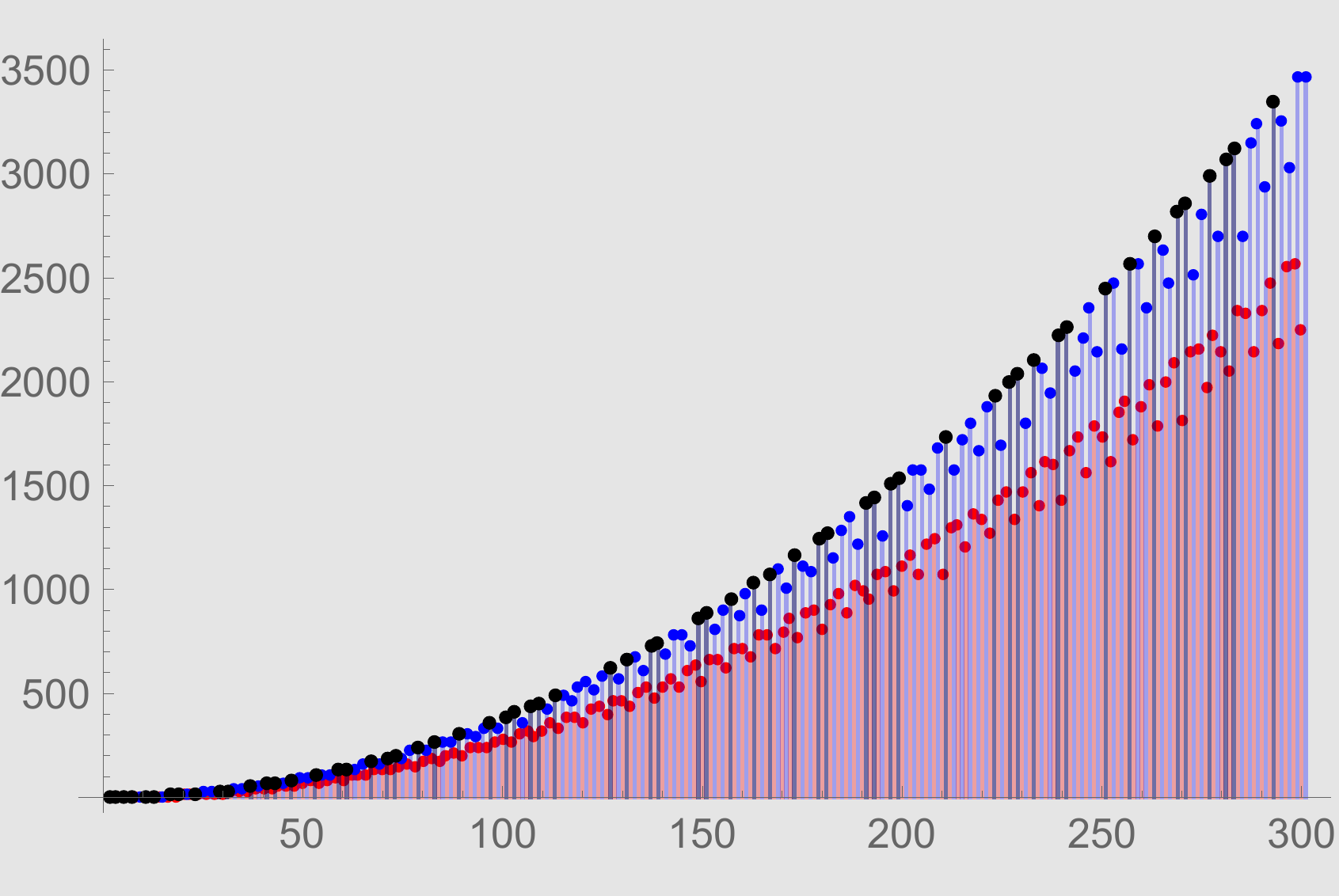}
\includegraphics[height=4.1cm,width=6.2cm]{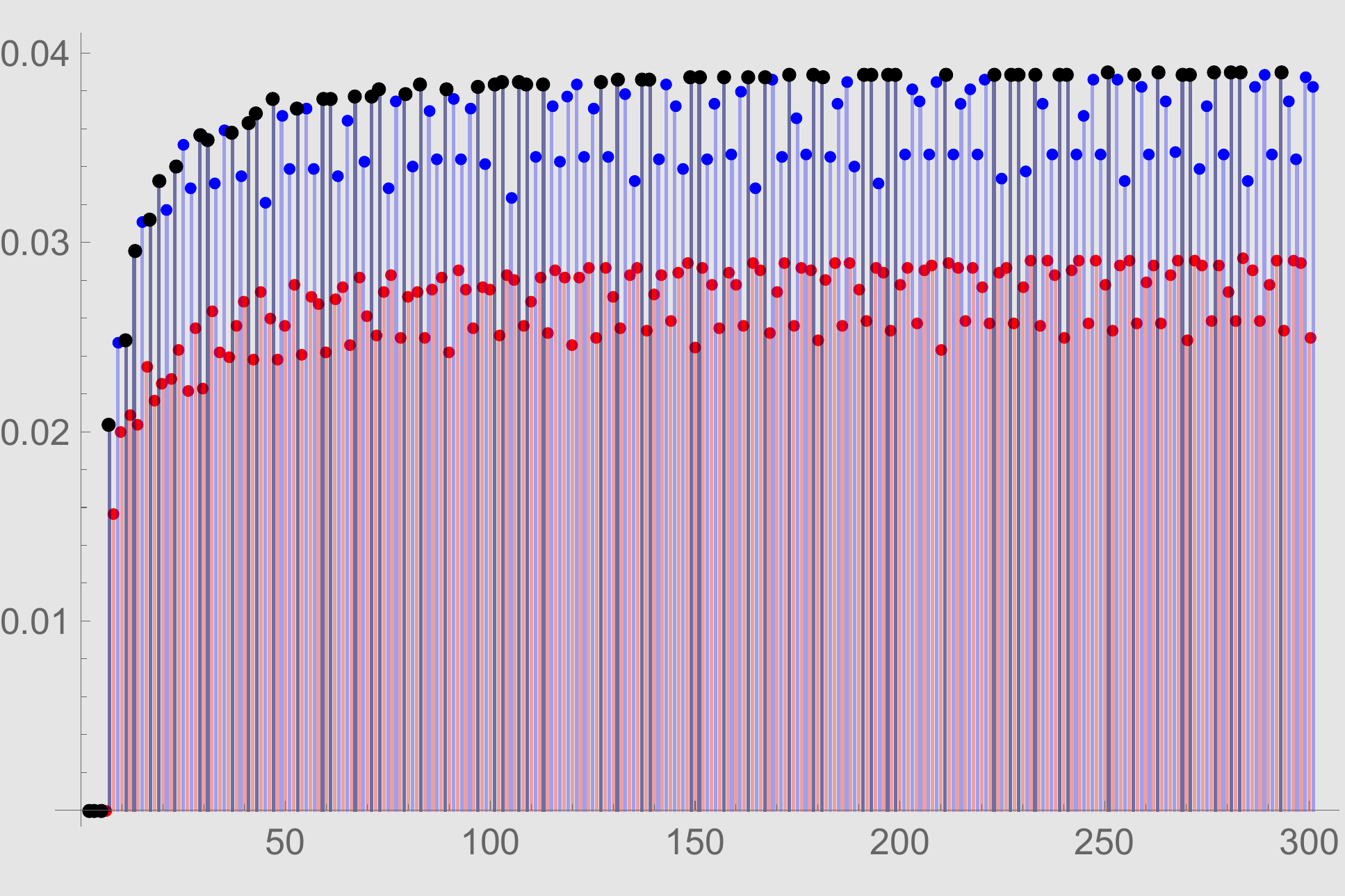}
\caption{\small{$\varrho(n)$ (left) and $\varrho(n)/n^2$  (right), $7\le n\le 300$; red = $n$ even, blue = $n$ is odd but not prime, black = $n$ is prime.}}\label{FIG5}
\end{center}
\end{figure}
\noindent There are no closed strings with wave number ${\mathtt n}<7$. In Table  $1$ we list the standard symmetrical configurations with wave numbers $7,8,9$  and their basic invariants: characteristic numbers ${\mathtt h}_1,{\mathtt k}_1,{\mathtt h}_2,{\mathtt k}_2$, characters $(m,\ell)$, wavelength $\omega_{m,\ell}$, total strain ${\mathfrak S}$, Maslov index $\mathfrak{r}$, Bennequin-Thurston invariant $\mathfrak{tb}$ and the knot type ${\rm kt}$. Figures \ref{FIG6A} and \ref{FIG6B} reproduce the corresponding standard symmetrical configurations. The characters $(m,\ell)$ are computed with numerical methods. The invariants, $\mathfrak{r}$ and $\mathfrak{tb}$ are found via numerical integration of the total curvature of the Lagrangian projection of $p_{h}\circ \gamma$ and of the Gaussian linking integral of $p_{h}\circ \gamma$ with $p_{h}\circ \gamma+\varepsilon \vec{k}$, 
$0<\varepsilon \ll 1$ \cite{GE,RN}.

\begin{center}
\begin{tabular}{|c|c|c|c|c|c|c|c|c|}
\hline 
\rule[-1ex]{0pt}{2.5ex} string & $(\frac{{\mathtt h}_1}{{\mathtt k}_1},\frac{{\mathtt h}_2}{{\mathtt k}_2})$&$(m,\ell)$&
$\omega_{m,\ell}$& ${\mathfrak S}$&$\mathfrak{r}$&$\mathfrak{tb}$&kt \\ 
\hline
\rule[-1ex]{0pt}{2.5ex} $|7,1,-5>$&$(\frac{5}{7},\frac{1}{7})$&$(0.894052,2.78109)$&$1.83449$&$12.8414$&$-4$&$-5$&trivial\\
\hline
\rule[-1ex]{0pt}{2.5ex} $|8,1,-6>$&$(\frac{3}{4},\frac{1}{8})$&$(0.762709,2.13126)$&$2.04567$&$16.3654$&$-5$&$-6$&trivial\\
\hline
\rule[-1ex]{0pt}{2.5ex} $|9,1,-7>$&$(\frac{7}{9},\frac{1}{9})$&$(0.616723,1.82908)$&$2.15197$&$19.3677$&$-6$&$-7$&trivial\\
\hline
\rule[-1ex]{0pt}{2.5ex} $|9,2,-6>$&$(\frac{2}{3},\frac{2}{9})$&$(0.906698,3.05894)$&$1.70697$&$15.3627$&$-4$&$-3$&trefoil\\
\hline
\end{tabular}\end{center}
\begin{center}
\begin{tabular}{c}
\footnotesize{TABLE 1}\\
\end{tabular}\end{center}
The experimental evidence suggest that a standard symmetrical configuration with ${\mathtt l}_1=1$ is a trivial Legendrian knot with $\mathfrak{tb}={\mathtt l}_2$ and $\mathfrak{r}=\mathfrak{tb}+1$. Thus, a string with ${\mathtt l}_1=1$ can be obtained from a cycle, via $|{\mathtt l}_2|$ negative stabilizations \cite{EF,ET1}.

\begin{figure}[h]
\begin{center}
\includegraphics[height=6cm,width=6cm]{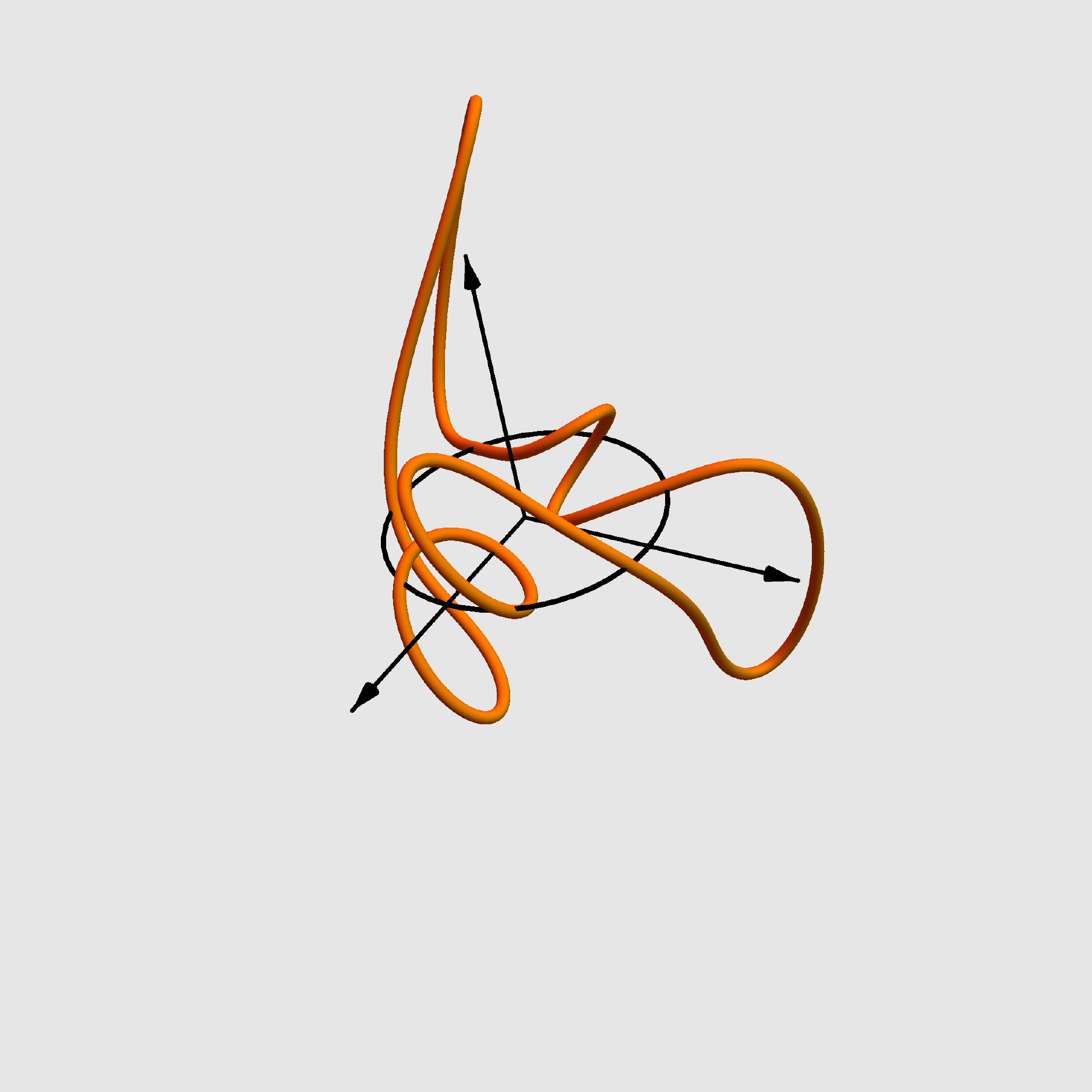}
\includegraphics[height=6cm,width=6cm]{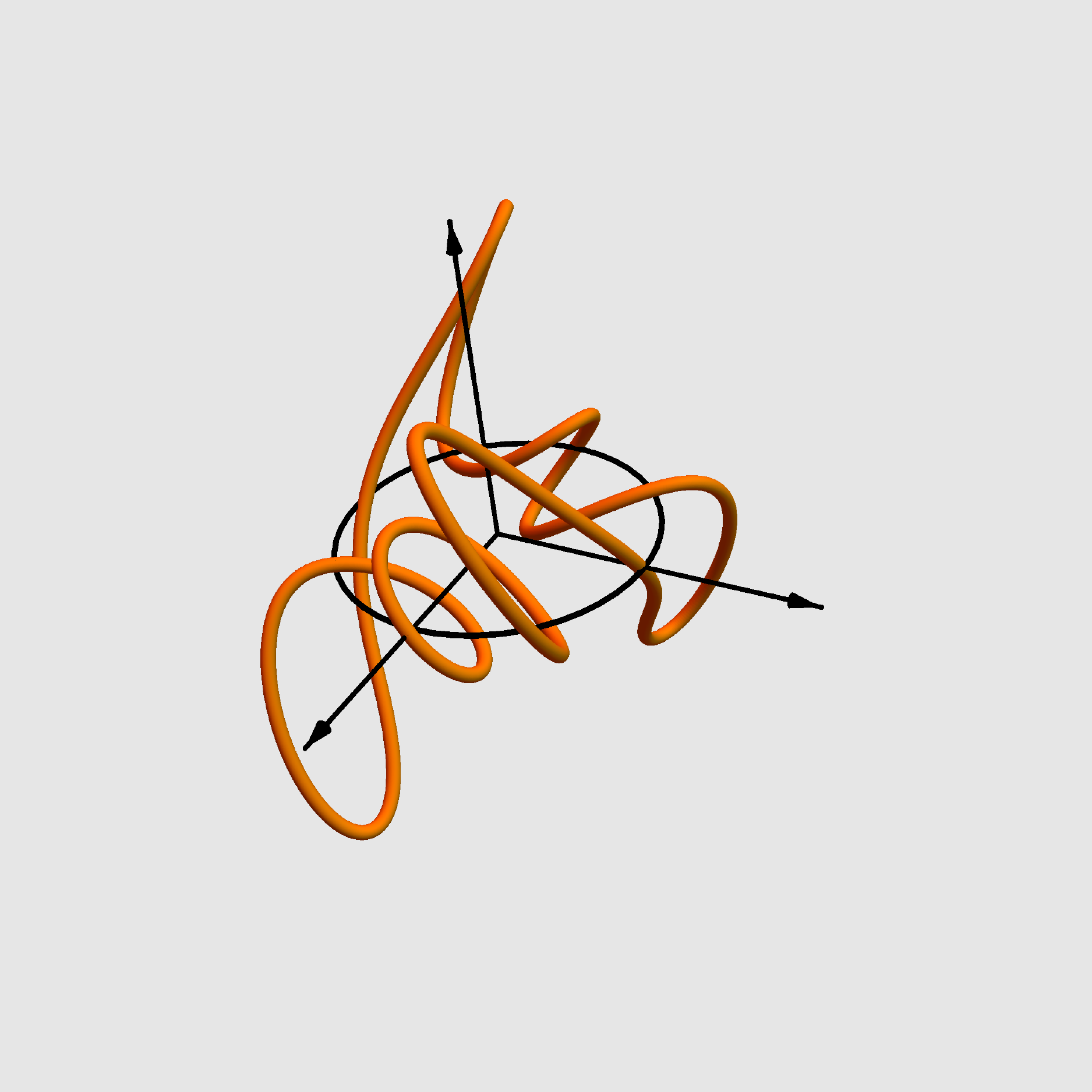}
\caption{\small{$|7,1,-5>$ (left) and $|8,1,-6>$ (right).}}\label{FIG6A}
\end{center}
\end{figure}

\noindent In Table $2$ we list the basic invariants of two standard symmetrical configurations with less obvious knot types. Figure \ref{FIG7} depicts these two strings.
\begin{center}
\begin{tabular}{|c|c|c|c|c|c|c|c|c|}
\hline 
\rule[-1ex]{0pt}{2.5ex} string & $(\frac{{\mathtt h}_1}{{\mathtt k}_1},\frac{{\mathtt h}_2}{{\mathtt k}_2})$&$(m,\ell)$&
$\omega_{m,\ell}$& ${\mathfrak S}$&$\mathfrak{r}$&$\mathfrak{tb}$&kt \\ 
\hline
\rule[-1ex]{0pt}{2.5ex} $|13,3,-9>$&$(\frac{9}{13},\frac{3}{13})$&$(0.70944,2.14341)$&$1.94971$&$25.3462$&$-6$&$-1$&$8_{19}$\\
\hline
\rule[-1ex]{0pt}{2.5ex} $|21,5,-15>$&$(\frac{5}{7},\frac{5}{21})$&$(0.36972,1.71141)$&$2.05338$&$43.1209$&$-10$&$9$&${\rm T}_{(7,5)}$\\
\hline
\end{tabular}\end{center}
\begin{center}
\begin{tabular}{c}
\footnotesize{TABLE 2}\\
\end{tabular}\end{center}

\begin{figure}[h]
\begin{center}
\includegraphics[height=6cm,width=6cm]{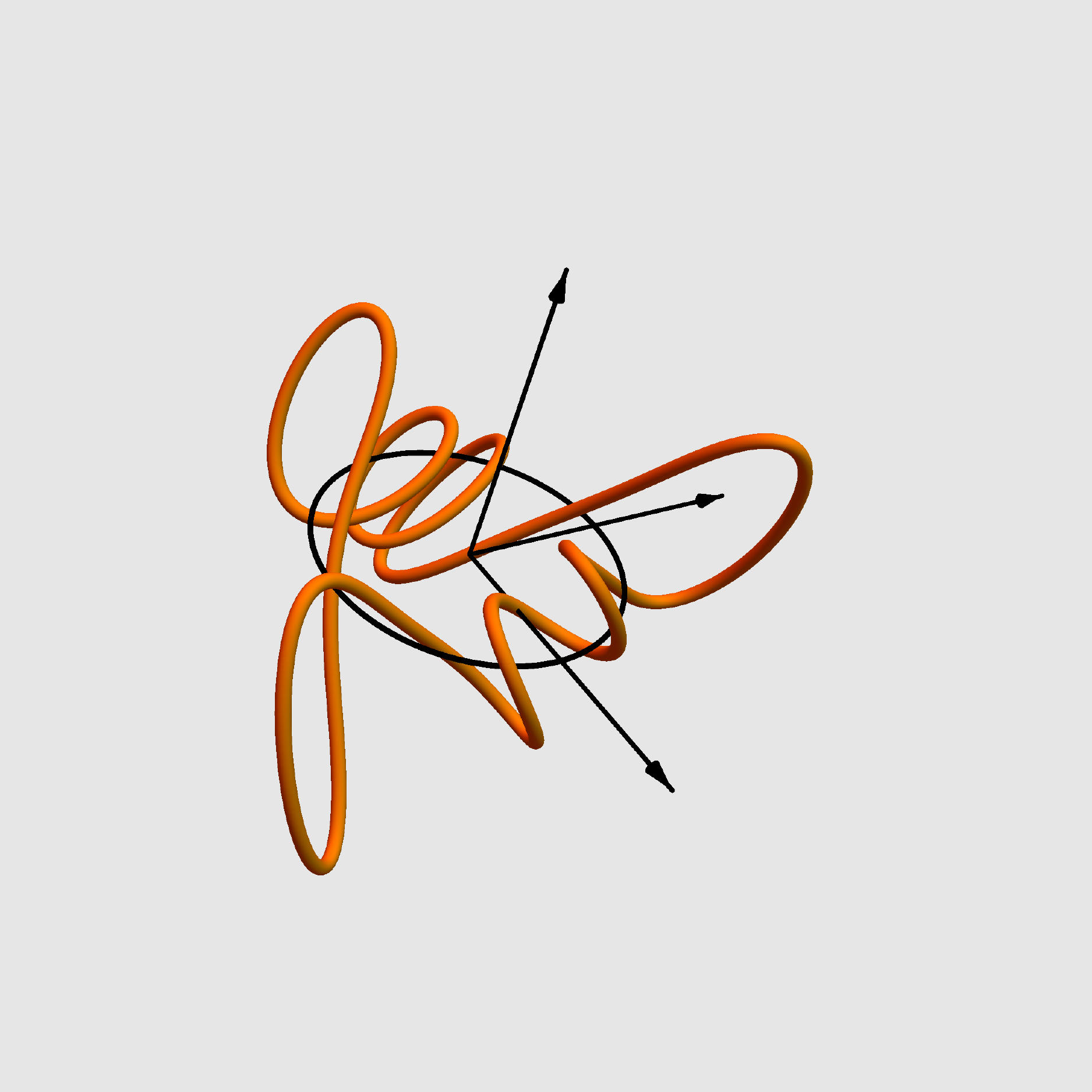}
\includegraphics[height=6cm,width=6cm]{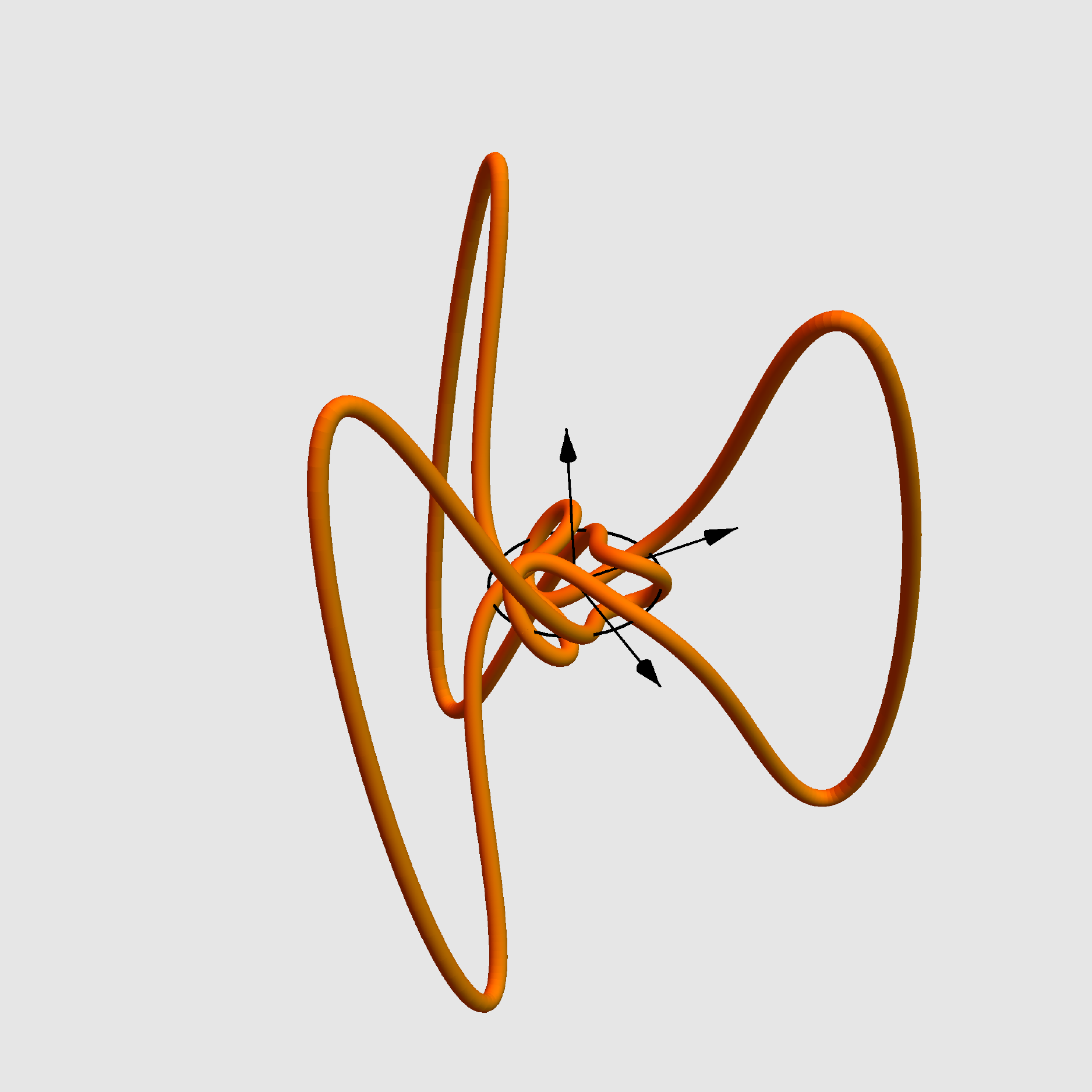}
\caption{\small{$|9,1,-7>$ (left) and $|9,2,-6>$ (right).}}\label{FIG6B}
\end{center}
\end{figure}

\begin{figure}[h]
\begin{center}
\includegraphics[height=6cm,width=6cm]{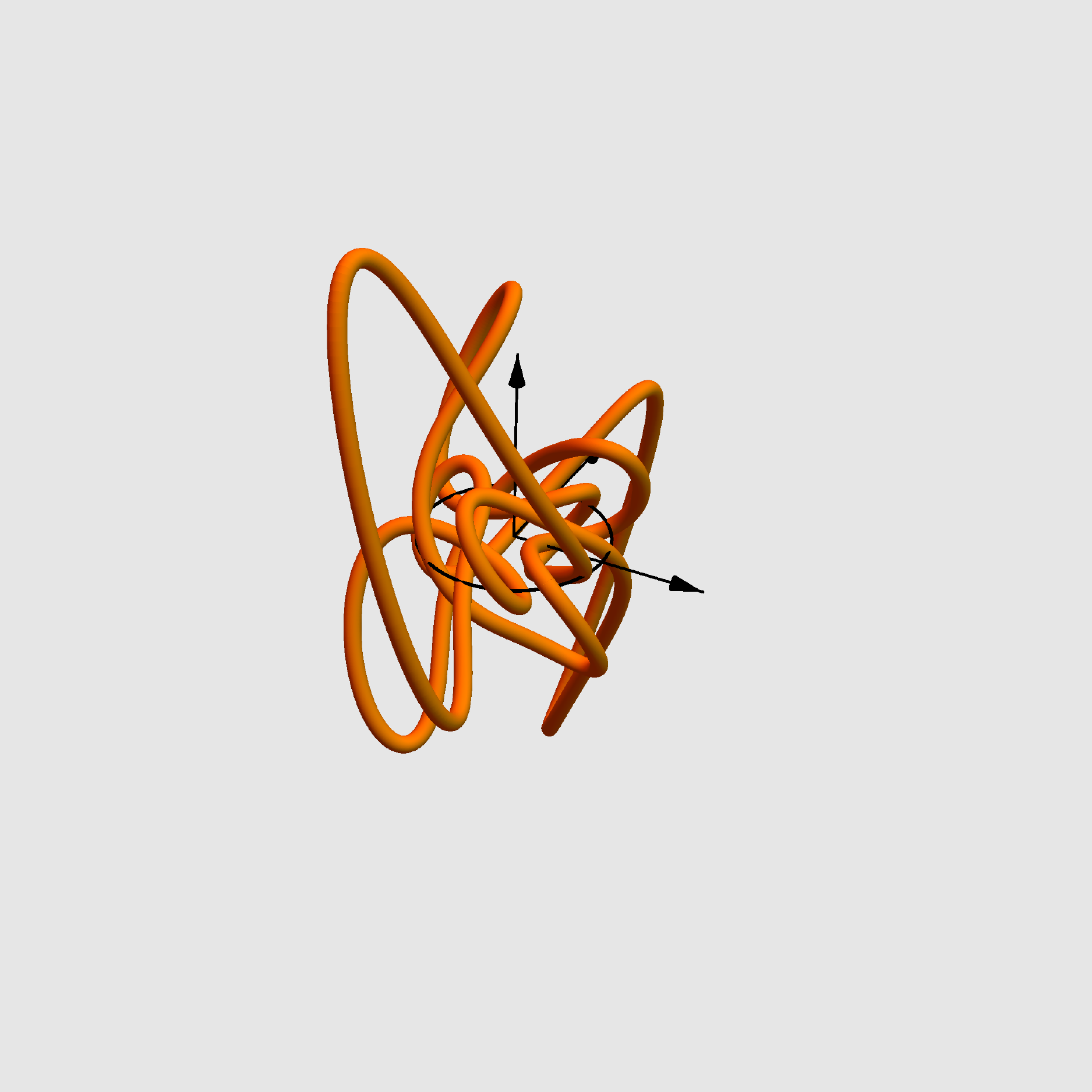}
\includegraphics[height=6cm,width=6cm]{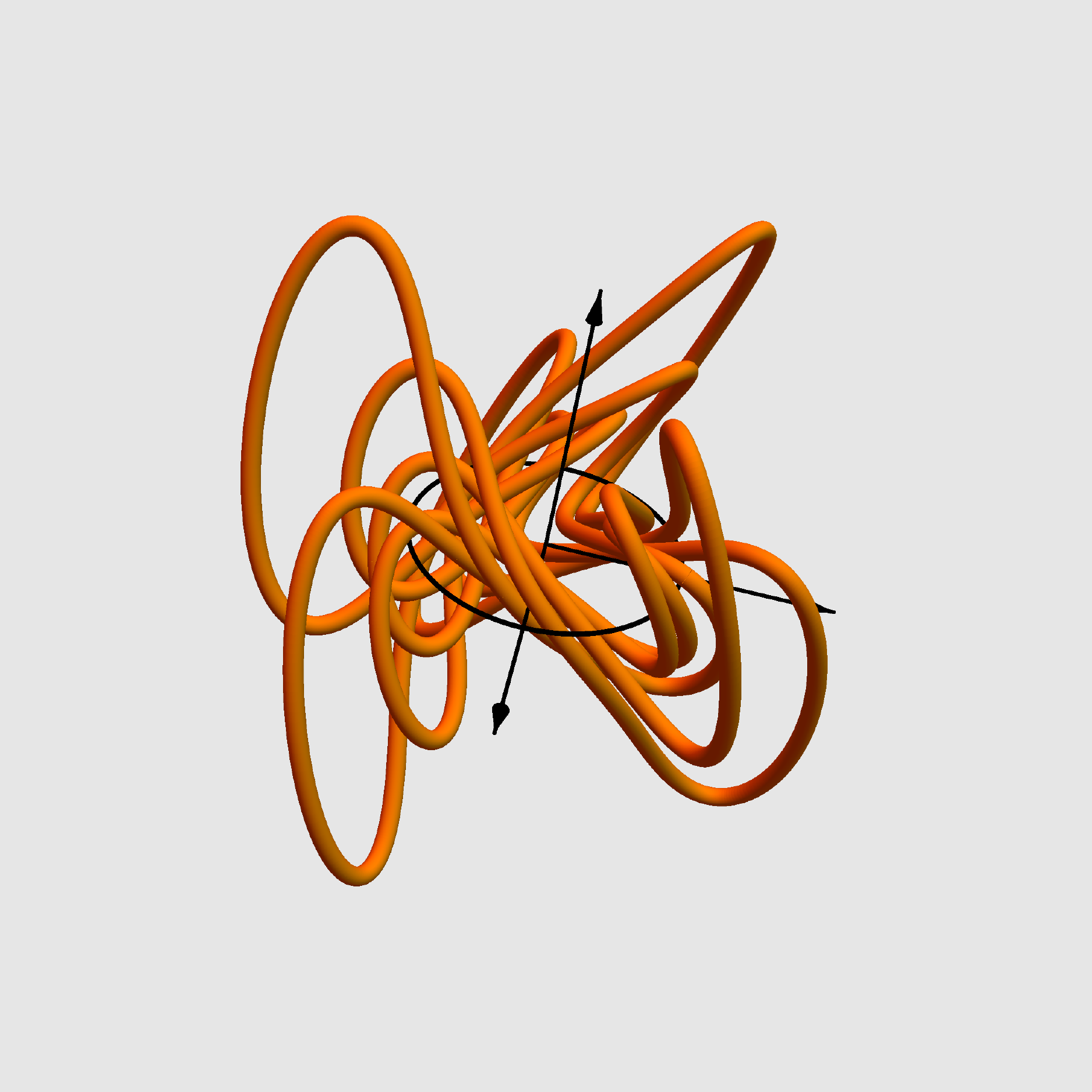}
\caption{\small{$|13,3,-9>$ (left) and $|21,5,-15>$ (right).}}\label{FIG7}
\end{center}
\end{figure}

\noindent The shape of the strings becomes more complicated when ${\mathtt n}$, ${\mathtt l}_1$ and ${\mathtt l}_2$ increase. Figure \ref{FIG8} reproduces the standard symmetrical configurations $|70,2,-42>$ and $|112,21,-80>$. As one can see from the pictures, the stands of the string may approach each other and it is not always evident if the string is simple or not.

\begin{figure}[h]
\begin{center}
\includegraphics[height=6cm,width=6cm]{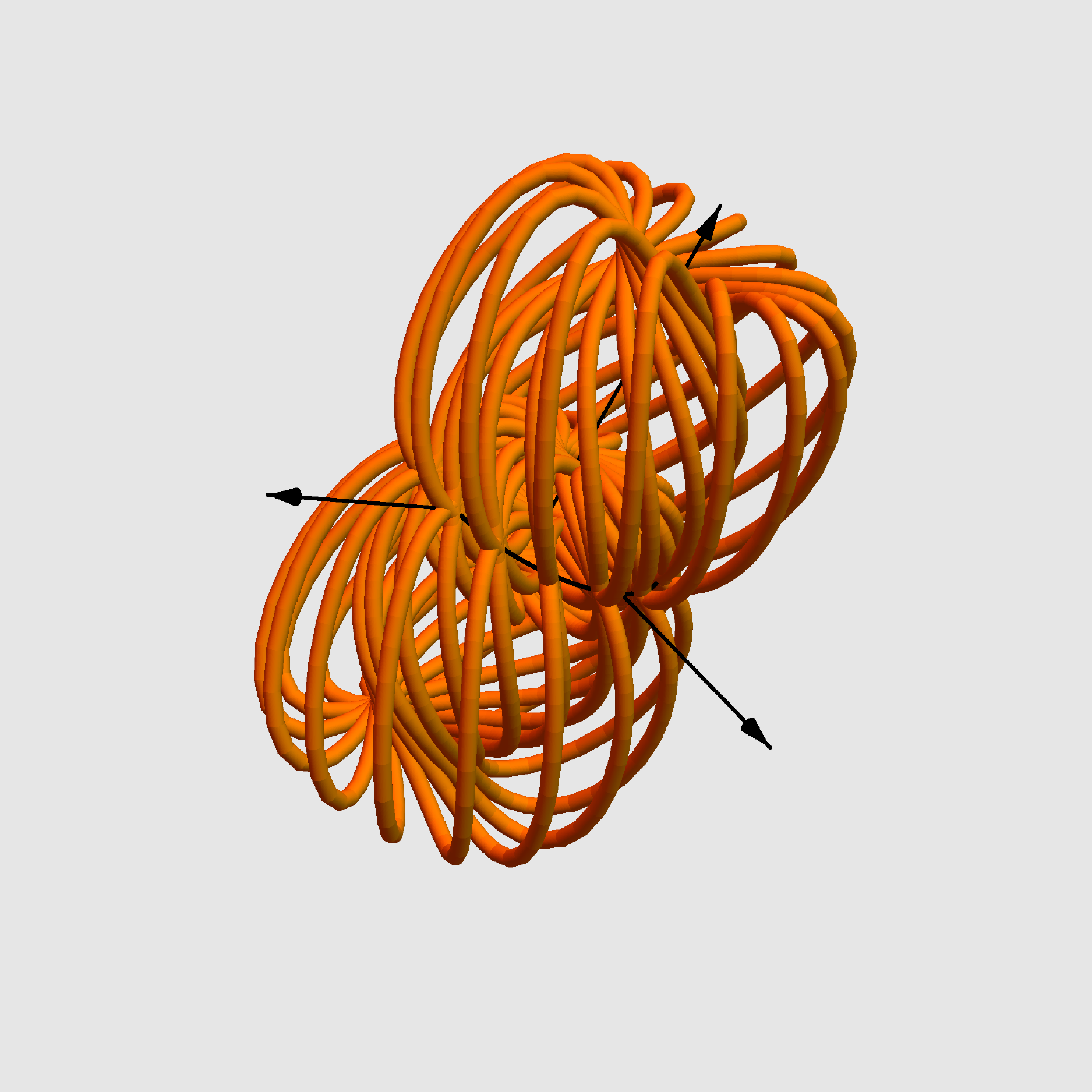}
\includegraphics[height=6cm,width=6cm]{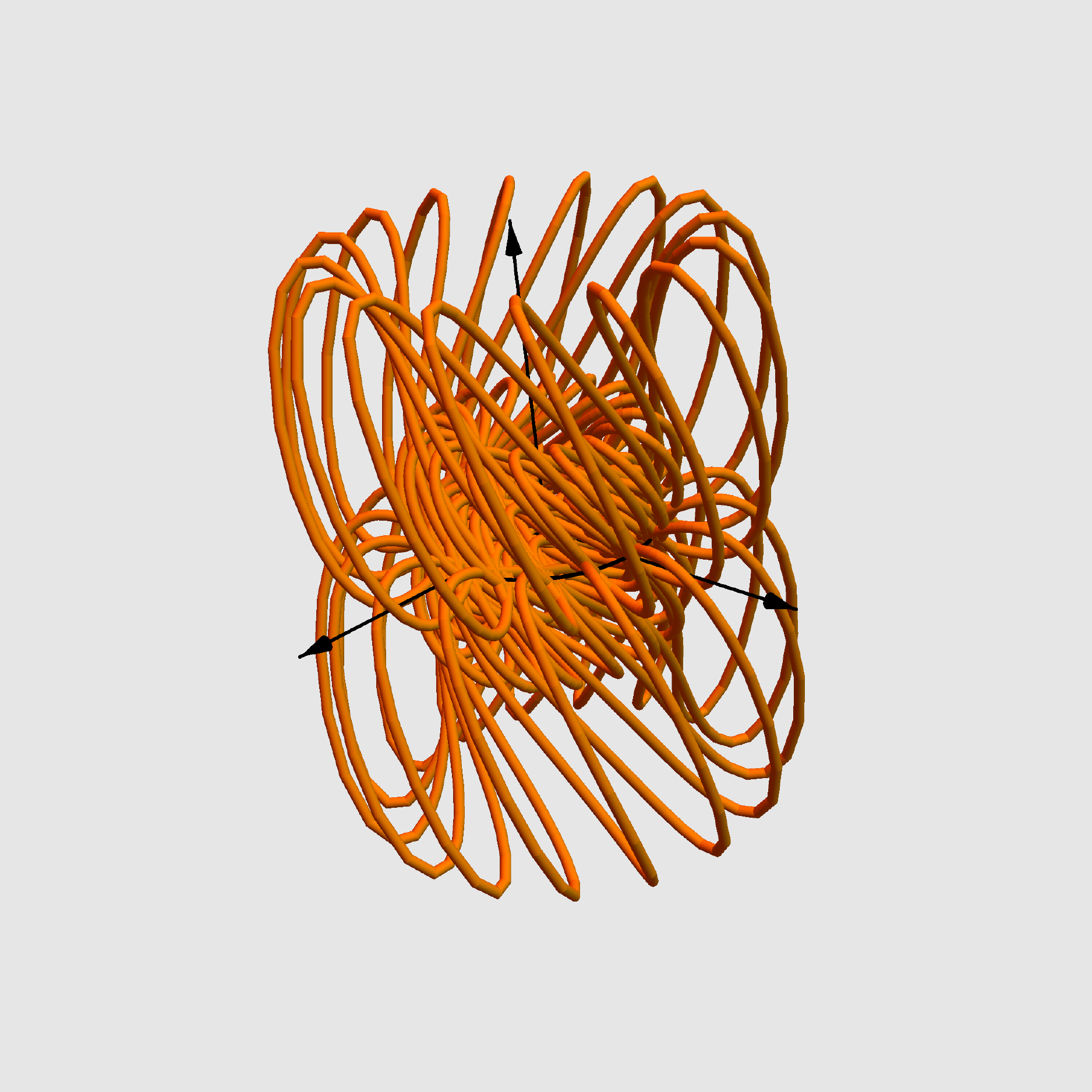}
\caption{\small{Overtwisted closed strings: $|70,2,-42>$ (left) and $|112,21,-80>$  (right).}}\label{FIG8}
\end{center}
\end{figure}

\bibliographystyle{amsalpha}

\end{document}